\documentclass[11pt]{article}
 
\usepackage{amsthm}
\usepackage{amsmath,amssymb,  amscd,eucal}
 \usepackage{latexsym}
\usepackage{graphicx}
\usepackage[usenames]{color}
\usepackage{times,tikz}
\usepackage{paralist}
\usetikzlibrary{decorations.markings}
\usepackage{multicol}
\usepackage{hyperref}

\newtheorem{thm}[equation]{Theorem}
\newtheorem{cor}[equation]{Corollary}
\newtheorem{lemma}[equation]{Lemma}
\newtheorem{prop}[equation]{Proposition}

\theoremstyle{definition}
\newtheorem{definition}[equation]{Definition}

\newtheorem{remark}[equation]{Remark}
\newtheorem{example}[equation]{Example}
\newtheorem{examples}[equation]{Examples}

\numberwithin{equation}{section}

\def\End{ \mathsf{End}}
\def\Res{\mathsf{Res}}

\def\diam{\mathsf {di}}  
\def\id{\mathsf{id}}
\def\ima{\mathsf{im}}
\def\ot{\otimes}

\renewcommand{\deg}{\mathsf{deg}}
\renewcommand{\b}{\mathsf{b}}
\newcommand{\ac}{\ell}
\newcommand{\kl}{\mathsf{K}}
\newcommand{\e}{\mathsf{e}}
\newcommand{\f}{\mathsf{f}}
\newcommand{\w}{\mathsf{w}}
\renewcommand{\b}{\mathsf{b}}
\newcommand{\br}{\mathsf{br}}

\newcommand{\tr}{\mathsf{tr}}
\newcommand{\Z}{\mathsf{Z}}
\newcommand{\G}{\mathsf{G}}
\renewcommand{\S}{\mathsf{S}}
\newcommand{\SU}{\mathsf{SU}}

\newcommand{\TT}{\mathbf{T}}

\newcommand{\OO}{\mathbf{O}}

\newcommand{\II}{\mathbf{I}}
\newcommand{\D}{\mathbf{D}}
\newcommand{\DD}{\mathbf{D}}

\newcommand{\ZZ}{\mathsf{Z}}  
\newcommand{\CC}{\mathbb{C}}  
\newcommand{\C}{\mathbf{C}}  
\newcommand{\GG}{\mathsf{G}}

\newcommand{\NN}{\mathbb{N}}

\newcommand{\msk}{\medskip}
\newcommand{\rf}{\mathsf{r}}
\newcommand{\sm}{\mathsf{s}}
\newcommand{\ssk}{\smallskip}
\newcommand{\p}{\mathsf{p}}

\newcommand{\TL}{\mathsf{TL}}  
\newcommand{\Tr}{\mathsf{tr}}

\newcommand{\vf}{\mathsf{v}}
  
\newcommand{\VV}{\mathsf{V}}
\newcommand{\W}{\mathsf{W}}

\newcommand{\modd}{\ \mathsf{mod}\,}
\newcommand{\half}{\frac{1}{2}} 
\newcommand{\dimm}{\mathsf{dim}\,}
\newcommand{\spann}{\mathsf{span}}
\renewcommand{\v}{\mathsf{v}}
\newcommand{\overbar}[1]{\mkern 1.5mu\overline{\mkern-1.5mu#1\mkern-1.5mu}\mkern 1.5mu}
\usepackage[letterpaper,margin=1.00in]{geometry}

\begin{document} 
\title {McKay Centralizer Algebras} 
 
 \author{Jeffrey M.~Barnes, Georgia Benkart, and Tom Halverson}
   \date{}
 
 \maketitle
 
 \begin{abstract}
For a finite subgroup $\G$ of the special unitary group $\SU_2$, we study the 
centralizer algebra $\Z_k(\G) = \End_\G(\VV^{\otimes k})$ of $\G$ acting on the $k$-fold tensor product of
its defining representation $\VV = \mathbb C^2$.  These subgroups are in bijection with  the simply-laced affine Dynkin diagrams.  The McKay correspondence relates the representation theory of these groups to the associated Dynkin diagram, and we use this connection to show that the structure
and representation theory  of $\Z_k(\G)$ as a semisimple algebra is controlled by the combinatorics 
of the corresponding  Dynkin diagram.  \end{abstract}

 \textbf{MSC Numbers (2010)}: 05E10 (primary), 14E16, 20C15 (secondary)  \hfill \newline

\section*{Introduction}

In 1980, John McKay \cite{McKay} made the remarkable discovery that there is a natural one-to-one correspondence between the finite subgroups of the special unitary group  $\SU_2$  and the  simply-laced affine Dynkin diagrams,   
which can be described as follows.  Let $\VV= \mathbb C^2$ be the defining representation of $\SU_2$ and let $\G$  be a finite subgroup of $\SU_2$  with irreducible modules $\G^\lambda, \lambda \in \Lambda(\G)$.  The representation graph $\mathcal R_{\VV}(\G)$ (also known as the McKay graph or McKay quiver) has vertices indexed by the $\lambda \in \Lambda(\G)$ and $a_{\lambda, \mu}$ edges from  $\lambda$ to $\mu$ if $\G^\mu$ occurs in $\G^\lambda \otimes \VV$ with multiplicity $a_{\lambda, \mu}$. 
Almost a century earlier,  Felix Klein had determined that a finite subgroup of $\SU_2$ must be one of the following:  (a) a cyclic group $\C_n$ of order $n$, (b) a binary dihedral group $\DD_n$ of order $4n$, or (c) one of the 3 exceptional groups: the binary tetrahedral group $\TT$ of order 24,  the binary octahedral group $\OO$ of order 48, or the binary icosahedral group $\II$ of order 120.   McKay's observation was that the representation graph of  $\C_n, \DD_n, \TT, \OO, \II$ corresponds exactly to the Dynkin diagram $\mathsf{\hat{A}}_{n-1}$, $\mathsf{\hat{D}}_{n+2}$, $\mathsf{\hat{E}}_6$, $\mathsf{\hat{E}}_7$, $\mathsf{\hat{E}}_8$, respectively (see Section \ref{subsec:Ddiag}  below).

In this paper,  we examine the McKay correspondence from the point of view of Schur-Weyl duality.  Since the McKay graph provides a way  to encode the rules for tensoring by $\VV$, it is natural to consider the  $k$-fold tensor product  module $\VV^{\otimes k}$ and to study the centralizer algebra $\Z_k(\G) = \End_\G( \VV^{\otimes k})$  of endomorphisms that commute with the action of $\G$ on  $\VV^{\otimes k}$. The algebra $\Z_k(\G)$  provides essential information about  the structure
of $\VV^{\ot k}$ as a $\G$-module,  as the projection maps from $\VV^{\ot k}$ onto its irreducible $\G$-summands are idempotents in $\Z_k(\G)$, and the multiplicity of $\G^\lambda$ in $\VV^{\ot k}$ is
the dimension of the $\Z_k(\G)$-irreducible module corresponding to $\lambda$.  The problem of studying centralizer algebras of tensor powers of  the natural $\G$-module
 $\VV = \CC^2$ for $\G \subseteq \SU_2$ via the McKay correspondence is discussed in \cite[4.7.d]{GHJ} in the general framework of derived towers, subfactors, and von Neumann algebras, an approach not adopted here.   Our aim is to develop the  
structure and representation theory of the algebras $\Z_k(\G)$ and to show how they are 
controlled by the combinatorics of the representation graph 
$\mathcal R_\VV(\G)$ (the Dynkin diagram) via double-centralizer theory (Schur-Weyl duality). 

The itemized results below, which follow from known facts  on Schur-Weyl duality (see for example,
\cite[Secs. 3B and 68]{CR} and \cite[Secs.~3 and 5]{HR2}),  can be used for any finite group $\G$ and any finite-dimensional $\G$-module $\VV$.  In this work, 
we will apply them to the case $\G$ is a finite subgroup of $\SU_2$ and $\VV=\CC^2$.   More details will be given in
Section 1.4. \msk
 
\noindent \textbf {Consequences of Schur-Weyl duality} \ssk

\begin{itemize} 
\item[$\bullet$]   The irreducible $\Z_k(\G)$-modules  are  in bijection with the vertices  of $\mathcal R_\VV(\G)$ which correspond  to the irreducible modules  $\G^\lambda$  that occur in 
$\VV^{\ot k}$.
\item[$\bullet$]    The dimensions of these $\Z_k(\G)$-modules enumerate walks on $\mathcal R_\VV(\G)$ of $k$ steps.  
\item[$\bullet$]    When $\VV$ is a self-dual $\G$-module, the dimension of $\Z_k(\G)$  equals the number of walks of  $2k$ steps  on $\mathcal R_\VV(\G)$ starting and ending at node 0, which corresponds to the trivial $\G$-module. 
\item[$\bullet$]  The Bratteli diagram $\mathcal B_{\VV}(\GG)$  (see Section \ref{subsec:Bratteli}) is constructed recursively from $\mathcal R_{\VV}(\GG)$,  and paths on  
$\mathcal B_\VV(\G)$ correspond to walks on $\mathcal R_\VV(\GG)$. 
\item[$\bullet$]  When $k$ is less than or equal to the diameter of the 
graph $\mathcal R_\VV(\G)$,  the algebra $\Z_k(\G)$ has generators labeled by nodes of $\mathcal R_\VV(\G)$, and the relations they satisfy are determined  by the edge structure of $\mathcal R_\VV(\G)$.
\end{itemize}\msk

When $\G$ is a  subgroup of $\SU_2$,  the centralizer algebras satisfy the reverse inclusion $\Z_k(\SU_2)  \subseteq \Z_k(\G)$. It is well known that $\Z_k(\SU_2)$ is isomorphic to the Temperley-Lieb algebra $\TL_k(2)$.  Thus, the centralizer algebras constructed here all contain a Temperley-Lieb subalgebra.   The dimension of  $\TL_k(2)$ is the Catalan number $\mathcal{C}_{k} = \frac{1}{k+1}{{2k} \choose k}$, which  counts walks of $2k$ steps that begin and end at 0 on the representation graph  of $\SU_2$, i.e.  the Dynkin diagram $\mathsf{A}_{+\infty}$,  
 (see \eqref{SU2-Dynkin}). 

Our paper is organized as follows:
\begin{enumerate}
\item[(1)] In Section 1, we derive general properties of the centralizer algebras $\Z_k(\G)$  for subgroups
$\G$ of $\SU_2$.   Many of these results hold for subgroups that are not necessarily finite.  We study the tower $\Z_0(\G) \subseteq \Z_1(\G) \subseteq \Z_2(\G) \subseteq \cdots$ and show that $\Z_k(\G)$ can be constructed from $\Z_{k-1}(\G)$ by adjoining generators that are  (essential) idempotents;  usually there is just one additional generator except when we encounter a branch node in the graph. By using the Jones basic construction, we develop  a procedure for constructing  idempotent generators  of  $\Z_k(\G)$ inspired by the Jones-Wenzl idempotent construction.   

\item[(2)] Section 2 examines  the special case that  $\G$ is the cyclic subgroup $\C_n$.   In Theorems \ref{T:centcycbasis} and \ref{thm:dims}  we present  dimension formulas for $\Z_k(\C_n)$ and  its irreducible modules (equivalently, expressions for the number of walks  
on the affine Dynkin diagram of type  $\mathsf{\hat{A}}_{n-1}$, which is a circular graph with $n$ nodes).  We explicitly exhibit a basis of matrix units for $\Z_k(\C_n)$. These matrix units can be viewed using diagrams that correspond to subsets of $\{1, 2, \ldots, 2k\}$ that satisfy a special $\mathsf{mod}\,n$ condition (see Remark \ref{rem:cycdiag}).  We also consider the case that $\G$ is the infinite cyclic group $\C_\infty$, which has as its representation graph 
the Dynkin diagram $\mathsf{A}_{\infty}$.  Our results on $\C_\infty$, which are summarized in Theorem \ref{thm:cycinf},  show that  $\Z_k(\C_\infty)$ can be regarded,  in some sense, as the limiting case of $\Z_k(\C_n)$ as $n$ grows large.  The algebra 
$\Z_k(\C_\infty)$ is isomorphic to the planar rook algebra $\mathsf {PR}_k$   (see Remark \ref{rem:prook}).

\item[(3)]  Section 3 is devoted to the case that $\G$ is  the binary dihedral group $\DD_n$.  We compute $\dimm \Z_k(\DD_n)$ and the dimensions of the irreducible $\Z_k(\DD_n)$-modules  and construct a basis of matrix units  for
$\Z_k(\DD_n)$ (see Theorems \ref{thm:dicentbasis} and \ref{thm:simplemods}). These matrix units can be described diagrammatically using diagrams that correspond to set partitions of $\{1, 2, \ldots, 2k\}$ into at most 2 parts that satisfy a certain $\mathsf{mod}\,n$ condition.   Corollary \ref{Cor:Dwalks} gives expressions for
the number of walks of $k$ steps on the affine Dynkin diagram of type  $\hat{\mathsf{D}}_{n+2}$ 
starting at node $0$ and ending at node $\ac$.   
 Theorem \ref{thm:diinfcentbasis} treats the centralizer algebra $\Z_k(\DD_\infty)$  of the  infinite dihedral group  $ \DD_\infty$, which has as its representation graph the Dynkin diagram $\mathsf{D}_\infty$ and can be viewed  as the limiting case of the
groups $\D_n$.

\item[(4)] In Section 4, we illustrate how the results of Section 2 can be used to compute  $\dimm \Z_k(\G)$ for $\G = \TT, \OO, \II$ and  the dimensions of the irreducible modules of these algebras.  The case of $\II$ is noteworthy, as the expressions involve the Lucas numbers. \msk
\end{enumerate}  

The names for the exceptional subgroups $\TT, \OO, \II$ derive  from the fact that they are 2-fold covers of  classical polyhedral groups.   If  the center $\mathcal Z(\G) = \{\mathbf{ 1,-1}\}$ is factored out,  these groups have the following   quotients:  $\TT/ \{\mathbf{ 1,-1}\} \cong \mathsf{A}_4$, the alternating group on 4 letters, which  is the rotation  group of the tetrahedron; \ $\OO/  \{\mathbf{ 1,-1}\} \cong \mathsf{S}_4$, the symmetric group on 4 letters, which is the  rotation  group of the cube; \, and $\II/ \{\mathbf{ 1,-1}\} \cong \mathsf{A}_5$, the alternating group on 5 letters,  which is the  rotation  group of the icosahedron. 
An exposition of this result based on an argument of Weyl  can be found in \cite[Sec.~1.4]{Si}.  Our  sequel  \cite{BH} studies the exceptional centralizer algebras  $\Z_k(\G)$ for $\G = \TT, \OO, \II$, exhibiting remarkable connections between them and the Jones-Martin partition algebras. Another sequel \cite{GH} gives linear bases of $\Z_k(\G)$, which are indexed by pairs of paths on the Bratteli diagram $\mathcal B_\VV(\GG),$ and which are analogous to the Temperley-Lieb basis of $\Z_k(\SU_2)$.   

\bigskip
\noindent{\bf Acknowledgments}.   This paper was begun while the authors (G.B. and T.H.)  participated 
in the program  Combinatorial Representation Theory  
at the Mathematical Sciences Research Institute (MSRI) in Spring 2008.  They acknowledge
with gratitude the hospitality  and stimulating research
environment of MSRI.  G. Benkart is grateful to the Simons
Foundation for the support she received as  a Simons Visiting Professor at MSRI.  T. Halverson was partially supported by National Science Foundation grant DMS-0800085 and partially supported by Simons Foundation grant 283311.   J.\ Barnes was partially supported by  National Science Foundation grant DMS-0401098.  
 
The idea of studying McKay centralizer algebras  was inspired by conversions of T.\ Halverson  with Arun Ram, whom we thank for many useful discussions in the initial stages of the investigations.  The project originated as part of the senior capstone thesis \cite{Ba} of J.~Barnes  at Macalester College under the direction of T.~Halverson.  In \cite{Ba},  Barnes  found matrix unit bases, in a different format from what is presented here, for $\Z_k(\C_n)$ and $\Z_k(\DD_n)$  and discovered and proved the  dimension formulas in Theorem \ref{thm:ExceptionalDimensions}.   

\begin{section}{McKay Centralizer Algebras}

\begin{subsection}{$\SU_2$-modules}

Consider the special unitary group $\SU_2$ of $2 \times 2$ complex matrices defined by
\begin{equation}
\SU_2 = \left\{ \begin{pmatrix} \ \alpha & \beta \\ - \bar \beta & \bar \alpha \end{pmatrix} \bigg\vert\ \alpha, \beta \in \CC,  
\alpha \bar \alpha + \beta \bar \beta = 1 \right\},
\end{equation}
where $\bar \alpha$ denotes the complex conjugate of $\alpha$.  For each $r \ge  0,$ $\SU_2$ has an irreducible module $\VV(r)$ of dimension $r+1$.  The module $\VV = \VV(1) = \CC^2$ corresponds to the natural two-dimensional representation on which $\SU_2$ acts by matrix multiplication. Let  $\v_{-1} = (1,0)^{\tt t}, \v_1= (0,1)^{\tt t}$  (here $\tt t$ denotes transpose)
 be the standard basis for this action so that if $g = \left(\begin{smallmatrix} \ \alpha & \beta \\ - \bar \beta & \bar \alpha \end{smallmatrix}\right)$ then $g \v_{-1} = \alpha \v_{-1} - \bar\beta \v_1$ and $g\v_1 = \beta \v_{-1} + \bar \alpha \v_1$.

 Finite-dimensional modules for $\SU_2$ are completely reducible and satisfy the Clebsch-Gordan formula,
\begin{equation}\label{eq:ClebschGordan}
\VV(r) \otimes \VV = \VV(r-1) \oplus \VV(r+1),
\end{equation}
where  $\VV(-1) = 0$.    The representation graph $\mathcal{R}_{\VV}(\SU_2)$ is the infinite graph shown in \eqref{SU2-Dynkin} with vertices labeled by $r= 0, 1, 2, \cdots$ and an edge connecting vertex $r$ to vertex $r+1$ for each $r$
(which can be thought of as the Dynkin diagram ${\mathsf A}_{+\infty}$).   Vertex $r$ corresponds to $\VV(r)$,  and the edges correspond to the tensor product rule \eqref{eq:ClebschGordan}.  Above vertex $r$ we place $\dimm \VV(r) = r+1$.  The trivial module is indicated in white  and the defining module $\VV$ in black.
\begin{equation}\label{SU2-Dynkin}
\begin{array}{rll}
\SU_2: \qquad & \begin{array}{c} \includegraphics[scale=.70,page=1]{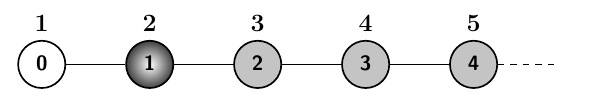} \end{array} &({\mathsf A}_{+\infty}) \\
\end{array}
\end{equation} 
\end{subsection}

\begin{subsection}{Subgroups of $\SU_2$ and their representation graphs}

Let $\G$ be a subgroup of  $\SU_2$.  Then $\G$ acts on the natural two-dimensional representation $\VV = \CC^2$
as $2 \times 2$ matrices with respect to the basis $\{\v_{-1}, \v_1\}$.    We assume that the tensor powers $\VV^{\ot k}$ of $\VV$  are completely reducible $\G$-modules (which is always the case
when $\G$ is finite), and let   $\{\, \G^\lambda \mid  \lambda \in \Lambda(\G)\, \}$ 
denote a complete set of pairwise non-isomorphic irreducible finite-dimensional  $\G$-modules
occurring in some  $\VV^{\ot k}$ for $k = 0,1, \dots$.  
We adopt the convention that $\VV^{\ot 0} =\G^{(0)}$, the trivial $\G$-module.    The representation graph $\mathcal{R}_{\VV}(\G)$ (also called the McKay graph or McKay quiver) is the graph with vertices labeled by elements of $\Lambda(\G)$ with $a_{\lambda,\mu}$ edges  from $\lambda$ to $\mu$ if the decomposition  of $\G^\lambda \otimes \VV$ into irreducible $\G$-modules is given by 
\begin{equation}\label{eq:tensorproductrule}
\G^\lambda \otimes \VV = \bigoplus_{\mu \in \Lambda(\G)} a_{\lambda,\mu} \, \G^\mu.
\end{equation}

The following properties of $\mathcal{R}_{\VV}(\G)$  hold for all finite subgroups $\G \subseteq \SU_2$ (see \cite{St}),
and we will assume that they hold for the groups considered here:
\begin{enumerate}\label{eq:graphasspts} 
\item $a_{\lambda,\mu} = a_{\mu,\lambda}$ for all pairs $\lambda,\mu \in \Lambda(\G)$.
\item $a_{\lambda,\lambda} = 0$ for all $\lambda \in \Lambda(\G), \, \lambda \neq 0$.
\item If $\G \not = \{{\bf 1}\}, \{{\bf 1},-{\bf 1}\}$, where $\bf 1$ is the $2 \times 2$ identity matrix,  then $a_{\lambda,\mu} \in \{0,1\}$ for all $\lambda,\mu \in \Lambda(\G)$.
\end{enumerate}
Thus, $\mathcal{R}_{\VV}(\G)$ is an undirected, simple graph.  
Since $\VV$ is faithful (being the defining module for $\G$), all irreducible $\G$-modules occur in 
some $\VV^{\ot k}$ when $\G$ is finite, and thus $\mathcal{R}_{\VV}(\G)$ is connected.   
Moreover, if  $c_{\lambda,\mu} = 2 \delta_{\lambda,\mu} - a_{\lambda,\mu}$ for $\lambda, \mu \in \Lambda(\G)$, where $\delta_{\lambda,\mu}$ is the Kronecker delta, then McKay \cite{McKay} observed that  $\mathcal C(\G) = [c_{\lambda,\mu}]$ is the Cartan matrix corresponding to the simply-laced affine Dynkin diagram of type $\mathsf{\hat{A}}_{n-1}$, $\mathsf{\hat{D}}_{n+2}$, $\mathsf{\hat{E}}_6$, $\mathsf{\hat{E}}_7$, $\mathsf{\hat{E}}_8$, when $\G$ is one of the finite groups  $\C_n, \D_n, \TT, \OO,$ and $\II$,
respectively.    
The trivial module $\G^{(0)}$ corresponds to the affine node in those cases.   

\end{subsection}

\begin{subsection}{Tensor powers and Bratteli diagrams}

For $k \ge 1$,  the $k$-fold tensor power 
 $\VV^{\otimes k}$ is $2^k$-dimensional and has a basis of simple tensors 
\[
\VV^{\otimes k} =\spann_\CC\left\{\ \v_{r_1} \ot \v_{r_2} \ot \cdots \ot \v_{r_k} \mid  \ r_j \in \{-1,1\} \ \right\}.
\]  If $\rf = (r_1, \ldots, r_k) \in \{-1,1\}^k$, we 
adopt the notation
\begin{equation}\label{SimpleTensorNotation}
\v_\rf = \v_{(r_1, \ldots, r_k)} =  \v_{r_1} \ot \v_{r_2} \ot \cdots \ot \v_{r_k} 
\end{equation} as a shorthand. 
Group elements $g \in \G$ act on simple tensors  by the diagonal action
\begin{equation}\label{DiagonalAction}
g (\v_{r_1} \ot \v_{r_2} \ot \cdots \ot \v_{r_k}) =  g \v_{r_1} \ot g\v_{r_2} \ot \cdots \ot g\v_{r_k}.
\end{equation}

Let
\begin{equation}
\Lambda_k(\G) = \{\ \lambda \in \Lambda(\G) \mid \G^\lambda \hbox{ appears as a summand in the decomposition of $\VV^{\otimes k}$}\}
\end{equation}
index the  irreducible $\G$-modules occurring in  $\VV^{\ot k}$.  
Then, since  $\mathcal{R}_\VV(\G)$ encodes the tensor product rule \eqref{eq:tensorproductrule}, 
$\Lambda_k(\G)$  is the set of vertices in $\mathcal{R}_\VV(\G)$ that can be reached by  walks of $k$ steps  starting from 0.  Furthermore, 
\begin{equation}
\Lambda_{k}(\G) \subseteq \Lambda_{k+2}(\G), \qquad\hbox{for all $k \ge 0$},
\end{equation}
since if a node can be reached in $k$ steps,  then it can also be reached in $k+2$ steps.  

The Bratteli diagram $\mathcal{B}_{\VV}(\G)$ is the infinite graph with vertices labeled by $\Lambda_k(\G)$ on level $k$ and $a_{\lambda,\mu}$ edges from vertex $\lambda \in \Lambda_k(\G)$ to vertex $\mu \in \Lambda_{k+1}(\G)$.  The Bratteli diagram for $\SU_2$ is shown in Figure \ref{fig:SU_2Bratteli}, and the Bratteli diagrams corresponding to $\C_n, \D_n, \TT, \OO,$ $\II$,  as well as to
the infinite subgroups $\C_\infty$, $\D_\infty$, are displayed in Section 4.2.
\begin{figure}[h!]\label{fig:SU_2Bratteli}
\[
\begin{array}{c} \includegraphics[scale=.80,page=9]{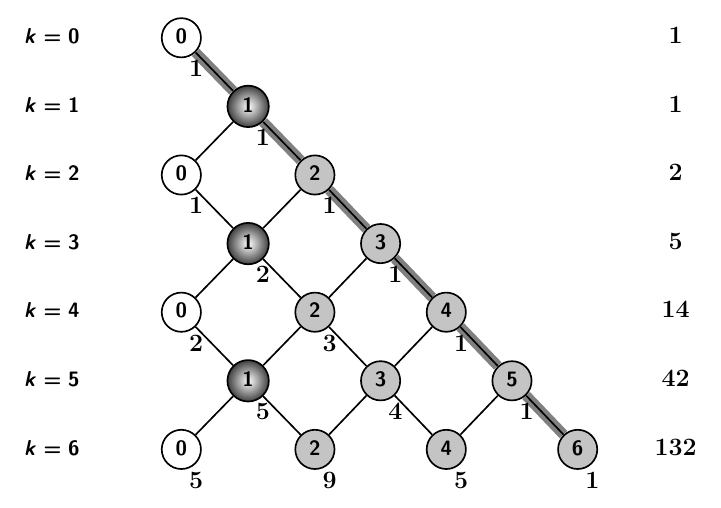} \end{array}
\]
\caption{ Levels $0,1,\ldots, 6$ of the Bratteli diagram for $\SU_2$.}
\label{fig:SU_2Bratteli}
\end{figure} 

A walk of length $k$ on the representation graph $\mathcal{R}_{\VV}(\G)$ from $0$ to $\lambda \in \Lambda(\G)$ is a sequence $\big(0,{\lambda^1}, {\lambda^2}, \ldots,$ ${\lambda^k}=\lambda\big)$  starting at $\lambda^0 = 0$ such that  $\lambda^j \in \Lambda(\G)$   for each $1 \le j \le k$, and  $\lambda^{j-1}$ is connected to  ${\lambda^{j}}$ by an edge in $\mathcal{R}_{\VV}(\G)$.   Such a walk is equivalent to a unique path of length $k$ on the Bratteli diagram $\mathcal{B}_{\VV}(\G)$ from $0 \in \Lambda_0(\G)$ to $\lambda \in \Lambda_k(\G)$.
Let $\mathcal{W}_k^\lambda(\G)$ denote the set of walks on $\mathcal{R}_{\VV}(\G)$ of length $k$ from $0 \in \Lambda_k(\G)$ to $\lambda \in \Lambda_k(\G)$, and 
let $\mathcal{P}_k^\lambda(\G)$   denote the set of paths on $\mathcal{B}_{\VV}(\G)$ of length $k$ from $0 \in \Lambda_0(\G)$ to $\lambda \in \Lambda_k(\G)$. Thus,
$\vert\mathcal{P}_k^\lambda(\G)\vert = \vert\mathcal{W}_k^\lambda(\G)\vert$. 

A pair of walks of length $k$ from $0$ to $\lambda$ corresponds uniquely (by reversing the second walk) to a  walk of length $2k$ beginning and ending at $0$.   Hence, \begin{equation}\label{eq:walks} \vert  \mathcal{W}_{2k}^{0}(\G) \vert  =  \sum_{\lambda \in \Lambda_k(\G)} 
\vert\mathcal{W}^\lambda_k(\G) \vert^2 = 
 \sum_{\lambda \in \Lambda_k(\G)} \vert\mathcal{P}^\lambda_k(\G)\vert ^2 = \vert\mathcal{P}^{0}_{2k}(\G)\vert.\end{equation}
 
Let $m_k^\lambda$  denote the multiplicity of $\G^\lambda$ in $\VV^{\otimes k}$.  Then, by induction on \eqref{eq:tensorproductrule} and the observations in the previous paragraph, we see that this multiplicity is enumerated as
\begin{equation}
\begin{aligned}
m_k^\lambda&=  \vert\mathcal{W}_k^\lambda(\G)\vert  
= \#(\hbox{walks on  $\mathcal{R}_{\VV}(\G)$ of length $k$ from at $0$ to $\lambda$}) \\
&= \vert\mathcal{P}^\lambda_k(\G)\vert 
= \#(\hbox{paths in  $\mathcal{B}_{\VV}(\G)$ of length $k$ from $0 \in \Lambda_0(\G)$  to $\lambda \in \Lambda_k(\G)$}).
\end{aligned}
\end{equation}

The first six rows of the Bratteli diagram $\mathcal{B}_{\VV}(\SU_2)$ for $\SU_2$ are displayed  in Figure \ref{fig:SU_2Bratteli}, and the labels below vertex $r$ on level $k$ give the number of paths from the top of the diagram to $r$, which is also multiplicity of $\VV(r)$ in $\VV^{\otimes k}$.  These numbers also give the number of walks 
of length $k$  from $0$ to $r$ on the representation graph $\mathcal{R}_\VV(\SU_2)$. The column to the right contains the sum of the squares of the multiplicities.
\end{subsection}
\begin{subsection}{Schur-Weyl duality}

The \emph{centralizer of $\G$ on $\VV^{\otimes k}$} is the algebra
\begin{equation}
\Z_k(\G) = \End_\G(\VV^{\otimes k}) = \left\{\  a \in \End(\VV^{\otimes k}) \  \big  |  \   a(g w) = ga(w) \hbox{ for all $g \in \G$, $w \in \VV^{\otimes k}$} \ \right\}.
\end{equation}
If the group $\G$ is apparent from the context,  we will simply write  $\Z_k$ for $\Z_k(\G)$.   
Since  $\VV^{\otimes 0} =  \G^{(0)}$, we have $\Z_0(\G) = \CC 1$. There is a natural embedding $\iota: \Z_k(\G) \hookrightarrow \Z_{k+1}(\G)$ given by 
\begin{equation}\label{embedding}
\begin{array}{cccc}
\iota:&\Z_k(\G) & \to & \Z_{k+1}(\G) \\
&a & \mapsto & a \otimes {\bf 1}
\end{array}
\end{equation}
where $a \otimes {\bf 1}$ acts as $a$ on the first $k$ tensor factors and ${\bf 1}$ acts as the identity in the $(k+1)$st tensor position. 
Iterating this embedding gives an infinite tower of algebras 
\begin{equation}\label{eq:tower}
\Z_0(\G) \subseteq \Z_1(\G) \subseteq \Z_2(\G) \subseteq \cdots .
\end{equation}

By classical double-centralizer theory (see for example \cite[Secs. 3B and 68]{CR}), we know the
following:
\begin{itemize}
\smallskip
\item[$\bullet$] $\Z_k(\G)$ is a semisimple associative $\CC$-algebra whose  irreducible modules $
\left\{\Z_k^\lambda \, \big | \,  \lambda\in \Lambda_k(\G) \right\}
$ are  labeled by $\Lambda_k(\G)$.  
\smallskip
\item[$\bullet$] $\dimm \Z_k^\lambda = m_k^\lambda = \vert\mathcal{W}_k^\lambda(\G)| = \vert\mathcal{P}_k^\lambda(\G)\vert$.
\smallskip
\item[$\bullet$] The edges from level  $k-1$ to level $k$ in $\mathcal{B}_{\VV}(\G)$ represent the induction and restriction rules  for $\Z_{k-1}(\G) \subseteq \Z_{k}(\G)$.  
\smallskip
\item[$\bullet$] If $d^\lambda = \dimm \G^\lambda$, then the tensor space $\VV^{\otimes k}$ has the following decompositions:  \begin{equation}
\begin{array}{rll}
\VV^{\otimes k} & \cong \displaystyle{\bigoplus_{\lambda\in \Lambda_k(\G)}}\,\, m_k^\lambda\, \G^\lambda & \hbox{ as a $\G$-module}, \\
& \cong \displaystyle{\bigoplus_{\lambda \in \Lambda_k(\G)}}\,  d^\lambda \, \Z_k^\lambda & \hbox{ as a $\Z_k(\G)$-module}, \\
& \cong \displaystyle{\bigoplus_{\lambda \in \Lambda_k(\G)}}\left(\G^\lambda \otimes \Z_k^\lambda \right) & \hbox{ as a $(\G,\Z_k(\G))$-bimodule}. \\
\end{array}
\end{equation}
As an immediate consequence of these isomorphisms, we have from counting dimensions that 
\begin{equation} \label{eq:bimoduleidentity}
(\dimm \VV)^k =  \sum_{\lambda \in \Lambda_k(\G)} d^\lambda  m_k^\lambda.
\end{equation}
\smallskip
\item[$\bullet$] By general Wedderburn theory, the dimension of $\Z_k(\G)$ is the sum of the squares of the dimensions of its irreducible modules,  
\begin{equation}\label{eq:evida}
\dimm \Z_k(\G)  = \sum_{\lambda \in \Lambda_k(\G)} (m_k^\lambda)^2 = \sum_{\lambda \in \Lambda_k(\G)} \vert\mathcal{W}_k^\lambda(\G)\vert^2
 = \sum_{\lambda \in \Lambda_k(\G)}   \vert\mathcal{P}_k^\lambda(\G)\vert ^2.
\end{equation}
Therefore, when the adjacency matrix of the  representation graph $\mathcal{R}_\VV(\G)$ is symmetric,  it follows as in \eqref{eq:walks} that  
\begin{equation}\label{eq:evid}
\dimm \Z_k(\G)  = \sum_{\lambda \in \Lambda_k(\G)} (m_k^\lambda)^2 = m_{2k}^{0} = \left\vert \mathcal{W}_{2k}^{0}(\G)\right\vert = 
\dimm \Z_{2k}^{(0)},
\end{equation}
the number of walks  of length $2k$ that begin and end at 0 on $\mathcal{R}_\VV(\G)$ (which is the associated affine Dynkin diagram when $\G$ is a finite subgroup of $\SU_2$).
\end{itemize}

\end{subsection}
\begin{subsection}{The Temperley-Lieb algebras}

Let $\S_k$ denote the symmetric group of permutations on $\{1,2, \ldots,k\}$, and let $\sigma \in \S_k$ act on a simple tensor by place permutation as follows:
\[
\sigma \cdot (\v_{r_1} \ot \v_{r_2} \ot \cdots \ot \v_{r_k}) = \v_{r_{\sigma^{-1}(1)}} \ot \v_{r_{\sigma^{-1}(2)}} \ot \cdots \ot \v_{r_{\sigma^{-1}(k)}}.
\]  
It is well known, and easy to verify, that under this action $\S_k$ commutes with $\SU_2$ on $\VV^{\otimes k}$.  Thus,  there is a representation $\Phi_k: \CC \S_k \to \End_{\SU_2}(\VV^{\otimes k})$;  however, this map is injective only for $k \le 2$.

For $1 \le i \le k-1$, let $\sigma_i = (i\ i+1) \in \S_k$ be the simple transposition that exchanges $i$ and $i+1$, and set
\begin{equation}\label{eidef}
e_i = 1 - \sigma_i. 
\end{equation}
Then $e_i$ acts on tensor space as
\begin{equation}\label{eq:eirep}
\e_i = \underbrace{{\bf 1} \otimes \cdots \otimes {\bf 1}}_{\text{$i-1$ factors}}\ \otimes\ \e \otimes\! \underbrace{ {\bf 1} \otimes \cdots \otimes {\bf 1}}_{\text{$k-i-1$ factors}}, 
\end{equation} 
where ${\bf 1}$ is the $2 \times 2$ identity matrix, which we identify with the identity map $\mathsf{id}_\VV$ of $\VV$, and  $\e: \VV \otimes \VV \to \VV \otimes \VV$ acts in tensor positions $i$ and $i+1$ by
\begin{equation}\label{edef}
 \e (\v_\ell \otimes \v_m) =  \v_\ell \otimes \v_m - \v_m \otimes \v_\ell, \qquad \ell,m \in \{-1,1\}.
\end{equation}
For any $\G \subseteq \SU_2$, the vector space $\VV^{\ot 2} = \VV \ot \VV$ decomposes into $\G$-modules as 
\[
\VV^{\ot 2} = \mathsf{A}(\VV^{\ot 2}) \oplus \mathsf{S}(\VV^{\ot 2}),
\]
where $\mathsf{A}(\VV^{\ot 2}) =\spann_\CC \{\v_{-1} \ot \v_1 - \v_1 \ot \v_{-1}\}$ are the antisymmetric tensors and 
$\mathsf{S}(\VV^{\ot 2}) = \spann_\CC \{\v_{-1} \ot \v_{-1}, \v_{-1}\ot \v_1 + \v_1 \ot \v_{-1}, \v_1 \ot \v_1\}$ are the symmetric tensors.
The operator $\e: \VV^{\ot 2}  \to \VV^{\ot 2} $ projects onto the $\G$-submodule  $\mathsf{A}(\VV^{\ot 2})$
and  $\frac{1}{2}\e$ is the corresponding idempotent.  

The image  $\ima(\Phi_k)$ of the representation $\Phi_k: \CC \S_k \to \End_{\SU_2}(\VV^{\otimes k})$ can be identified with the Temperley-Lieb algebra $\TL_k(2)$.
Recall that the  Temperley-Lieb algebra $\TL_k(2)$ is the unital associative algebra with generators $\e_1,\dots, \e_{k-1}$ 
 and relations
\begin{equation}\label{eq:TLrels}
\begin{array}{lll}
{\rm (TL1)}\qquad & \e_i^2 = 2 \e_i, & 1 \le i \le k-1, \\
{\rm (TL2)} &\e_i  \e_{i \pm 1} \e_i = \e_i, \quad & 1 \le i \le k-1, \qquad    \\
{\rm (TL3)} & \e_i \e_j = \e_j \e_i, & |i-j| > 1,
\end{array} \hskip1truein
\end{equation}
(see \cite{TL} and \cite{GHJ}).  Since the generator $\e_i$ in $\TL_k(2)$  is identified with the map in \eqref{eq:eirep},  we
are using the same notation for them.   If  $\Psi_k: \SU_2 \to  \End(\VV^{\otimes k})$ is the tensor-product representation,  then  $\SU_2$ and $\TL_k(2)$ generate full centralizers of each other in $\End(\VV^{\otimes k})$, so that 
\begin{equation}
\TL_k(2) \cong\ima(\Phi_k) = \End_{\SU_2} (\VV^{\otimes k}) \qquad\hbox{and}\qquad
\ima(\Psi_k) = \End_{\TL_k(2)} (\VV^{\otimes k}).
\end{equation}

Since $\TL_k(2) \cong \Z_k(\SU_2) = \End_{\SU_2} (\VV^{\otimes k})$, the set  
\[
\Lambda_k(\SU_2) = \begin{cases}
\{0,2,\ldots, k\}, & \text{if $k$ is even}, \\
\{1,3,\ldots, k\}, & \text{if $k$ is odd},
\end{cases}
\]
also indexes the irreducible $\TL_k(2)$-modules. 
The number of walks of length $k$ from 0 to $k-2\ell  \in \Lambda_k(\SU_2)$ on $\mathcal{R}_{\VV}(\SU_2)$ is equal to the number of walks from 0 to $k-2\ell$ on the set $\NN$ of natural numbers and is known to be (see \cite[p. 545]{Wb}, \cite[Sec.\ 5]{Jo})
\[
 \left \langle {{k}\atop {\ell}} \right\rangle  := {k \choose \ell}-{k \choose \ell-1}.
\]  
For each $k-2\ell \in \Lambda_k(\SU_2)$, where $\ell=0,1,\dots, \lfloor k/2\rfloor$,   let $\TL_k^{(k-2\ell)} = \Z_k^{(k-2\ell)}$ be the irreducible $\TL_k(2)$ module labeled by $k-2\ell$.  Then $\TL_k^{(k-2\ell)}$ has dimension $\left \langle {{k}\atop {\ell}} \right\rangle$, and these modules are constructed explicitly in \cite{Wb}.   Moreover,  \begin{equation}
\begin{array}{rll}
\VV^{\otimes k} & \cong \displaystyle{\bigoplus_{k-2\ell \in \Lambda_k(\SU_2)}}\,\, \left \langle {{k}\atop {\ell}} \right\rangle \VV(k-2\ell), & \hbox{ as an $\SU_2$-module}, \\
& \cong \displaystyle{\bigoplus_{k-2\ell \in \Lambda_k(\SU_2)}}\,  (k-2\ell+1) \TL_k^{(k-2\ell)}, & \hbox{ as a $\TL_k(2)$-module}, \\
& \cong \displaystyle{\bigoplus_{k-2\ell \in \Lambda_k(\SU_2)}}\left(\VV(k-2\ell) \otimes \TL_k^{(k-2\ell)} \right), & \hbox{ as an $(\SU_2,\TL_k(2))$-bimodule}. \\
\end{array}
\end{equation}
The dimension of $\TL_k(2)$ is given by   the Catalan number ${\mathcal C}_k
= \frac{1}{k+1}{{2k} \choose k}$, as can be seen in the right-hand column of the 
Bratteli diagram for $\SU_2$ in Figure \ref{fig:SU_2Bratteli}.

\end{subsection}

\begin{subsection} {The Jones basic construction}
Let  $\G$ be a subgroup of $\SU_2$ such that $\G \neq \{{\bf 1}\}, \{{\bf -1}, {\bf 1}\}$.
Any transformation that commutes with $\SU_2$ on $\VV^{\ot k}$  also commutes with $\G$.  Thus,  we have the reverse inclusion of centralizers $\TL_k(2) =\End_{\SU_2}(\VV^{\otimes k}) \subseteq \End_{\G}(\VV^{\otimes k}) = \Z_k(\G)$
and  identify the subalgebra of  $ \End_{\G}(\VV^{\otimes k})$  generated by the $\e_i$ in \eqref{eq:eirep} with $\TL_k(2)$.  
In this section, we use the Jones basic construction to find additional generators for the
centralizer algebra $\Z_k = \Z_k(\G) =  \End_{\G}(\VV^{\otimes k})$ for each $k$.   The construction uses the
natural embedding  of $\Z_k$ into $\Z_{k+1}$ given by  $a  \mapsto  a \ot {\bf 1}$, which holds for any $k \geq 1$.

In what follows,   if $\mathsf{q} = (q_1,\dots, q_{k}) \in \{-1,1\}^{k}$ and $
\mathsf{r} = (r_1,\dots, r_\ell) \in \{-1,1\}^\ell$ for some $k,\ell \geq 1$,  then $[\mathsf{q},\mathsf{r}] = (q_1,\dots,q_k,r_1,\dots,r_\ell) 
\in \{-1,1\}^{k+\ell}$ is the concatenation of the two tuples.   In particular,  if $t \in \{-1,1\}$,  then $[\mathsf{q},t] = (q_1,\dots, q_k,t)$.     

Now if  $a \in  \End(\VV^{\ot k})$,   say  $a = \sum_{\mathsf{r,s} \in \{-1,1\}^k } a^{\mathsf r}_{\mathsf s} \mathsf{E}_{\mathsf r, \mathsf s}$,
where $\mathsf{E}_{\mathsf r,\mathsf s}$ is the standard matrix unit,  then  under the embedding $a \mapsto  a \ot {\bf 1}$,   

\begin{equation}\label{eq:embed} a^{[{\mathsf r}, r_{k+1}]}_{[{\mathsf s}, s_{k+1}]}  = (a \ot {\bf 1})^{[{\mathsf r}, r_{k+1}]}_{[{\mathsf s}, s_{k+1}]}  = \delta_{r_{k+1},s_{k+1}} a^{\mathsf r}_{\mathsf s}, \end{equation} 
where $r_{k+1}, s_{k+1} \in \{-1,1\}$.   \msk

{\it Note  in this section we are writing $a^{\mathsf r}_{\mathsf s}$ rather
than $a_{\mathsf{r, s}}$ to simplify the notation.} \msk

Define a map $\varepsilon_k:  \End(\VV^{\ot k})  \rightarrow \End(\VV^{\ot (k-1)})$,  
$\varepsilon_k(a) = \sum_{\mathsf p, \mathsf q \in \{-1,1\}^{k-1}} \varepsilon_k(a)^{\mathsf p}_{\mathsf q} \mathsf{E_{p,q}}$,
called the {\it conditional expectation}, such
that 
\begin{equation}\label{eq:varep} \varepsilon_k(a)^{\mathsf p}_{\mathsf q}  =  \half \left(a^{[\mathsf p, -1]}_{[\mathsf q,-1]} +   a^{[\mathsf p, 1]}_{[\mathsf q,1]}\right)\end{equation} 
for all $\mathsf p, \mathsf q \in \{-1,1\}^{k-1}$ and all $a \in \End(\VV^{\ot k})$.  

\begin{prop}\label{prop:condexp}  Assume $k \geq 1$.
\begin{itemize}
\item[{\rm (a)}]  If  $a \in \End(\VV^{\ot k}) \subseteq \End(\VV^{\ot (k+1)})$,  then $\e_k a \e_k  = 2 \varepsilon_k(a) \ot \e = 2\varepsilon_k(a) \e_k$. 
\item[{\rm (b)}] If $a \in \Z_k$, then $\varepsilon_k(a) \in \Z_{k-1}$, so that $\varepsilon_k: \Z_k \rightarrow \Z_{k-1}$. 
\item[{\rm (c)}]  $\varepsilon_k: \Z_k \rightarrow \Z_{k-1}$ is a $(\Z_{k-1}, \Z_{k-1})$-bimodule map; that is,  $\varepsilon_k(a_1 b a_2) 
= a_1 \varepsilon_k(b) a_2$ for all $a_1, a_2 \in \Z_{k-1} \subseteq \Z_k$, $b \in \Z_k$.   In particular,  $\varepsilon_k(a) = a$
for all $a \in \Z_{k-1}$.  
\item[{\rm (d)}]  Let $\Tr_k$ denote the usual (nondegenerate)  trace on $\End(\VV^{\otimes k})$.    Then for
all $a \in \Z_k$ and $b \in \Z_{k-1}$,  we have $\Tr_k(ab) = \Tr_{k}(\varepsilon_k(a)b)$.  
\end{itemize}
\end{prop} 
\begin{proof}  (a)  It suffices to show that these expressions have the same action on a simple tensor $\v_{\mathsf r}$
where $\mathsf{r}  = (r_1, \dots, r_k, r_{k+1}) \in \{-1,1\}^{k+1}$.   If $r_k = r_{k+1}$,  then both 
$\e_k a \e_k$ and  $2 \varepsilon_k(a) \ot \e$ act as 0 on $\v_{\mathsf r}$.   So we suppose that $(r_k, r_{k+1}) = (-1,1)$,  and
let $\mathsf{p} = (r_1,\dots, r_{k-1})$.    Now
\begin{eqnarray*} (\e_k a \e_k)\v_{\mathsf r}  &=&  \e_k a\left (\v_{\mathsf p} \ot \v_{-1} \ot \v_1 - \v_{\mathsf p} \ot \v_{1} \ot \v_{-1}\right) \\
&=&  \e_k \sum_{\mathsf q \in \{-1,1\}^{k}}a^{\mathsf q}_{[{\mathsf p},-1]} \v_{\mathsf q}  \ot \v_1 -  \e_k \sum_{\mathsf q \in \{-1,1\}^{k}}a^{\mathsf q}_{[\mathsf p,1]} \v_{\mathsf q}  \ot \v_{-1} \\
&=& \sum_{\mathsf n \in \{-1,1\}^{k-1}}a^{[\mathsf n,-1]}_{[\mathsf p,-1]}\left( \v_{\mathsf n}  \ot \v_{-1}\ot \v_1 -\v_{\mathsf n}  \ot \v_{1}\ot \v_{-1} \right) \\
&& \qquad -\sum_{\mathsf n \in \{-1,1\}^{k-1}}a^{[\mathsf n,1]}_{[\mathsf p,1]}\left( \v_{\mathsf n}  \ot \v_{1}\ot \v_{-1} -  
\v_{\mathsf n}  \ot \v_{-1}\ot \v_{1} \right).\end{eqnarray*}  
Thus, the coefficient of $\v_{\mathsf n} \ot \v_{-1} \ot \v_1$ is $a^{[\mathsf n,-1]}_{[\mathsf p,-1]} + a^{[\mathsf n,1]}_{[\mathsf p,1]}$,
and the coefficient of $\v_{\mathsf n} \ot \v_{1} \ot \v_{-1}$ is the negative of that expression.   These are
exactly the coefficients that we get when $2 \varepsilon_k(a) \ot \e$ acts on $\v_{\mathsf r}$.   The proof when 
$(r_k, r_{k+1}) = (1,-1)$ is completely analogous. 

(b) When $a \in \Z_k \subseteq \Z_{k+1}$,  then $ \e_k a \e_k  \in \Z_{k+1}$, so from part (a) it follows that 
$\varepsilon_k(a) \ot \e \in \Z_{k+1}$.  Therefore, if $g \in \G$,  then $\varepsilon_k(a) \ot \e \in \Z_{k+1}$ commutes
with $g^{\ot (k+1)} = g^{\ot (k-1)} \ot g^{\ot 2}$ on $\VV^{(k-1)} \ot \VV^{\ot 2}$.   Since these actions occur independently 
on the first $k-1$ and last 2 tensor slots,  $\varepsilon_k(a)$ commutes with $g^{\ot (k-1)}$ for all $g \in \G$.   Hence
$\varepsilon_k(a) \in \Z_{k-1}$.   

(c)  Let $a_1, a_2 \in \Z_{k-1}$ and $b \in \Z_k$.   Then using \eqref{eq:embed} and \eqref{eq:varep} we have for
$\mathsf{p,q} \in \{-1,1\}^{k-1}$,   
\begin{eqnarray*} \varepsilon_k(a_1ba_2)_{\mathsf q}^{\mathsf p} &=&  \half \sum_{t \in \{-1,1\}}  (a_1 b a_2)^{[{\mathsf p},t]}_{[\mathsf q,t]}
= \half  \sum_{t \in \{-1,1\}} \sum_{\mathsf{m,n} \in \{-1,1\}^{k-1}} (a_1)^{[{\mathsf p},t]}_{[{\mathsf n},t]}b^{[{\mathsf n},t]}_{[{\mathsf m},t]}   (a_2)^{[{\mathsf m},t]}_{[{\mathsf q},t]} \\
&=& \half   \sum_{t \in \{-1,1\}} \sum_{\mathsf{m,n} \in \{-1,1\}^{k-1}} (a_1)^{{\mathsf p}}_{{\mathsf n}}b^{[{\mathsf n},t]}_{[{\mathsf m},t]}   (a_2)^{{\mathsf m}}_{{\mathsf q}} \\
&=& \sum_{\mathsf{m,n} \in \{-1,1\}^{k-1}} (a_1)^{{\mathsf p}}_{{\mathsf n}}\left(\half   \sum_{t \in \{-1,1\}} b^{[{\mathsf n},t]}_{[{\mathsf m},t]}\right)   (a_2)^{{\mathsf m}}_{{\mathsf q}} \\
&=&\left(a_1 \varepsilon_k(b) a_2 \right)_{\mathsf q}^{\mathsf p}.  
\end{eqnarray*} 

(d)  Let $a \in \Z_k$ and $b \in \Z_{k-1} \subseteq \Z_k$.    Then applying \eqref{eq:embed} gives
\begin{eqnarray*} \Tr_k(ab) &=& \sum_{\mathsf r \in \{-1,1\}^k} \sum_{\mathsf s \in \{-1,1\}^k}   a^{\mathsf r}_{\mathsf s} b^{\mathsf s}_{\mathsf r} =  \sum_{\mathsf{r} = [\mathsf{r'}, r_{k}]  \in \{-1,1\}^k} \sum_{\mathsf s = [\mathsf{s'},s_{k}]  \in \{-1,1\}^k}   a^{\mathsf r}_{\mathsf s} \delta_{r_k, s_k} b^{\mathsf s'}_{\mathsf r'} \\
&=& \sum_{\mathsf{r'}\in \{-1,1\}^{k-1}} \sum_{\mathsf{s'} \in \{-1,1\}^{k-1}} \left(a^{[\mathsf{r'}, -1]}_{[\mathsf{s'},-1]} +
a^{[\mathsf{r'}, 1]}_{[\mathsf{s'},1]} \right) b^{\mathsf s'}_{\mathsf r'} \\
&=& \sum_{\mathsf{r'}\in \{-1,1\}^{k-1}} \sum_{\mathsf{s'} \in \{-1,1\}^{k-1}} 2 \varepsilon_k(a)^{\mathsf r'}_{\mathsf s'}b^{\mathsf s'}_{\mathsf r'} \\
&=&\sum_{\mathsf{r} = [\mathsf{r'}, r_{k}]  \in \{-1,1\}^k} \sum_{\mathsf s = [\mathsf{s'},s_{k}]  \in \{-1,1\}^k}  \delta_{r_k, s_k}  \varepsilon_k(a)^{\mathsf r}_{\mathsf s}b^{\mathsf s}_{\mathsf r} \\
&=& \Tr_k(\varepsilon_k(a)b).
\end{eqnarray*}   \end{proof}

Relative to the inner product $\langle, \rangle: \Z_k \times \Z_k \to \CC$ defined by
$\langle a, b \rangle = \Tr_k(ab),$ for all $a,b \in \Z_k$,  the conditional expectation $\varepsilon_k$  is the orthogonal  projection  $\varepsilon_k: \Z_k \to \Z_{k-1}$ with respect to $\langle, \rangle$  since $\langle a-\varepsilon_k(a), b\rangle = \Tr_k(ab)-\Tr_k(\varepsilon_k(a)b) = 0$ by Proposition \ref{prop:condexp} (d). 
Its values are uniquely determined by the 
nondegeneracy of the trace.   

Proposition  \ref{prop:condexp}\,(a) tells us that $\Z_k \e_k \Z_k$ is a subalgebra of $\Z_{k+1}$.   Indeed,
for $a_1,a_2, a_1',$  $a_2'$ in $\Z_k$,  we have 
\[(a_1 \e_k a_2)(a_1'\e_k a_2') = a_1 \e_k (a_2 a_1') \e_k a_2' = 2 a_1 \varepsilon_k(a_2 a_1') \e_k a_2' \in \Z_k \e_k \Z_k.\]
Part (b) of the next result says that in fact $\Z_k \e_k \Z_k$ is an ideal of $\Z_{k+1}$.  

\begin{prop}\label{prop:ZkZk1}  For all $k \geq 0$,  \begin{itemize}
\item[{\rm (a)}]  For $a \in \End(\VV^{\ot (k+1)})$ there is a {\it unique} $b \in \End(\VV^{\ot k})$  so 
$a \e_k  = (b \ot {\bf 1})\e_k$,  and  for all $\mathsf{r,s} \in \{-1,1\}^k$, 
\begin{equation}\label{eq:brs} b_{\mathsf s}^{\mathsf r} =  \half\left(a_{[\mathsf{s},-s_k]}^{[\mathsf{r}, -s_k]} - a^{[\mathsf{r},-s_k]}_{[\mathsf{s'},-s_k,s_k]}\right), \end{equation}
where  $\mathsf{s' }= (s_1,\dots, s_{k-1})$, and  $\mathsf{s} = [\mathsf{s'}, s_k]$.   If $a \in \Z_{k+1}$,  then $b \in \Z_k$.   
\item[{\rm (b)}] $\Z_k \e_k \Z_k = \Z_{k+1} \e_k \Z_{k+1}$ is an ideal of $\Z_{k+1}$.
\item[{\rm (c)}]  The map $\Z_k \rightarrow \Z_k \e_k \subseteq \Z_{k+1}$ given by $a \mapsto a\e_k$ is injective.
\end{itemize}
\end{prop}

\begin{proof}  (a)  First note that for all $\mathsf{p,q} \in \{-1,1\}^{k+1}$,  
\begin{equation}\label{eq:eek}  (\e_k)^{\mathsf p}_{\mathsf q} = \frac{1}{4}\bigg(\prod_{j=1}^{k-1} \delta_{p_j,q_j}\bigg)(p_{k+1}-p_k)(q_{k+1}-q_k). \end{equation}  
Now assume $a \in \End(\VV^{\ot (k+1)})$,  $b \in \End(\VV^{\ot k})$ and  $\mathsf{n,q} \in \{-1,1\}^{k+1}$, and let $\mathsf{q}'' = (q_1, \dots, q_{k-1})$.   Then  

\begin{eqnarray*} (a\e_k)^{\mathsf n}_{\mathsf q}  &=&  \sum_{\mathsf{p} \in \{-1,1\}^{k+1}}  a^{\mathsf n}_{\mathsf p}\, (\e_k)_{\mathsf q}^{\mathsf p} \\
 &=&  \frac{1}{4}\sum_{\mathsf{p} \in \{-1,1\}^{k+1}} a^{\mathsf n}_{\mathsf p}\left(\prod_{j=1}^{k-1} \delta_{p_j,q_j}\right)(p_{k+1}-p_k)(q_{k+1}-q_k),\\
 &=&  \frac{1}{4}\sum_{p_k,p_{k+1} \in \{-1,1\}}a^{\mathsf n}_{[\mathsf q'',p_k, p_{k+1}]}(p_{k+1}-p_k)(q_{k+1}-q_k), \\
  &=&  \frac{1}{2}(q_{k+1}-q_k) \left( a^{\mathsf n}_{[\mathsf{q''}, -1,1]} - a^{\mathsf n}_{[\mathsf{q''}, 1,-1]}\right),   \qquad \hbox{\rm while}\\
\end{eqnarray*}
\begin{eqnarray*}
 \left( (b \ot {\bf 1})\e_k\right)^{\mathsf n}_{\mathsf q}  &=&  \sum_{\mathsf{p} \in \{-1,1\}^{k+1}}(b \ot {\bf 1})^{\mathsf n}_{\mathsf p}\,(\e_k)
 ^{\mathsf p}_{\mathsf q}\\  &=& \frac{1}{4}\sum_{\mathsf{p} \in \{-1,1\}^{k+1}} (b \ot {\bf 1})^{\mathsf n}_{\mathsf p}\left(\prod_{j=1}^{k-1} \delta_{p_j,q_j}\right)(p_{k+1}-p_k)(q_{k+1}-q_k)\\ 
 &=& \frac{1}{4}(q_{k+1}-q_k)\sum_{p_k \in \{-1,1\}} (b \ot {\bf 1})^{\mathsf n'}_{[\mathsf q'',p_k]}(n_{k+1}-p_k) \    \hbox{{\rm where} \ $\mathsf{n'} = (n_1,\dots, n_{k})$} \\ 
 &=& \frac{1}{2}(q_{k+1}-q_k)\left(b^{\mathsf n'}_{[\mathsf{q''},-1] }(n_{k+1}+1) +b^{\mathsf n'}_{[\mathsf{q''},1] }(n_{k+1}-1) \right). \end{eqnarray*} Therefore $a \e_k  = (b \ot {\bf 1})\e_k$ if and only if 
$b^{\mathsf n'}_{\mathsf q'}  =   \frac{1}{2} \left( a^{[\mathsf{n'},-q_k]}_{[\mathsf{q''}, q_k,-q_k]} - a^{[\mathsf{n'},q_{k}]}_{[\mathsf{q''}, -q_k,q_k]}\right)$  for all 
$\mathsf{n,q} \in \{-1,1\}^k$.   Setting $\mathsf{r}
= \mathsf{n'}$ and $\mathsf{s} = (s_1,\dots, s_k) = \mathsf{q'}$ gives the expression in  \eqref{eq:brs}.  
Now assume $a \in \Z_{k+1}$ and $b$ is the unique element in $\End(\VV^{\ot k})$ so that $a \e_k = (b \ot {\bf 1})\e_k$. 
Then $a \e_k\in \Z_{k+1}$ so that $a \e_k = ga \e_k g^{-1} $ for all $g \in \G$.   It follows that $(b \ot {\bf 1})\e_k = g(b \ot {\bf 1})\e_k g^{-1} = g(b \ot {\bf 1})g^{-1}\e_k 
= (gbg^{-1} \ot {\bf 1})\e_k$.   The uniqueness of $b$ forces $gbg^{-1} = b$ to hold for all $g \in \G$, so that $b \in \Z_k$
as claimed in (a).    

(b) Part (a) implies that  $\Z_{k+1}\e_k = \Z_k \e_k$ (where $\Z_k$ is identified with $\Z_k \ot {\bf 1}$).   A symmetric argument
gives $\e_k \Z_{k+1} =  \e_k \Z_k$, and it follows that  $\Z_k \e_k \Z_k = \Z_{k+1}\e_k \Z_{k+1}$ is  an ideal of $\Z_{k+1}$.  Part (c) 
is a consequence of the uniqueness of $b$ in the above proof.  
\end{proof}
 
 The Jones basic construction for $\Z_k \subseteq \Z_{k+1}$ is based on the ideal 
 $\Z_{k} \e_{k} \Z_{k}$  of $\Z_{k+1}$ and the fact that $\Lambda_{k-1}(\G) \subseteq \Lambda_{k+1}(\G)$,   and it involves the following two key ideas.   
\begin{enumerate}
\item[(1)]  When decomposing $\VV^{\ot (k+1)}$,  let  
\begin{eqnarray}  \VV^{\ot (k+1)}_\text{old} &=& \bigoplus_{\lambda \in  \Lambda_{k-1}(\G)} 
m_{k+1}^\lambda \, \G^\lambda \\
 \VV^{\ot (k+1)}_\text{new} &=& \bigoplus_{\lambda \in   \Lambda_{k+1}(\G)\setminus \Lambda_{k-1}(\G)}
m_{k+1}^\lambda \, \G^\lambda. \end{eqnarray} 
Thus,  $\VV^{\ot (k+1)} = \VV^{\ot (k+1)}_\text{old} \oplus \VV^{\ot (k+1)}_\text{new}$. Using the fact that  $\frac{1}{2}\e_k$ corresponds to the projection onto the trivial $\G$-module in the last two tensor slots of $\VV^{\otimes (k+1)}$,  Wenzl  (\cite[Prop.~4.10]{W3},\cite[Prop.~2.2]{W4}) proves that  $\Z_k \e_k \Z_k \cong \End_\G(\VV^{\otimes (k+1)}_\text{old})$. 
Applying the decomposition $\End_\G(\VV^{\otimes (k+1)}) \cong \End_\G(\VV^{\otimes (k+1)}_\text{old}) \oplus \End_\G(\VV^{\otimes (k+1)}_\text{new})$ then gives
\begin{equation}\label{JBC-decomp}
\Z_{k+1} \cong  \Z_k \e_k \Z_k \oplus \End_\G(\VV^{\otimes (k+1)}_\text{new}).
\end{equation}

\item[(2)] There is an algebra isomorphism $\Z_{k-1} \cong \e_k \Z_k \e_k$ via the map  that sends $a \in \Z_{k-1}$ to $\e_k a \e_k =  2 a \e_k = 2\e_k a$.  Viewing $\Z_k \e_k$ as a module for $\Z_k \e_k \Z_k$  and for $\Z_{k-1} \cong \e_k \Z_k \e_k$ by multiplication on the left and right, respectively, we have that 
 these actions commute and centralize one another:
\[
\Z_k \e_k \Z_k \cong \End_{\Z_{k-1}}(\Z_k \e_k) \qquad  \text{and} \qquad  \Z_{k-1}   \cong \End_{ \Z_k \e_k  \Z_k}(\Z_k \e_k).
\]
Double-centralizer theory (e.g.,  \cite[Secs.\ 3B and 68]{CR})  then implies that the simple summands  of the semisimple algebras $\Z_k \e_k \Z_k$ and $\Z_{k-1}$ (hence their irreducible modules)  can be indexed by the same set  $\Lambda_{k-1}(\G)$.  

\end{enumerate}

As before, let  $\Z^\lambda_k, \lambda \in \Lambda_k(\G)$, denote the irreducible $\Z_k$-modules.  By restriction, $\Z^\lambda_k$ is a $\Z_{k-1}$-module and 
\[
\Res^{\Z_k}_{\Z_{k-1}} (\Z^\lambda_k) = \bigoplus_{\mu \in  \Lambda_{k-1}(\G)} \Theta_{\lambda,\mu}\, \Z^\mu_{k-1},
\]
where $\Theta_{\lambda,\mu}$ is the multiplicity of $\Z^\mu_{k-1}$ in $\Z^\lambda_{k}$.   The $| \Lambda_k(\G)| \times |\Lambda_{k-1}(\G)|$ matrix $\Theta$ whose $(\lambda,\mu)$-entry is $\Theta_{\lambda,\mu}$ is the \emph{inclusion matrix} for $\Z_{k-1} \subseteq \Z_{k}$.  For all of the groups $\G$ in this paper, the restriction is ``multiplicity free" meaning that each $\Theta_{\lambda,\mu}$ is either 0 or 1.   

General facts from double-centralizer theory imply that  the inclusion matrix for $\Z_{k-1} \subseteq \Z_k$ is the transpose of the inclusion matrix for $\End_{\Z_{k}}(\Z_k \e_k )\subseteq \End_{\Z_{k-1}}( \Z_k \e_k) $,  which implies the following: 

\begin{quote} \emph{In the Bratteli diagram
for the tower of algebras $\Z_k$, the edges between levels $k$ and $k+1$ corresponding to $ \Z_k \e_k \Z_k \subseteq \Z_{k+1}$ are the reflection over level $k$ of the edges between $k-1$ and $k$ corresponding to $\Z_{k-1} \subseteq \Z_k$.}
\end{quote}
In Section \ref{subsec:Bratteli}, we have highlighted the edges of the Bratteli diagrams that are \emph{not} reflections over level $k$ and left unhighlighted the edges corresponding to the Jones basic construction.  
 
 \begin{quote} 
\emph{The highlighted edges give a copy of the representation graph $\mathcal{R}_\VV(\G)$ (i.e. the Dynkin diagram)  embedded in the Bratteli diagram,  except when $\G$ is a cyclic group of odd order, in which case they form the double of $\mathcal{R}_\VV(\G)$.}\end{quote}   
This will be discussed further in Examples \ref{ex:idems}. 

\subsection{Projection mappings}
The Jones-Wenzl idempotents (see \cite{W1}, \cite[Sec.\ 3]{Jo}, \cite{FK}) in $\TL_k(2)$  are defined recursively by setting  $\f_1 = {\bf 1}$ and letting 
\begin{equation}\label{JonesWenzl}
\f_{n}  = \f_{n-1} - \frac{n-1}{n} \f_{n-1} \e_{n-1} \f_{n-1}, \qquad 1 < n \le k.
\end{equation}
These idempotents satisfy the following properties (see \cite{W1}, \cite{FK}, \cite{CJ} for proofs),
\begin{equation}\label{JWrelations}
\begin{array}{lll}
{\rm (JW1)}\qquad\qquad & \f_n^2 =  \f_n, & 1 \le n \le k-1, \\
{\rm (JW2)} & \e_i  \f_n = \f_n \e_i = 0, \quad & 1 \le i < n \le k, \qquad    \\
{\rm (JW3)} & \e_i \f_n = \f_n \e_i, & 1 \le n < i  \le k -1, \\ 
{\rm (JW4)} & \e_n \f_n \e_n = \frac{n+1}{n} \f_{n-1} \e_n, \qquad & 1 \le n \le k-1,  \\ 
{\rm (JW5)} & 1-\f_n \in \langle \e_1, \ldots, \e_{n-1} \rangle,  \qquad  &  1 \le n \le k-1,  \\
{\rm (JW6)} & \f_m\f_n = \f_n \f_m  & 1 \le m,n \le k,
\end{array} \hskip1truein
\end{equation}
where $\langle \e_1, \ldots, \e_{n-1} \rangle$ stands for  the subalgebra of $\TL_k(2)$ generated by  $\e_1, \ldots, \e_{n-1}$.   An expression for  $\f_n$ in the $\TL_k(2)$ basis of  words in the generators $\e_1, \ldots, \e_{k-1}$ can be
found in \cite{FK,Mo}.

The simple $\SU_2$-module $\VV(k)$ appears in $\VV^{\otimes k}$ with multiplicity 1,  and it is does not appear as a simple summand of $\VV^{\otimes \ell}$ for any $\ell < k$.
To locate  $\VV(k)$ inside  $\VV^{\otimes k}$,  let $\mathsf{r} = (r_1, \dots, r_k) \in \{-1,1\}^k$ for some $k \geq 1$,  and
set 
\begin{equation}\mathsf{ \vert r \vert}  = \left |\left\{\,r_i \mid r_i = -1\,\right \}\right |. \end{equation} 
Then the totally   
 symmetric tensors $\mathsf{S}(\VV^{\otimes k})$ form the $(k+1)$-dimensional  subspace of $\VV^{\otimes k}$ spanned by the vectors $\w_0, \w_1, \ldots, \w_k$, where
\begin{equation}\label{weightvector}
\w_t = \sum_{|\mathsf{r}| = t}  \v_{\mathsf{r}}, \qquad 0 \le t \le k.
\end{equation}
It is well known (see for example \cite[Sec.~11.1]{FH}) that   
$\mathsf{S}(\VV^{\otimes k})\cong \VV(k)$ as an $\SU_2$-module, and that $\f_k(\VV^{\otimes k}) = \mathsf{S}(\VV^{\otimes k})$  (\cite[Prop.~1.3, Cor.~1.4]{FK}).  In particular, 
\[\mathsf{S}(\VV^{\ot 2}) =
\spann_{\mathbb C}\{\w_0 =  \v_1 \ot \v_1,  \, \w_1 = \v_{-1} \ot \v_1 + \v_1 \ot \v_{-1}, \, 
\w_2=\v_{-1} \ot \v_{-1}\} \cong \VV(2).\]
Observe that $\f_2 ={\bf 1} -\half\e_1$ and $\f_2(\w_t) = \w_t$ for $t = 0,1,2$.   

\begin{subsection}{Projections related to branch nodes} \end{subsection}
A  \emph{branch node} in  the representation graph $\mathcal{R}_\VV(\G)$ is any vertex of degree greater than 2.  Let $\br(\G)$ denote the branch node in  $\mathcal{R}_\VV(\G)$, and in the case of $\DD_n (n > 2)$, which has 2 branch nodes, set $\br(\DD_n) = 1$. In the special case of $\mathcal{R}_\VV(\C_n)$ for $n \leq \infty$,  we consider the affine node itself to be the branch node,  so that $\br(\C_n) = 0$.   When  $\G = \D_n,\TT,\OO,\II, \C_\infty$,  or $\D_\infty$,  we say that the \emph{diameter} of $\mathcal{R}_\VV(\G)$, denoted by $\diam(\G)$, is the maximum distance  between any vertex $\lambda \in \Lambda(\G)$ and $0 \in \Lambda(\G)$.  In particular,  $\diam(\G) = \infty$ for $\G = \C_\infty$ or $\D_\infty$,
 as in Table \ref{eq:branch}.   For $\G = \C_n$, we let  $\diam(\G) = \tilde n$,  where $\tilde n$ is as in \eqref{eq:branch2}. 
\begin{equation}\label{eq:branch}
\begin{array}[t]{c|cccccccc}
\G & \SU_2 & \C_n & \D_n & \TT & \OO & \II & \C_\infty & \D_\infty \\ \hline 
\diam(\G) & \infty & \tilde{n} & n & 4 & 6 & 7 & \infty & \infty  \\  
\br(\G) & \emptyset & 0 & 1 & 2 & 3 & 5 & 0 & 1\\  
\end{array}
\end{equation} 
\begin{equation}\label{eq:branch2}
\tilde n = \begin{cases}
\half n, & \text{if $n$ is even},\\
n, & \text{if $n$ is odd}.
\end{cases}
\end{equation}

In this section, we develop a recursive procedure for constructing  the idempotents $\f_\nu$  that project onto the irreducible $\G$-summands $\G^\nu$ of
$\VV^{\ot k}_{\mathsf{new}}$.   Assume  $k \leq \ell$, where  $\ell = \br(\G)$.   
Then $\VV^{\ot k}_{\mathsf{new}} = \G^{(k)} = \VV(k)$.  In this case,  the projection of $\VV^{\ot k}$ onto $\G^{(k)}$
is given by  $\f_{(k)} := \f_k$, where $\f_k$ is the Jones-Wenzl idempotent.    The irreducible $\SU_2$-module $\VV(\ell+1)$ is reducible as a $\G$-module.   Suppose $\mathsf{deg}(\ell)$  is the degree of the branch node indexed by $\ell$
in the representation graph of $\G$, so that  $\mathsf{deg}(\ell) = 3$,  except when $\G = \D_2$  where $\mathsf{deg}(\ell) = 4$.   We assume the  decomposition into irreducible $\G$-modules is given by   $\VV^{\ot (\ell+1)}_{\mathsf{new}} = \VV(\ell+1) = \bigoplus_{j = 1}^{\mathsf{deg}(\ell)-1}  \G^{\beta_j}$.    For example,
when $\G = \OO$,  then $\ell=3$  and  $\VV^{\ot 4}_{\mathsf{new}} =  \VV(4)  = \G^{(4^+)} \oplus \G^{(4^-)}$;  and when $\G = \D_2$,  then  $\ell = 1$  and $\VV^{\ot 2}_{\mathsf{new}} =  \VV(2) = \G^{(0')} \oplus \G^{(2)} \oplus \G^{(2')}$, where the labels of the irreducible $\G$-modules are as in Section \ref{subsec:Ddiag}.    

 Let 
\begin{equation}\label{eq:branchdec} 
\f_{\ell+1}=  \sum_{j=1}^{\mathsf{deg}(\ell) - 1}  \f_{\beta_j} 
\end{equation}
be the decomposition of the Jones-Wenzl idempotent $\f_{\ell+1}$  into minimal orthogonal idempotents that commute with $\G$ and project $\VV^{\ot (\ell+1)}_{\mathsf{new}}$ onto the irreducible $\G$-summands $\G^{\beta_j}$.  For finite subgroups $\G$, these
idempotents can be constructed using the corresponding irreducible characters $\chi_{\beta_j}$ as
 \begin{equation}\label{projectors}
\f_{\beta_j} = \frac{ \dimm \G^{\beta_j}}{|\G|}  \sum_{g \in \G}  \overbar{\chi_{\beta_j}(g)}\,  g^{\ot (\ell+1)},
\end{equation}
where $g^{\ot (\ell+1)}$ is the matrix of $g$ on $\VV^{\otimes (\ell+1)}$ and ``$\overbar{\ \ }$'' denotes complex conjugate. 
(See for example, \cite[(2.32)]{FH}.)   We will not need these explicit expressions in this paper. \msk  

\begin{lemma}\label{lem:fmu}Let $\Z_k = \Z_k(\G)$ for all $k$.   Let  $\lambda \in \Lambda_k(\G)$, and assume $\f_\lambda \in \Z_k$ projects $\VV^{\ot k}$ onto the
irreducible  $\G$-module $\G^\lambda$ in $\VV^{\ot k}_{\mathsf{new}}$.   Let  $ d^{\lambda} = 
\dimm \G^\lambda$,  and suppose $\G^\lambda \ot \VV = \bigoplus_i  \G^{\mu_i}$.   Let $\f_{\mu_i}$
be the orthogonal idempotents in $\Z_{k+1}$ that project $\VV^{\ot (k+1)}$ onto the irreducible $\G$-modules
$\G^{\mu_i}$, and assume $d^{\mu_i} = \dimm \G^{\mu_i}$.    Then the following hold for all $i$  such
that $\G^{\mu_i}$ is a summand of $\VV^{\ot (k+1)}_{\mathsf{new}}$; {\it i.e., for all $\mu_i \in \Lambda_{k+1}(\G) \setminus \Lambda_{k-1}(\G)$:}
\begin{itemize}
\item[{\rm (i)}]  $ \f_\lambda \f_{\mu_i} = \f_{\mu_i} =\f_{\mu_i} \f_\lambda$;
\item[{\rm (ii)}]  $\f_{\mu_i}$ commutes with $\e_j$ for $j > k+1$;
\item[{\rm (iii)}]  $\varepsilon_{k+1}(\f_{\mu_i}) = \frac{d^{\mu_i}}{2d^\lambda} \f_\lambda$\,;   and 
\item[{\rm (iv)}]  $\e_{k+1}\f_{\mu_i}\e_{k+1} = 2 \varepsilon_{k+1}(\f_{\mu_i}) \e_{k+1} = \frac{d^{\mu_i}}{d^\lambda} \f_\lambda \e_{k+1}$.
\end{itemize}  
\end{lemma} 
 \begin{proof} (i)  Now $\f_{\mu_i} \f_\lambda \left (\VV^{\ot (k+1)}\right)  =\f_{\mu_i} (\G^\lambda \ot \VV) = 
\G^{\mu_i} = \f_{\mu_i}(\VV^{\ot (k+1)})$, so $\f_{\mu_i} \f_\lambda = \f_{\mu_i}$.    For the product in the other order, 
we have  $\f_\lambda  \f_{\mu_i}\left (\VV^{\ot (k+1)}\right)  = \f_\lambda (\G^{\mu_i} )$.   Since
$\G^{\mu_i}$ is contained in $\G^\lambda \ot \VV$ and $\f_\lambda$ acts as the identity on that space, 
$\f_\lambda(\G^{\mu_i} ) = \G^{\mu_i} =  \f_{\mu_i}(\VV^{\ot (k+1)})$, which implies the result. 

(ii) This is clear, because  $\f_{\mu_i} \in \Z_{k+1}$,  and $\Z_{k+1}$ commutes with $\e_j$ for $j > k+1$.

(iii)  From (i) and part (c) of Proposition 1.26,  we have for each $i$, 
\begin{equation}\label{eq:fscommute} \f_\lambda \varepsilon_{k+1}(\f_{\mu_i})  = 
\varepsilon_{k+1}(\f_\lambda \f_{\mu_i}) = \varepsilon_{k+1}(\f_{\mu_i}) 
 = \varepsilon_{k+1}(\f_{\mu_i}\f_\lambda)  = \varepsilon_{k+1}(\f_{\mu_i}) \f_\lambda.\end{equation} 
 
 Suppose $\W = \VV^{\ot k}$,  and let $\W = \W_0 \oplus \W_1$ be the eigenspace decomposition of $\f_\lambda$
 on $\W$ so that $\f_\lambda w = j w$ for $w \in \W_j$, $j=0,1$.  Since $\f_\lambda$ and $\varepsilon_{k+1}(\f_{\mu_i})$ commute by \eqref{eq:fscommute},
$\varepsilon_{k+1}(\f_{\mu_i})$  maps $\W_j$ into itself for $j=0,1$.   Now if $w \in \W_0$, then
$\varepsilon_{k+1}(\f_{\mu_i})w =  \varepsilon_{k+1}(\f_{\mu_i}) \f_\lambda w = 0$,  so that $\varepsilon_{k+1}(\f_{\mu_i})$
is $0$ on $\W_0$.       Since $\varepsilon_{k+1}(\f_{\mu_i}) \in \Z_k$,  we have
$\varepsilon_{k+1}(\f_{\mu_i})\in \End_\G(\W_1) = \CC \id_{\W_1} = \CC \f_\lambda$ by Schur's lemma, as $\W_1 =
\G^{\lambda}$, an irreducible $\G$-module.  Therefore, there exists $\xi_i \in \CC$ such that
$\varepsilon_{k+1}(\f_{\mu_i}) = \xi_i \f_\lambda$ on $\W_1$.   But since these transformations agree on $\W_0$
(as both equal 0 on $\W_0$),   we have  $\varepsilon_{k+1}(\f_{\mu_i}) = \xi_i \f_\lambda$ on $\VV^{\ot k}$. 
Taking traces and using (d) of Proposition 1.26 gives
\[ d^{\mu_i} = \tr_{k+1}(\f_{\mu_i}) =  \tr_{k+1}\left(\varepsilon_{k+1}(\f_{\mu_i})\right) = \xi_i \tr_{k+1}(\f_\lambda) = 
2 \xi_i \tr_k(\f_\lambda) = 2\xi_i  d^\lambda.\]
Therefore $\xi_i =  \frac{d^{\mu_i}}{2 d^\lambda}$ so that 
$\varepsilon_{k+1}(\f_{\mu_i}) =   \frac{d^{\mu_i}}{2 d^\lambda}\f_\lambda$, as asserted in (iii). 
Part (iv) follows immediately from (iii) and Proposition 1.26\,(a).   \end{proof} 

\begin{prop}\label{prop:fnu}  Assume the notation of  Lemma \ref{lem:fmu},  and let  $\mu = \mu_i  \in \Lambda_{k+1}(\G)\setminus \Lambda_{k-1}(\G)$ for some $i$.   Suppose 
$\G^{\mu} \ot \VV = \G^\lambda \oplus \G^\nu$,   where $\G^\nu$ is an irreducible $\G$-module of
dimension $d^\nu$ and $\nu \in \Lambda_{k+2}(\G)\setminus \Lambda_k(\G)$.  
Set
\[\f_\nu =  \f_{\mu} -  \frac{d^\lambda}{d^\mu} \f_\mu \e_{k+1} \f_\mu.\]
Then the following hold:
\begin{itemize}
\item[{\rm (i)}]  $\f_\nu$ is an idempotent in $\Z_{k+2}$;
\item[{\rm (ii)}]  $\f_\mu \f_\nu = \f_\nu = \f_\nu \f_\mu$; 
\item[{\rm(iii)}]   $\f_\nu$ commutes with $\e_j$ for $j > k+2$;
\item[{\rm (iv)}]  $\e_{k+2} \f_\nu \e_{k+2}  = 
 \frac{d^\nu}{d^\mu} \f_\mu \e_{k+2}$, so that  $\varepsilon_{k+2}(\f_\nu) =   \frac{d^\nu}{2d^\mu}\f_\mu$.
\item[{\rm (v)}]  $\f_\nu$ projects $\VV^{\ot(k+2)}$ onto $\G^\nu$.   \end{itemize}
\end{prop}

\begin{proof} (i) We have 
\begin{eqnarray*}  \f_\nu^2 &=& \left( \f_{\mu} -  {\frac{d^\lambda}{d^\mu}} \f_\mu \e_{k+1} \f_\mu\right)^2 
= \f_\mu - { 2 \frac{d^\lambda}{d^\mu}} \f_\mu \e_{k+1} \f_\mu +   {\frac{(d^\lambda)^2}{(d^\mu)^2}} \f_\mu \e_{k+1} \f_\mu   \e_{k+1} \f_\mu\\
&=& \f_\mu - { 2 \frac{d^\lambda}{d^\mu}} \f_\mu \e_{k+1} \f_\mu + 
 { \frac{(d^\lambda)^2}{(d^\mu)^2}\frac{d^\mu}{d^\lambda}} 
 \f_\mu \f_\lambda  \e_{k+1} \f_\mu  \qquad  \hbox{\rm by (iv) of Lemma \ref{lem:fmu}} \\
&=& \f_\mu - { 2 \frac{d^\lambda}{d^\mu}} \f_\mu \e_{k+1} \f_\mu +   { \frac{d^\lambda}{d^\mu}}
 \f_\mu \e_{k+1}\f_\mu  \qquad \qquad \quad \   \hbox{\rm by (i) of Lemma \ref{lem:fmu}}   \\
&=&\f_{\mu} -  {\frac{d^\lambda}{d^\mu}} \f_\mu \e_{k+1} \f_\mu = \f_\nu.  \end{eqnarray*}    
By construction, $\f_\nu \in\Z_{k+2}$.  

Part (ii) follows easily from the definition of $\f_\nu$ and the fact that $\f_\mu$ is an idempotent.  

(iii) Since $\f_\mu \in \Z_{k+1}$ commutes with $\e_j$ for $j > k+1$,  and $\e_{k+1}$ commutes with $\e_j$ for $j > k+2$, 
$\f_\nu$ commutes with $\e_j$ for $j > k+2$.

For (iv) we compute 
 \begin{eqnarray*} \e_{k+2} \f_\nu \e_{k+2}   &=& \e_{k+2} \left(\f_\mu-  { \frac{d^{\lambda}}{d^{\mu}}}\f_\mu \e_{k+1} \f_\mu\right )\e_{k+2} \\
 &=&\e_{k+2}^2 \f_\mu -   { \frac{d^{\lambda}}{d^{\mu}}}\f_\mu \e_{k+2}\e_{k+1}\e_{k+2}\f_\mu  \qquad  \qquad  \hbox{\rm using (ii) of Lemma \ref{lem:fmu}}  \\
 &=& 2\e_{k+2}\f_\mu  -  { \frac{d^{\lambda}}{d^{\mu}}}\f_\mu \e_{k+2} \f_\mu  \\
 &=&\left(2 -  { \frac{d^{\lambda}}{d^{\mu}}}\right) \f_\mu \e_{k+2} =   { \frac{2d^\mu - d^{\lambda}}{d^\mu}}\f_\mu \e_{k+2}
=  { \frac{d^{\nu}}{d^\mu}}\f_\mu \e_{k+2}.  \end{eqnarray*}  
 
 This  equation along with Proposition 1.26\,(a)  implies that $\varepsilon_{k+2}(\f_\mu)
 =  { \frac{d^{\nu}}{2d^\mu}}\f_\mu  \in \Z_{k+1}$.  

(v) From \eqref{JBC-decomp} with $k+2$ instead of $k+1$,  we have
$\Z_{k+2} = \Z_{k+1}\e_{k+1}\Z_{k+1} \oplus \End_\G(\VV_{\mathsf{new}}^{\ot (k+2)})$. 
Now
\begin{equation}\label{eq:decomp} 
\f_\mu \ot \mathbf{1} = \f_\mu  =  { \frac{d^\lambda}{d^\mu}} \f_\mu \e_{k+1} \f_u   + \f_\nu  \in  \Z_{k+2},
\end{equation}  
and the idempotent
$\f_\mu \ot {\bf 1}$ projects  $\VV^{\ot (k+2)}$ onto $\G^\mu \ot \VV = \G^\lambda \oplus \G^\nu$.   Observe that 
$\p: =   { \frac{d^\lambda}{d^\mu}} \f_\mu  \e_{k+1} \f_\mu \in \Z_{k+1}\e_{k+1}\Z_{k+1}$, 
and 
\[\p^2 = \left( { \frac{d^\lambda}{d^\mu}} \f_\mu  \e_{k+1} \f_\mu \right)^2 =  { \frac{(d^\lambda)^2}{(d^\mu)^2}}
\f_\mu \e_{k+1} \f_\mu \e_{k+1} \f_\mu =   { \frac{(d^\lambda)^2}{(d^\mu)^2}   \frac{d^\mu}{d^\lambda}}\f_\mu \f_\lambda \e_{k+1}\f_\mu 
= { \frac{d^\lambda}{d^\mu}}  \f_\mu  \e_{k+1} \f_\mu = \p\]
using  Lemma \ref{lem:fmu} (iv), so we can conclude that $\p$  is an idempotent once we know it
is nonzero.  But if $\p = 0$,   then   
\[0 = 2\e_{k+1} \f_\mu \e_{k+1} \f_\mu \e_{k+1} = (\e_{k+1}\f_\mu \e_{k+1})^2 \\
=
\left({\frac{d^\mu}{d^\lambda}} \f_\lambda e_{k+1}\right)^2 = 2{\frac{(d^\mu)^2}{(d^\lambda)^2}} \f_\lambda  \e_{k+1}\]
by Lemma \ref{lem:fmu}\,(iv).  Since $\f_\lambda$ and $\e_{k+1}$ act on different tensor slots and both are nonzero, we have
reached a contradiction.   Thus,  $\p$ is an idempotent.      Moreover,
$\f_\nu\p  = (\f_\mu - \p)\p =  0$.    Therefore,   \eqref{eq:decomp}  gives the decomposition of
$\f_\mu \ot {\bf 1}$  into orthogonal idempotents, with the first  idempotent  
$\p =  { \frac{d^\lambda}{d^\mu}} \f_\mu  \e_{k+1} \f_\mu$ in $\Z_{k+1}\e_{k+1}\Z_{k+1} = \End_{\G}(\VV_{\mathsf{old}}^{\ot (k+2)}).$
Then $\G^\mu \ot \VV = \p \left( \VV^{\ot (k+2)} \right) \oplus \f_\nu \left( \VV^{\ot (k+2)}\right)$ is a decomposition of
$\G^\mu \ot \VV = \G^\lambda \oplus \G^\nu$ into $\G$-submodules such that   
$\p \left( \VV^{\ot (k+2)}\right) \in \VV_{\mathsf{old}}^{\ot (k+2)}$.    Hence,  $\p \left( \VV^{\ot (k+2)}\right) = \G^\lambda$
and $\f_\nu \left( \VV^{\ot (k+2)}\right) = \G^\nu$.  
\end{proof}  

This procedure can be applied recursively to produce the projection idempotents  for all $k \leq \diam(\G)$ that do not come from  \eqref{projectors}, as illustrated  in the next series of examples.
 Our labeling of the irreducible $\G$-modules is as in Section \ref{subsec:Ddiag}.

\begin{examples}
\label{ProjectionExample} \begin{itemize}
\smallskip
\item[$\bullet$] \emph{The $\G = \TT,\OO, \II$ cases:}  { Let $\ell = \br(\G)$ (so $\ell  = 2,3,5$, respectively)  and set  $\f_{(k)} = \f_k$ for $0 \leq k \leq \ell$, where $\f_k$} is given by (\ref{JonesWenzl}). 
Let  $\f_{((\ell+1)^+)}$ and $\f_{((\ell+1)^-)}$ be 
the projections onto $\G^{((\ell+1)^+)}$ and $\G^{((\ell+1)^-)}$, respectively, which can be constructed using \eqref{projectors}.  When $\G = \TT$, applying Proposition
\ref{prop:fnu} with $\lambda = (2)$, $\mu = (3^\pm)$, and $\nu = (4^\pm)$ produces  the idempotents
$\f_{(4^\pm)} = \f_{(3^{\pm})} - \frac{3}{2} \f_{(3^{\pm})} \e_3  \f_{(3^{\pm})}$  that project $\VV^{\ot 4}$ onto
$\G^{(4^\pm)}$.    
When  $\G = \OO$,  first taking  $\lambda = (3)$, $\mu = (4^+)$, $\nu = (5)$  and then taking
$\lambda =  (4^+)$, $\mu = (5)$, $\nu = (6)$ in Proposition \ref{prop:fnu} will construct  the two remaining idempotents $\f_{(5)}$ and $\f_{(6)}$. 
Similarly, for $\G = \II$,  applying the procedure to   $\lambda = (5)$, $\mu = (6^+)$, and $\nu = (7)$, will produce
the last idempotent $\f_{(7)}$. 
\smallskip
\item[$\bullet$] \emph{The $\G = \C_n$, $n \leq \infty$  case:}  
Let  $\f_{(1)}$ and $\f_{(-1)}$  project $\VV$ 
onto the one-dimensional modules $\G^{(1)}$ and $\G^{(-1)}$, respectively.     Applying Proposition
\ref{prop:fnu} with $\lambda = (0)$ (i.e. with $\f_{(0)} = {\bf 1}$), $\mu = (\pm 1)$, and $\nu = (\pm 2)$ begins the recursive process
and constructs $\f_{(\pm 2)}$. 
 (We are adopting the conventions that  $(+j)$ stands for $(j)$ 
and $\G^{(j)} \cong \G^{(i)}$ whenever $n < \infty$ and $j \equiv i \modd n$.) Then assuming we have constructed $\f_{(\pm j)}$ for all $1\leq j \leq k$,  we obtain from the proposition that  
$\f_{(\pm(k+1))} = \f_{(\pm k)} - \f_{(\pm k)} \e_k \f_{(\pm k)}$  projects  $\VV_{\mathsf{new}}^{\ot (k+1)}$ onto
$\G^{(\pm(k+1))}$.    In the case that $\G = \C_\infty$, iterations of this process produce the idempotent projections onto
$\VV_{\mathsf{new}}^{\ot (\pm (k+1))}$ for all $k \geq 1$.  \smallskip

When $n < \infty$, the diameter is $\tilde n$, and   we proceed as above  to construct the idempotents $\f_{(\pm k)}$ for $k \leq \tilde n$.   Now 
$\VV_{\mathsf{new}}^{\ot \tilde n}  = \G^{(\tilde n)} \oplus \G^{(\tilde n)}$ when $n$ is even,  as $-\tilde n \equiv \tilde n \modd n$; 
and $\VV_{\mathsf{new}}^{\ot \tilde n}  = \G^{(0)} \oplus \G^{(0)}$ when $n$ is odd,  as $ \tilde n = n$.  The idempotent $\f_{(\pm \tilde n)}$ projects onto the  space $\CC \v_{\pm \underline{\mathsf 1}}$,   where $\underline{\mathsf 1}$ is the $\tilde n$ tuple of all 1s, and 
$
\v_{\pm \underline{\mathsf 1}}
$
has $\tilde n$ tensor factors  equal to $\v_{\pm {\mathsf 1}}$.  
 In the centralizer algebra $\Z_{\tilde n}$  there is 
a corresponding  $2 \times 2$ matrix block.   The idempotents $\f_{( \tilde n)}$, $\f_{(- \tilde n)}$ act as the diagonal matrix units
$\mathsf{E_{\underline{\mathsf 1},\underline{\mathsf 1}}}$, $\mathsf{E_{-\underline{\mathsf 1},-\underline{\mathsf 1}}}$, respectively.   The remaining basis elements of the matrix block are the matrix units
$\mathsf{E_{\underline{\mathsf 1},-\underline{\mathsf 1}}}$, $\mathsf{E_{-\underline{\mathsf 1},\underline{\mathsf 1}}}$.
\smallskip
\item[$\bullet$] \emph{The $\G = \D_n$, $2 \leq n \leq \infty$ case:}  Suppose  first that $n \ge 3$. The symmetric tensors in $\VV^{\ot 2}$ 
are reducible in the $\D_n$-case  and decompose into a direct sum of  the  one-dimensional $\G$-module $\G^{(0')}$ and the two-dimensional irreducible $\G$-module  $\G^{(2)}$.  Let    $\f_{(0')}$ and $\f_{(2)}$ denote the corresponding projections.  
Note that $\f_{(0')} + \f_{(2)} = \f_2 = {\bf 1} - \half \e_1$.   
Starting with $\lambda = (1)$ (so $\f_{(1)} = {\bf 1}$), $\mu = (2)$ and $\nu = (3)$, and applying  the recursive
procedure, we obtain the idempotents $\f_{(k)} = \f_{(k-1)} - \f_{(k-1)}\e_{k-1}\f_{(k-1)}$ for all $k= 3,4, \dots$ in the $\D_\infty$-case, and  for $3 \leq k \leq n-1$ in the $\D_n$-case.   Now $\G^{(n-1)} \ot \VV = \G^{(n-2)} \oplus \G^{(n)} \oplus \G^{(n')}$, where
$ \G^{(n)}$ and $\G^{(n')}$ are   one-dimensional $\G$-modules.  Denoting the projections onto them by 
$\f_{(n)}$ and $\f_{(n')}$   (they can be constructed using \eqref{projectors}), we have    $\f_{(n-2)} +\f_{(n)}+\f_{(n')} = \f_{(n-1)} \ot {\bf 1}$.
Thus,  $\Z_n = \langle \Z_{n-1}, \e_n, \f_{(n)} \rangle$.  

\quad Now when  $\G = \DD_2$,  then $\br(\G) = 1$, and the branch node $\br(\G)$   has degree 4 in the
corresponding Dynkin diagram.      In this case  \[1-\half \e_1 = \f_2 = \f_{(0')} + \f_{(2')} + \f_{(2)},\] 
where the 3 summands on the right are mutually orthogonal
idempotents giving the projections onto the  one-dimensional irreducible $\G$-modules
$\G^{(0')}$, $\G^{(2')}$, and $\G^{(2)}$, respectively.   Thus  $\Z_2  = \CC {\bf 1} \oplus \CC \e_1 \oplus \CC \f_{(0')} \oplus \CC \f_{(2')}$.

\quad  Note that when $2 \leq n < \infty$, and $\mu = (n)$ or $(n')$ (or $(0')$ when $n=2$),   then 
 by Lemma \ref{lem:fmu},   
 \[\varepsilon_n(\f_\mu) =   \frac{1}{4}\f_{(n-1)}, \qquad 
\e_n \f_\mu \e_n  =  2 \varepsilon_n(\f_\mu) \e_n  =    \half \f_{(n-1)} \e_n,\]
(where $\f_{(n-1)} = \f_{(1)} = \f_1 = {\bf 1}$  when $n =2$).  

\end{itemize}
\end{examples} 

\begin{thm}\label{thm:JBC}   Let $\G$,  $\br(\G)$, and $\diam(\G)$ be as in {\rm (\ref{eq:branch})}, and let  $\Z_k = \Z_k(\G)$.  Then $\Z_1 = \CC{\bf 1} \cong \Z_0$.   Moreover, 
\begin{enumerate}
 \item[{\rm (a)}]  if $1 \le k < \diam(\G)$, and $k \not= \tilde n -1$ in the case $\G = \C_n$,  then  
 $\Z_{k+1} =  \Z_k  \e_k \Z_k \oplus \End_\G(\VV^{\otimes k}_\text{\rm new})$, where $\End_\G(\VV^{\otimes k}_\text{\rm new})$ is a commutative subalgebra of dimension equal to the number of nodes in $\mathcal{R}_\VV(\G)$ a distance $k$ from the trivial node;
 \item[{\rm (b)}]  if  $k \ge\diam(\G)$,  then $\Z_{k+1} = \Z_k \e_k \Z_k$;
 \item[{\rm (c)}]  if  $1 \le k  < \diam(\G)$, $k \neq \br(\G)$, and $k \neq n-1$ in the case $\G = \DD_n$, then $\Z_{k+1} = \langle \Z_k, \e_k \rangle$;   
\item[{\rm (d)}]    if $k = \br(\G)$ and $\G \neq \DD_2$,  then  $\Z_{k+1} = \langle \Z_k, \e_k, \f_{\mu} \rangle$, where $\mu$ is either of the two elements in $\Lambda_{k + 1}(\G) \setminus \Lambda_{k -1}(\G)$,  and $\f_\mu$ is the projection of $\VV^{\otimes (k+1)}_{\mathsf{new}}$ onto $\G^\mu$,
\item[{\rm (e)}]  if $\G = \C_n$ {\rm ($n < \infty$)},   then $\Z_{\tilde n} = \langle \Z_{\tilde n-1}, \e_{\tilde n -1}, \mathsf{E_{p,q}}, \hbox{\rm for} \ \mathsf{p,q} \in \{-\underline{\mathsf{1}}, \underline{\mathsf{1}}\}\rangle$, where $\mathsf{E_{p,q}}$ is the matrix unit in Examples \ref{ProjectionExample}.

\item[{\rm (f)}]   if  $\G = \DD_n$  {\rm ($2 < n < \infty$)}    then $\Z_n = \langle \Z_{n-1}, \e_{n-1}, \f_{\mu}\rangle$, where $\mu\in \{(n), (n')\}$

\item[{\rm (g)}]   
if $\G = \DD_2$,  then  $\Z_2 = \langle \Z_1, \e_1, \f_{\mu_1}, \f_{\mu_2}  \rangle$ where $\mu_1,  \mu_2 \in \{ (0'), (2), (2')\}$, $\mu_1 \not= \mu_2. $
\end{enumerate}
\end{thm} 
\begin{proof} (a)  Since $k < \diam(\G)$,  $\VV^{\otimes (k+1)}_\text{new}$ is a direct sum of irreducible $\G$-modules each with multiplicity 1 (except for the case where $\G = \C_n$ and $k  =\tilde n -1$ which is handled in part (e)), and the centralizer $\End_\G(\VV^{\otimes (k+1)}_\text{new})$ is commutative and spanned by central idempotents which project onto the irreducible summands.  Thus,  the dimension equals the number of new modules that appear at level $k+1$.

(b) When  $k \ge \diam(\G)$,   
$\VV^{\otimes (k+1)}_\text{new} =0$, and $\Z_{k+1} =  \Z_k \e_k \Z_k$ follows from  \eqref{JBC-decomp}.  

\tikzstyle{rep}=[circle,   thick,  minimum size=.40cm,inner sep=0pt,draw= black,     fill=black!20] 
\tikzstyle{norep}=[circle,    thick,minimum size=.04cm,inner sep=0pt,draw= white,  fill=white]                                          
(c)  If  $ 
\begin{array}{c}                                   
\begin{tikzpicture}[>=latex,text height=1.5ex,text depth=0.25ex]
  \matrix[row sep=1cm,column sep=.9cm] {
   \node (V1) [rep] {$\lambda$};   &   
   \node (V2) [rep] {$\mu$};  &
   \node  (V3) [rep] {$\nu$}; 
   \\
        };    
    \path
        (V1) edge[thick] (V2)	
        (V2) edge[thick] (V3)	
	;           
\end{tikzpicture}\end{array}
$
represents the neighborhood of a node $\mu$ in $\mathcal{R}_\VV(\G)$  with $\mu$ a distance $k$ from the trivial node and $\deg(\mu) = 2$, then $\G^{(\nu)}$ has multiplicity 1 in $\VV^{\otimes (k+1)}$.   
By Proposition \ref{prop:fnu}, the central projection from $\VV^{\otimes (k+1)}$ onto $\G^{\nu}$ is given by 
$\f_\nu = \f_\mu -    \frac{d^{\lambda}}{d^\mu} \, \f_\mu \e_{k} \f_\mu$,
where $\f_\mu \in \Z_k$ projects $\VV^{\otimes k}$ to $\G^{\mu}$.  Thus  $\f_{\nu} \in \langle \Z_{k}, \e_{k} \rangle$, and the set of $\f_\nu$, { with} $\nu$ a distance $k+1$ from the trivial node in  $\mathcal{R}_\VV(\G)$,  generate  $\End_\G(\VV^{\otimes (k+1)}_\text{new})$. 
The result then follows from part (a).

(d)  if $k = \br(\G)$ and $\G \neq \DD_2$,  then $\VV^{\otimes (k+1)}_\text{new} \cong \VV(k+1) \cong  \G^{\beta_1} \oplus \G^{\beta_2}$ where $\{\beta_1, \beta_2\} = \Lambda_{k + 1}(\G) \setminus \Lambda_{k -1}(\G)$. The Jones-Wenzl idempotent decomposes as  $\f_{k+1} = \f_{\beta_1} +  \f_{\beta_2}$  as in \eqref{eq:branchdec} (where the  $\f_{\beta_j}$ can be constructed as in \eqref{projectors}).  We know  from \eqref{JWrelations} that ${\bf 1} - \f_{k+1} \in \langle \e_1, \ldots, \e_{k} \rangle \subseteq \Z_k \e_k \Z_k,$ and we have  $\f_{\beta_1} + \f_{\beta_2} =  {\bf 1} - ({\bf 1} - \f_{k+1}) $. Thus,  $\Z_{k+1}$ is generated by $\Z_k, \e_k,  \f_{\beta_j}$ for $j = 1$ or $j=2$. 

(e)  If $\G = \C_n$ and $k  =\tilde n -1$, then $\G^{(\tilde n)}$ has multiplicity 2 in $\VV^{\otimes \tilde n}$, where $\G^{(\tilde n)} := \G^{(0)}$ if $n$ is odd. In this case,  $\End_\G(\VV^{\otimes \tilde n}_\text{new})$  is 4-dimensional with a basis of matrix units as  in  Examples \ref{ProjectionExample}.

 (f)   If  $\G = \D_n$ with $2 <  n < \infty$, then 
 $\VV^{\otimes n}_{\text{new}}  \cong \G^{(n)} \oplus \G^{(n')}$, and we let $\f_{(n)}$ and $\f_{(n')}$ project $\VV_{\text{new}} ^{\otimes n}$ onto $\G^{(n)}$ and $\G^{(n')}$, respectively.  As in part (d),  these minimal central idempotents can be constructed using \eqref{projectors}; the only difference in this case is that $\f_{(n)} + \f_{(n')}$ does not equal the Jones-Wenzl idempotent $\f_n$; however,  $\f_{(n-1)} = \f_{(n-2)} + \f_{(n)} + \f_{(n')}$ holds.

 (g)  When $\G = \DD_2$,  $\VV^{\otimes 2}_\text{new} \cong \G^{(0')} \oplus \G^{(2)}\oplus \G^{(2')}$ and the corresponding central idempotents $\f_{(0')}, \f_{(2)}, \f_{(2')}$ can be constructed as in \eqref{projectors}.  Furthermore $\f_{(0')}+\f_{(2)} + \f_{(2')} =
\f_2 = 1-\half \e_1$, so  $\Z_2$ is generated by  $\e_1$, $\Z_1$, and  any two of $\f_{(0')}, \f_{(2)}, \f_{(2')}$.
\end{proof}

\newpage
 
 \begin{examples}\label{ex:idems} 
 For all $\G \neq \C_n$ for $2 \leq n \leq \infty$, we have $\Z_1 = \CC {\bf 1} \cong \Z_0$. 
  \begin{itemize}
  \smallskip
 \item[$\bullet$]
If   $\G = \OO$, then from Theorem \ref{thm:JBC}  we deduce the following:  $\Z_2 = \Z_1\e_1\Z_1 + \CC \f_2 = \CC \e_1 + \CC \f_2 \cong \TL_2(2)$, where $\f_2 = 1 - \half \e_1$;
\ $\Z_3 =  \Z_2\e_2\Z_2 + \CC \f_3 \cong \TL_3(2)$ where $\f_3 = \f_2 -\frac{2}{3}\f_2 \e_2 \f_2$; \  $\Z_4 = \langle \Z_3, \e_3, \f_{(4^+)} \rangle  
= \langle \Z_3, \e_3, \f_{(4^-)} \rangle$ where $\f_{(4^+)} + \f_{(4^-)} = \f_4 = \f_3 - \frac{3}{4}\f_3 \e_3 \f_3$;  $\Z_5 = \langle \Z_4, \e_4\rangle$; \ $\Z_6 = \langle \Z_5, \e_5\rangle$; 
and $\Z_{k+1} = \Z_k \e_k \Z_k$ for all $k \geq 7$.     When $2 \leq k \leq \diam(\OO)=6$, there is exactly one new
idempotent added each time, except for $k=4$,  where the two idempotents $\half \e_3$ and $\f_{(4^+)}$ 
must be adjoined to $\Z_3$ to get $\Z_4$.   Each added idempotent corresponds to a highlighted edge in the Bratteli diagram of $\OO$    
(see Section \ref{subsec:Bratteli}).   The highlighted edges together with the nodes attached to them give the Dynkin diagram of
$\hat {\mathsf{E}}_7$.   (In fact, for all groups $\G$ except for the cyclic groups of odd order, the added idempotents give the corresponding Dynkin diagram as a subgraph of the Bratteli diagram.)
\smallskip
\item[$\bullet$]
If $\G = \DD_n$ for $n > 2$.   
 Then $\Z_2 = \langle \Z_1, \e_1, \f_{(0')} \rangle =  \langle \Z_1, \e_1, \f_{(2)}\rangle$ where $\f_{(0')} + \f_{(2)} = \f_2 = 1-\half \e_1$; \ $\Z_{k+1} = \langle \Z_k, \e_k \rangle$ for $2 \leq k \leq n-1$; \  $\Z_n =  \langle \Z_{n-1}, \e_{n-1},\f_{(n)}, \f_{(n')} \rangle$; \
$\Z_{k+1} = \Z_k \e_k \Z_k$ for all $k \geq n$.  (In particular when $n = \infty$,  $\Z_{k+1} = \langle \Z_k, \e_k \rangle$ for all $k \geq 2$.)   
\end{itemize} \end{examples} 
  
\subsection{Relations }
Recall that $\Z_k(\G) \supseteq \TL_k(2)$ for all $k \geq 0$  and that $\TL_k(2)$ has  
generators $\e_i \ (1 \leq i < k)$,  which satisfy the following relations from \eqref{eq:TLrels}: 
\begin{enumerate}
\item[{}]\begin{enumerate}
\item[{\rm (a)}] $\e_i ^2 = 2 \e_i, \qquad  \e_i \e_{i \pm 1} \e_i = \e_i$,   \qquad  \hbox{\rm and}  \qquad   $\e_i \e_j = \e_j \e_i, \ \ \hbox{\rm for}\ \  | i-j |>1.$
\end{enumerate}  \end{enumerate}
In the next two results, we identify additional generators needed to generate $\Z_k(\G)$  and the relations they satisfy.   In most instances,
these are not minimal sets of generators (as is evident from Theorem \ref{thm:JBC}), but rather the generators are chosen because they satisfy some
reasonably nice relations.

 \begin{prop}\label{prop:rels}  Let $\Z_k = \Z_k(\G)$ for all $k \ge 0$,  $\ell=\br(\G)$,  and $\diam(\G) = \ell + m$.   Suppose   $\nu_0 = (\ell), \nu_1, \dots, \nu_{m}$ is  a  sequence of distinct nodes from the branch node  $(\ell)$  to the node $\nu_m$, which is  a distance $\diam(\G) = \ell+m$ from $0$.   Set  $\b_j: = \f_{\nu_j}$ {\rm (}the projection of $\VV_{\mathsf{new}}^{\ot ( \ell+j)}$ onto $\G^{\nu_j}${\rm )},  and let $d_{j} = d^{\nu_j} = \dimm \G^{\nu_j}$   for $j = 0,1,\dots, m$.     
 Then the following hold:
 \begin{itemize} \item[{\rm (i)}]  If $k \leq \ell$,  then $\Z_k= \TL_k(2)$, and $\Z_k$ has generators
 $\e_i \ (1 \leq i < k)$ which satisfy {\rm (a)}. 
\item[{\rm (ii)}] If  $\ell <  k\leq \ell+m = \diam(\G)$, and $k \neq \diam(\G)$ for $\G = \C_n, \D_n$, $n < \infty$,   then  $\Z_k$ has generators  $\e_i \ (1 \leq i < k)$ and $ \b_j \  {\rm (}1 \leq j \leq k-\ell${\rm )} which  satisfy  {\rm (a)} and 
 \begin{enumerate} 
 \item[{\rm (b)}]  $\b_i \b_j = \b_j = \b_j \b_i,$  \quad  for all  $0 \leq i \leq j \leq   k - \ell$;  
\item[{\rm (c)}]    $\b_{j+1} = \b_j - \frac{d_{j-1}}{d_j} \b_j \e_{\ell+j} \b_j,$ for all $j =1,\dots,  k - \ell -1;$
 \item [{\rm (d)}]  $\e_i \b_j = 0 = \b_j \e_i$, \ \   for all   $1 \leq i < \ell+j$, \ and  \ $\e_i \b_j = \b_j \e_i,$ for
all  $ \ell +j < i \leq   k$;    
\item [{\rm (e)}]   $\e_{\ell+j} \b_j \e_{\ell+j}   = 
 \frac{d_j}{d_{j-1}} \b_{j-1}\e_{\ell+j}$,     for all $j = 1,\dots, k - \ell$. 
\end{enumerate}  
\item[{\rm (iii)}] If  $k > \diam(\G)$, and $\G \neq \C_n, \D_n$ for $n < \infty$, then $\Z_k$ has generators $\e_i$ $(1 \leq i < k)$ and
 $\b_j$ $(1 \leq j \leq m)$
which satisfy {\rm (a)-(e)} and  
\begin{enumerate}
\item[{\rm (f)}]  $\b_{m}  = \frac{d_{m-1}}{d_{m}}  \b_{m}  \e_{\ell +m} \b_{m}$.
\end{enumerate}
\end{itemize} 
 \end{prop}

\begin{proof}   For (i) we have  $\TL_k(2) \subseteq \Z_k$ for all $k \ge 0$, and thus the $\e_i$ satisfy the relations in (a) by \eqref{eq:TLrels}.  The equality $\TL_k(2) = \Z_k$ is proved by comparing dimensions for $k \le \ell.$   
Part (ii) follows from  Theorem \ref{thm:JBC},  Lemma \ref{lem:fmu}, and
Proposition \ref{prop:fnu}.  
(Note that $\ell = 0$ for $\C_n,$  and, in the notation of Examples \ref{ProjectionExample},  we may take $\b_j = \f_{(+j)}$ for all $j = 1, \dots, \tilde{n}-1$   or take $\b_j = \f_{(-j)}$ for all such $j$, since $\f_{(\pm (j+1))} = \f_{(\pm j)} - \f_{(\pm j)}\e_j\f_{(\pm j)}$ and $\f_{(\mp 1)} = \bf 1 - \f_{(\pm 1)}$.)   
For (iii),  we have from Theorem \ref{thm:JBC} that   $\Z_{k} = \Z_{k-1} \e_{k-1} \Z_{k-1} = \End_{\G}(\VV_{\mathsf{old}}^{\ot k})$ for all $k> \diam(\G)= \ell +m$.     Thus, $\e_1,\dots,\e_{k-1}$, $\b_1, \dots, \b_m$ generate $\Z_k$
and satisfy (a)-(e).     To show that (f) holds, consider  $\G^{\nu_m} \ot \VV  \cong  \G^{\nu_{m-1}}$.  We know  
$ \frac{d_{m-1}}{d_{m}}  \b_{m}  \e_{\ell+m} \b_{m}$ is an idempotent in $\End_{\G}(\G^{\nu_m} \ot \VV)$ (compare the argument for $\mathsf{p}$ in Proposition \ref{prop:fnu}) and $\G^{\nu_m} \ot \VV \cong \G^{\nu_{m-1}}$  is
an irreducible $\G$-module.  By Schur's lemma,   $ \frac{d_{m-1}}{d_m}  \b_m   \e_{\ell+m}  \b_m$
has to be a multiple of the identity,  which is  $\b_m = \b_m \ot {\bf 1}$, but since both are idempotents, they must be equal.   
  \end{proof}    
  
  \begin{remark}  Relation (ii)\,(c) shows that only the $\e_i$ and $\b_1$ are needed to generate $\Z_k$ for $k> \ell$, and the other
  generators $\b_j$, $2 \leq j \leq m$,  can be constructed recursively from them.    However, then relation (f) in (iii) needs to 
 be replaced with  a complicated expression in the $\e_i$ and $\b_1$.    \end{remark}
  
 Proposition \ref{prop:rels} covers all cases except when $k \geq  \diam(\G)$ for $\G = \C_n, \D_n$,   $n < \infty$,
 which will be considered next.  
 
\begin{prop}\label{prop:cdrels}  Assume $\G = \C_n,\D_n$ for $n < \infty$, and let $\Z_k = \Z_k(\G)$ where
$k \geq \diam(\G)$.  Then we have the following:
 \begin{itemize} \item[{\rm ($\C_n$)}]  $\Z_k$ has generators $\e_i \ (1 \leq i < k)$ and
 $\b_{j}^\pm \ (1 \leq j \leq \tilde n = \diam(\G))$ (where   $\b_j^\pm$ is the projection of $\VV_{\mathsf{new}}^{\ot j}$ onto $\G^{(\pm j)}$),   together with $\b_+^- =\mathsf{E_{-\underline{\mathsf 1},\underline{\mathsf 1}}}$, $\b_-^+ =\mathsf{E_{\underline{\mathsf 1},-\underline{\mathsf 1}}}$,  
 such that the relations in {\rm (a)-(e)} hold when  $\b_j =\b_{ j}^+$  or when   $\b_j=\b_{j}^-$.  In addition,   the following relations hold:
 \begin{enumerate}
 \item[{\hbox{\rm (f$_\C$)}}]  $(\b_{\tilde n}^\pm)^2  = \b_{\tilde n}^\pm$ \  and \ $ \b_{i}^\pm \b_{ j}^\mp = 0 = 
 \b_{j}^\mp  \b_{i}^\pm$,   \  for all \  $1 \leq i ,j \leq \tilde n$;  and   
\item[{\hbox{\rm (g$_\C$)}}]    for $\b_+^+  :=\b_{ \tilde n}^+$, and $\b_-^-  := \b_{\tilde n}^-$,   
\[\b_\gamma^\zeta \b_\eta^\vartheta  = \delta_{\gamma,\vartheta} \b_\eta^\zeta \ \ \hbox{\rm for } \ \,
\gamma,\zeta,\eta,\vartheta \in \{-,+\}  \ \ \hbox{\rm and}  \ \   \b_j^{\pm} \b_{\gamma}^\zeta  = 0  =  \b_{\gamma}^\zeta \b_j^{\pm}
 \ \  \ \hbox{\rm for } \ \ \zeta \neq \gamma,  \ \, 1 \leq j < \tilde n.\]
 \end{enumerate} 
  \item[{\rm ($\D_n$)}]  $\Z_k$ has generators $\e_i \ (1 \leq i < k)$,  $\b_j \  (1 \leq j  < n = \diam(\G))$, and $\b'$,  where 
   \ $\b_j$ is the projection of $\VV^{\ot (j+1)}$ onto  $\G^{(j+1)}$, and 
  $\b'$  is the projection of $\VV^{\ot n}$ onto $\G^{(n')}$.    They  satisfy the
  relations in {\rm (a)-(f)} of Proposition \ref{prop:rels},  and additionally
  \begin{enumerate} 
\item[{\hbox{\rm (g$_\D$)}}]   $\b_j \b' = \b' = \b' \b_j$, \ for $1 \leq j <n$, \,  $(\b')^2 = \b'$, and $\b_{n-1} \b' = 0 = \b' \b_{n-1}$;
\item[{\hbox{\rm (h$_\D$)}}]   $\e_n \b' \e_n = \half \b_{n-1} \e_n$, \  $\e_i \b' = 0 = \b'\e_i = 0$ for   $1 < i < n$,   and $\e_i \b' = \b' e_i$ for   $i > n$; 
\item[{\hbox{\rm (i$_\D$)}}] $\b_{n-1} =2\b_{n-1} \e_n \b_{n-1}$ and $\b' =2\b' \e_n \b'$. 
\end{enumerate} 
 \end{itemize}
 \end{prop}  
 
 \begin{proof} 
 For $\G = \C_n$, the fact that $\e_1, \ldots, \e_{k-1}, \b_1^\pm, \ldots, \b_{\tilde n}^\pm,  \b_+^-, \b_-^+$ generate $\Z_k$ follows from Theorem \ref{thm:JBC} parts (a), (d), and (e).  Relations (a)-(e) hold as in  Proposition \ref{prop:rels}.     Using (c) of Proposition \ref{prop:rels} and induction,  it is straightforward to prove that $\b_{k}^{\pm} = \mathsf{E_{\pm\underline{\mathsf 1},\pm\underline{\mathsf 1}}} \in \Z_k$,  where $\underline{\mathsf 1}$ is the $k$-tuple of all $1$s.  The relations in   (f$_\C$) and  (g$_\C$) then follow by multiplication of matrix units.
 
 For $\G = \DD_n$,  the  fact that $\e_1, \ldots, \e_{k-1},  \b_1, \ldots, \b_{n-1},  \b'$ generate $\Z_k$ follows from Theorem \ref{thm:JBC} parts (a), (d), and (f). 
 In $\Z_2$,  the projection onto $\G^{(2)}$ is $\b_1 = \mathsf{E_{(1,1),(1,1)}} + \mathsf{E_{(-1,-1),(-1,-1)}}$.   It follows easily by induction that for $2 \le j  <  n-1$,
$\b_j = \b_{j-1} - \b_{j-1}  \e_{j} \b_{j-1} =  \mathsf{E_{\underline{\mathsf 1},\underline{\mathsf 1}}} +  \mathsf{E_{-\underline{\mathsf 1},-\underline{\mathsf 1}}}$  projects $\VV^{\otimes j}$ onto $\G^{(j)}$, where
$\underline{\mathsf 1}$ is the $j$-tuple of all $1$s.  When $j =  n-1$ we have $\b_{n-1}  \e_{n} \b_{n-1} =  \mathsf{E_{\underline{\mathsf 1},\underline{\mathsf 1}}} +  \mathsf{E_{-\underline{\mathsf 1},-\underline{\mathsf 1}}}$ is  the projection onto $\G^{(n)} \oplus \G^{(n')}$,  and this splits into projections $\b_{n-1} = \frac{1}{2}( \mathsf{E_{\underline{\mathsf 1},\underline{\mathsf 1}}} +  \mathsf{E_{-\underline{\mathsf 1},-\underline{\mathsf 1}}}  - 
 \mathsf{E_{\underline{\mathsf 1},-\underline{\mathsf 1}}} - \mathsf{E_{-\underline{\mathsf 1},\underline{\mathsf 1}}} )$ and $\b' = \frac{1}{2}( \mathsf{E_{\underline{\mathsf 1},\underline{\mathsf 1}}} +  \mathsf{E_{-\underline{\mathsf 1},-\underline{\mathsf 1}}}  + 
 \mathsf{E_{\underline{\mathsf 1},-\underline{\mathsf 1}}} + \mathsf{E_{-\underline{\mathsf 1},\underline{\mathsf 1}}} )$, which project onto $\G^{(n)}$ and $\G^{(n')}$, respectively.
The other relations follow by multiplication of matrix units. 
\end{proof}

\begin{remark}\label{rem:nopresent}  The results of this section identify a set of generators for
 each centralizer algebra and relations that they satisfy.  In \cite{GH},  these relations are
used to give a presentation on generators and relations for $\Z_k(\G)$. \end{remark} 
  \end{subsection}  
 \end{section}

\begin{section}{The Cyclic Subgroups} \end{section} 
Let  $\C_n$ denote the cyclic subgroup of $\mathsf{SU}_2$ generated by 
\[g = \left (\begin{array}{cc} \zeta^{-1}  & 0 \\ 0 & \zeta \end{array}\right ) \in \mathsf{SU}_2,\]
where $\zeta = \zeta_n$, a primitive $n$th root of unity.    
The irreducible  modules for  $ \C_n$ are all one-dimensional and are given by \   $\C_n^{(\ac)}= \CC\v_\ac$ \ for \  $\ac = 0,1, \dots, n-1$,  where
\ $g \v_\ac = \zeta^\ac \v_\ac$,  and \break  $\C_n^{(\ac)} \otimes \C_n^{(m)}  \cong \C_n^{(\ac+m)}$ (superscripts interpreted $\mathsf{mod}\, n$).
Thus,  we can assume that the labels for the irreducible $\C_n$-modules are $(\ell)$, where 
$\ell \in \Lambda(\C_n) = \{0,1,\dots,n-1\}$,  with the understanding that  $(j) = (\ac)$ whenever
an integer $j$ such that $j \equiv \ac \modd n$ occurs in some expression.    

The  natural $\C_n$-module $\VV$  
 of $2 \times 1$ column vectors which $\C_n$ acts on by matrix
 multiplication can be identified with the module  $\C_n^{(-1)}  \oplus \C_n^{(1)}$.   As before, 
we let   $\v_{-1} = (1,0)^{\mathsf t}$  
and $\v_1 =(0,1)^{\mathsf t}$. 
 
 \begin{subsection}{The centralizer algebra $\ZZ_k(\C_n)$} \end{subsection}  
 
Our aim in this section  is to understand the centralizer algebra $\ZZ_k(\C_n)$ of the
$\C_n$-action on $\VV^{\otimes k}$ and the representation theory of $\ZZ_k(\C_n)$.  As in \eqref{eq:branch2},  let     
\begin{equation}\label{eq:defnprime} \tilde n = \begin{cases} n & \quad \hbox{if} \ \ n \ \ \hbox{is odd}  \\
\half n & \quad \hbox{if} \ \ n \ \ \hbox{is even}. \end{cases}\end{equation} 
Assume $\mathsf{r} = (r_1, \dots, r_k) \in \{-1,1\}^k$,  and
set 
\begin{equation}\label{eq:sigvert} \mathsf{\vert r \vert } = \left | \left\{\,r_i \mid r_i = -1\,\right \}\right |. \end{equation}    
Corresponding to $\mathsf{r} \in \{-1,1\}^k$ is 
the vector $\v_{\mathsf{r}} = \v_{r_1} \otimes \cdots \otimes \v_{r_k} \in \VV^{\otimes k}$, and 
\begin{equation}\label{eq:gact} g \v_{\mathsf{r}} =  \zeta^{k-2\vert\mathsf{r}\vert} \v_{\mathsf{r}}.\end{equation}
For two such $k$-tuples $\mathsf{r}$ and $\mathsf{s}$,    
\begin{equation}\label{eq:modrel}  k-2 \mathsf{\vert r \vert}    
\equiv k-2\mathsf{ \vert s \vert}  \  \modd n \\ 
\iff  \mathsf{ \vert r \vert  \equiv  \vert s \vert} \ \modd \tilde n.\end{equation} 
 
 Recall that $\Lambda_k(\C_n)$ is the subset of $\Lambda(\C_n) = \{0,1,\dots, n-1\}$ of labels for the irreducible
 $\C_n$-modules occurring in $\VV^{\ot k}$.     Now if $\ell \in \Lambda(\C_n)$,  then  
\[\ac \in \Lambda_k(\C_n)   \iff   
k-2 \mathsf {\vert r\vert} \equiv \ell  \modd n \ \ \hbox{\rm for some} \ \ \mathsf{r} \in \{-1,1\}^k. \] Thus,
\begin{equation}\label{eq:lamkCn} \Lambda_k(\C_n) = \left \{\ell \in \Lambda(\C_n) \mid \ell \equiv  k-2a_\ell \modd n 
\ \ \hbox{\rm for some} \ \ a_\ell \in \{0,1,\dots, k\} \right \}.\end{equation}
\emph{We will always assume $a_\ell$ is the minimal value in $\{0,1,\dots, k\}$ with that property.}

In particular,  $k-\ell$ must be even when $n$ is even.
Hence there are at most $\tilde n$ distinct values in $\Lambda_k(\C_n)$.    When $k \geq \tilde n -1$, then for every 
 $a \in \{0,1,\dots, \tilde n-1\}$,  there exists an $\ell \in \Lambda_k(\C_n)$ so that   $k-2a \equiv \ell \modd n$, and there are exactly $\tilde n$ distinct values   in $\Lambda_k(\C_n)$.  \msk 

\begin{lemma}\label{L:centmatrix}  Assume  $\mathsf{r}$ and $\mathsf{s}$ are two $k$-tuples
satisfying  $\mathsf{\vert r \vert  \equiv \vert s \vert} \modd \tilde n.$
Let $\mathsf{E_{r,s}}$ be the transformation on 
$\VV^{\otimes k}$ defined by
\begin{equation} \mathsf{E_{r,s}} \v_{\mathsf{t}} = \delta_{\mathsf{s}, \mathsf{t}} \v_{\mathsf{r}}.\end{equation}
Then $ \mathsf{E_{r,s}}  \in \ZZ_k({\C_n})$.  \end{lemma} 

\begin{proof}  Note that by \eqref{eq:gact}, $g \mathsf{E_{r,s}} \v_{\mathsf{t}}  
 =  \delta_{\mathsf{s}, \mathsf{t}}  \zeta^{k-2  \mathsf{\vert r \vert}} \v_{\mathsf{r}}$, 
while $\mathsf{E_{r,s}}g \v_{\mathsf{t}}  = \zeta^{k-2\mathsf{\vert t \vert}}  \delta_{\mathsf{s}, \mathsf{t}} \v_{\mathsf r} =  \zeta^{k-2\mathsf{\vert s \vert}}  \delta_{\mathsf{s}, \mathsf{t}} \v_{\mathsf r}$.   Consequently,  $g  \mathsf{E_{r,s}}  = 
 \mathsf{E_{r,s}} g$ for all such tuples $\mathsf{r}, \mathsf{s}$,   and   
$ \mathsf{E_{r,s}}  \in \ZZ_k(\C_n)$  by \eqref{eq:modrel}.   \end{proof}  

\begin{thm}\label{T:centcycbasis}  \begin{itemize} \item[{\rm (a)}]  The set 
\begin{equation}\label{eq:zcycbasis} \mathcal B^k(\C_n) = \{\mathsf{E_{r,s}} \mid \mathsf{r,s} \in \{-1,1\}^k, \ \mathsf{\vert r \vert  \equiv \vert s \vert} \modd \tilde n\} \end{equation} 
  is a basis for the centralizer algebra $\Z_k({\C_n}) = \mathsf{End}_{\C_n}(\VV^{\otimes k})$. 
\item[{\rm (b)}]  The set  $\{z_\ell \mid \ell \in \Lambda_k(\C_n)\}$ is a basis for the 
center of $\Z_k(\C_n)$, where  
\begin{equation} \label{eq:cencycbasis}  z_\ell = \sum_{{\mathsf{r} \in \{-1,1\}^k} \atop
{k-2\mathsf{\vert r\vert} \equiv \ell \modd n}}  \mathsf{E_{r,r}}   \qquad \hbox{\rm for} \ \  \ell \in \Lambda_k(\C_n). \end{equation} 
\item[{\rm (c)}] If $n = 2\tilde n$ and $\tilde n$ is odd,  then $\Z_k({\C_n}) \cong \Z_k(\C_{\tilde n})$.
\item[{\rm (d)}]  The dimension of  the centralizer algebra   $\ZZ_k(\C_n)$ is the
coefficient of $z^k$ in \newline  $(1+z)^{2k} \big |_{z^{\tilde n} = 1}$;  hence,  it  is given by 
\begin{equation}\label{eq:cyccentdim} \dimm \Z_k(\C_n) =  \sum_{{0 \leq a,b\leq k} \atop {a \equiv b \modd \tilde n}}  {k \choose a}{k \choose b}. \end{equation}  
 \item[{\rm (e)}]   The number of walks of $2k$ steps from $0$ to $0$ on a circular graph with $n$ nodes 
(i.e., on the affine Dynkin diagram of type $\hat{\mathsf A}_{n-1}$) is 
$\displaystyle{\sum_{{0 \leq a,b\leq k} \atop {a \equiv b \modd \tilde n}}  {k \choose a}{k \choose b}.}$
\end{itemize}     \end{thm}

\begin{remark}  The notation $(1+z)^{2k} \big |_{z^{\tilde n} = 1}$ used in the statement of this
result can be regarded as saying  consider $(1+z)^{2k}$ in  the polynomial algebra $\mathbb C[z]$ modulo
the ideal generated by $z^{\tilde n}-1$ where $\tilde n$ is as in \eqref{eq:defnprime}.   \end{remark}   

\begin{proof}  (a)  For  
$X \in \mathsf{End}(\VV^{\otimes k})$,  suppose that 
\ $X \v_{\mathsf{s}}  = \sum_{\mathsf{r}}  X_{\mathsf {r,s}} \v_{\mathsf{r}} $ \ 
for scalars $X_{\mathsf {r,s}} \in \mathbb C$, where $\mathsf{r}$ ranges over all the
$k$-tuples in $\{-1,1\}^k$.  Then  $X \in \Z_k({\C_n})$ if and only if  $g^{-1}Xg = X$ if and only if  
\[g^{-1}Xg \v_{\mathsf{s}}   =  \sum_{\mathsf{r}} \zeta^{(k-2\mathsf{\vert s\vert}) -(k-2\mathsf{\vert r\vert})} X_{\mathsf {r,s}}  \v_{\mathsf{r}} =
 \sum_{\mathsf{r}} X_{\mathsf {r,s}}  \v_{\mathsf{r}}.\]
Hence,  for all $\mathsf {r,s}$, with $X_{\mathsf {r, s}} \neq 0$,  it must be that $\zeta^{2(\mathsf{\vert r \vert - \vert s \vert})} = 1$; that is,
$\mathsf{\vert r \vert \equiv\vert s \vert} \modd \tilde n$ by \eqref{eq:modrel}.   Thus,  $X = \sum_{\vert \mathsf{r}\vert \equiv \vert\mathsf{s}\vert \modd \tilde n} 
X_{\mathsf {r,s}} \mathsf{E_{r,s}}$, and the transformations $\mathsf{E_{r,s}}$ with 
$\mathsf{\vert r \vert \equiv \vert s \vert}  \modd \tilde n$ span $ \Z_k({\C_n})$.  It is easy to see that the $\mathsf{E_{r,s}}$
multiply like matrix units and 
are linearly independent, so they form a basis of $ \Z_k({\C_n})$.

(b) For each $\ell \in \Lambda_k(\C_n)$, the basis elements $\mathsf{E_{r,s}}$  with $\mathsf{\vert r \vert = \vert s \vert} \equiv
\half(k-\ell) \modd \tilde n$ form a matrix algebra, whose center 
 is $\CC z_\ell$ where $z_\ell$ is as in \eqref{eq:cencycbasis}.
Since $\Z_k(\C_n)$ is the direct sum of these matrix algebra ideals
as $\ell$ ranges over $\Lambda_k(\C_n)$, the result follows.  

When $n = 2\tilde n$ and $\tilde n$ is odd,  the elements $\{\mathsf{E_{r,s}}
\mid \mathsf{\vert r \vert \equiv \vert s \vert } \modd \tilde n\}$  comprise a basis of both $\Z_k({\C_n})$ and $\Z_k(\C_{\tilde n})$ to give the assertion in (c).  

(d)  It follows from part (a)  that   
\[\dimm\Z_k(\C_n) = \sum_{{0 \leq a,b\leq k} \atop {a \equiv b \modd \tilde n}}  {k \choose a}{k \choose b} \ = \ 
 \sum_{{0 \leq a,b\leq k} \atop {a \equiv b \modd \tilde n}} {k \choose a}{k \choose k-b}.\]
Since $a+k-b \equiv k \modd \tilde n$,    this expression is the coefficient of
 $z^k$  in  $(1+z)^k(1+z)^k \big |_{z^{\tilde n} = 1} = (1+z)^{2k}\big |_{z^{\tilde n} = 1}$,
 as claimed. 
 
(e)  This is an immediate consequence of \eqref{eq:evid}.\end{proof}   

 \begin{example}\label{ex:cyc}  Suppose $n = 8$ (so $\tilde n = 4$) and $k = 6$.    Then 
\begin{eqnarray*}
&&\big\vert \{ ( \mathsf {r},\mathsf{s}) \mid \mathsf{\vert r\vert \equiv\vert s \vert} \equiv 0 \modd 4\} \big\vert=  {6 \choose 0}^2 + {6 \choose 4}^2  + 2 {6 \choose 0}{6 \choose 4}   
= \  256 \\
&&\big\vert \{ ( \mathsf {r},\mathsf{s}) \mid\mathsf{\vert r\vert \equiv\vert s \vert}  \equiv 1 \modd 4\}\big\vert= {6 \choose 1}^2 + {6 \choose 5}^2 +  2 {6 \choose 1}{6 \choose 5}  
= \ 144 \\
&& \big\vert\{ ( \mathsf {r},\mathsf{s}) \mid \mathsf{\vert r\vert \equiv\vert s \vert} \equiv 2 \modd 4\}\big\vert= {6 \choose 2}^2 + {6 \choose 6}^2 +  2 {6 \choose 2}{6 \choose 6}  
= \  256 \\
&&\big\vert \{ ( \mathsf {r},\mathsf{s}) \mid \mathsf{\vert r\vert \equiv\vert s \vert} \equiv 3 \modd 4\}\big\vert  = {6 \choose 3}^2   = \  400.
\end{eqnarray*} 
Therefore $\dimm \ZZ_6(\C_8) = 1056.$     Now observe that when $k = 6$ and $n = 8$ that 
\begin{eqnarray*}
(1+z)^{2k} \big |_{z^{\tilde n} = 1}  &=&   (1+z)^{12} \big |_{z^4 = 1}  \\
&=& 1 + 12 z \, +\, 66 z^2 + 220 z^3  
+ 495 + 792z + 924 z^2 + 792 z^3 \\
&&\quad  +495 + 220z + 66 z^2 + 12z^3  +1. 
\end{eqnarray*} 
Since $k = 6  \equiv 2  \modd 4$,   by (c) of Theorem \ref{T:centcycbasis}  we have that  $\dimm \Z_6({\C_8})$ is
the coefficient of $z^2$ in this expression, so
 $\dimm  \ZZ_6({\C_8}) = 66 + 924 + 66 = 1056$, in agreement
 with the above calculation.  \end{example} 

\begin{remark}\label{rem:cycdiag}  The matrix units can be viewed diagrammatically.
For example,  if $k = 12$, $n = 6$, $\tilde n = 3$, and
$\mathsf{r} = (-1, -1, 1,-1,-1,1,1,1,1,1,1,-1)\in \{-1,1\}^{12}$,  then
 $\mathsf{\vert r \vert}= 5 \equiv 2 \modd 3$,
and if   $\mathsf{s} =
(1,-1,-1,-1,-1,-1,1,-1,-1,1,-1,1)$, then
$\vert \mathsf{s}\vert = 8
\equiv 2 \equiv \vert \mathsf{r}\vert \modd 3$.  In this case,  we  identify the matrix unit $\mathsf{E_{r,s}}$ with the diagram below
\[
\mathsf{E_{r,s}} =
\begin{array}{c}
 \begin{tikzpicture}[scale=.5,line width=1pt]
\foreach \i in {1,...,12}  { \path (\i,1) coordinate (T\i); \path (\i,-1)
coordinate (B\i); }
\filldraw[fill= black!10,draw=black!10,line width=4pt]  (T1) -- (T12) --
(B12) -- (B1) -- (T1);
\draw[black,line width = 2 pt,fill=black!60]   (T1) .. controls +(-90:0.5)
and +(-90:0.5) ..  (T2) .. controls +(-90:0.5) and +(-90:0.5) .. (T4) ..
controls +(-90:0.5) and +(-90:0.5) .. (T5)  .. controls +(-90:1.15) and
+(-90:1.15) .. (T12)
.. controls +(-45:0.5) and +(90:0.5) .. (B11)  .. controls +(90:0.5) and
+(90:0.5) ..  (B9) .. controls +(90:0.5) and +(90:0.5) .. (B8) .. controls
+(90:0.5) and +(90:0.5) .. (B6) .. controls +(90:0.5) and +(90:0.5) .. (B5)
.. controls +(90:0.5) and +(90:0.5) .. (B4) .. controls +(90:0.5) and
+(90:0.5) .. (B3) .. controls +(90:0.5) and +(90:0.5) .. (B2) .. controls
+(90:0.5) and +(-90:0.5) .. (T1);
\foreach \i in {1,...,12}  { \draw  (T\i)
 node[above=0.1cm]{$\scriptstyle{\i}$}; \draw  (B\i)
 node[below=0.1cm]{$\scriptstyle{\i'}$};  }
\foreach \i in {1,...,12}  { \fill (T\i) circle (6pt); \fill (B\i) circle
(6pt); }
\draw (T1) node[white] {\bf -}; \draw (T2) node[white] {\bf -}; \draw (T3)
node[white] {\bf +}; \draw (T4) node[white] {\bf -}; \draw (T5) node[white]
{\bf -}; \draw (T6) node[white] {\bf +}; \draw (T7) node[white] {\bf +};
\draw (T8) node[white] {\bf +};\draw (T9) node[white] {\bf +};\draw (T10)
node[white] {\bf +};\draw (T11) node[white] {\bf +};\draw (T12) node[white]
{\bf -};
\draw (B1) node[white] {\bf +};\draw (B2) node[white] {\bf -};\draw (B3)
node[white] {\bf -};\draw (B4) node[white] {\bf -};\draw (B5) node[white]
{\bf -};\draw (B6) node[white] {\bf -};
\draw (B7) node[white] {\bf +};\draw (B8) node[white] {\bf -}; \draw (B9)
node[white] {\bf -};\draw (B10) node[white] {\bf +};\draw (B11) node[white]
{\bf -};\draw (B12) node[white] {\bf +};
\end{tikzpicture}.
\end{array}
\]

Each such two-rowed $k$-diagram $d$ corresponds to two subsets,  $t(d)  \subseteq \{1, 2, \ldots, k\}$ and
$b(d) \subseteq \{1', 2', \ldots, k'\}$,  recording the positions of the $-1$s in the top and bottom rows of $d$, hence in $\mathsf{r}$ and $\mathsf{s}$ respectively,
and  $\vert t(d)\vert \equiv \vert b(d)\vert  \modd \tilde n$.   Under this correspondence,  diagrams
multiply as matrix units.  Thus, if  $d_1$ and $d_2$ are diagrams,  then
\[
d_1 d_2 = \delta_{b(d_1),t(d_2)} d_3
\]
where $d_3$ is the unique diagram given by $t(d_3) = t(d_1)$ and $b(d_3) =
b(d_2)$. For example, if $\tilde n = 3$, 
\[
\begin{array}{c}
 \begin{tikzpicture}[scale=.5,line width=1pt]
\foreach \i in {1,...,8}  { \path (\i,1) coordinate (T\i); \path (\i,-1)
coordinate (B\i); }
\filldraw[fill= black!10,draw=black!10,line width=4pt]  (T1) -- (T8) --
(B8) -- (B1) -- (T1);
\draw[black,line width = 2 pt,fill=black!60]   (T1) .. controls +(-90:0.5)
and +(-90:0.5) ..  (T2) .. controls +(-90:0.5) and +(-90:0.5) .. (T4) ..
controls +(-90:0.5) and +(-90:0.5) .. (T5)  .. controls +(-90:1.15) and
+(-90:1.15) .. (T8) -- (B8)   .. controls
+(90:0.5) and +(90:0.5) .. (B6)
.. controls +(90:0.5) and +(90:0.5) .. (B4) .. controls +(90:0.5) and
+(90:0.5) .. (B3) .. controls +(90:0.5) and +(90:0.5) .. (B2) .. controls
+(90:0.5) and +(-90:0.5) .. (T1);
\foreach \i in {1,...,8}  { \fill (T\i) circle (6pt); \fill (B\i) circle
(6pt); }
\draw (T1) node[white] {\bf -}; \draw (T2) node[white] {\bf -}; \draw (T3)
node[white] {\bf +}; \draw (T4) node[white] {\bf -}; \draw (T5) node[white]
{\bf -}; \draw (T6) node[white] {\bf +}; \draw (T7) node[white] {\bf
+};\draw (T8) node[white] {\bf -};
\draw (B1) node[white] {\bf +};\draw (B2) node[white] {\bf -};\draw (B3)
node[white] {\bf -};\draw (B4) node[white] {\bf -};\draw (B5) node[white]
{\bf +};\draw (B6) node[white] {\bf -};
\draw (B7) node[white] {\bf +};\draw (B8) node[white] {\bf -};
\draw (-.25,0) node {$d_1=$};
\end{tikzpicture}  \\
 \begin{tikzpicture}[scale=.5,line width=1pt]
\foreach \i in {1,...,8}  { \path (\i,1) coordinate (T\i); \path (\i,-1)
coordinate (B\i); }
\filldraw[fill= black!10,draw=black!10,line width=4pt]  (T1) -- (T8) --
(B8) -- (B1) -- (T1);
\draw[black,line width = 2 pt,fill=black!60]
(B3) .. controls +(+90:0.5) and +(+90:0.5) .. (B6)  .. controls +(45:0.5)
and +(-90:0.5) .. (T8)   .. controls
+(-90:0.5) and +(-90:0.5) .. (T6)
.. controls +(-90:0.5) and +(-90:0.5) .. (T4) .. controls +(-90:0.5) and
+(-90:0.5) .. (T3) .. controls +(-90:0.5) and +(-90:0.5) .. (T2) .. controls
+(-90:0.5) and +(+90:0.5) .. (B3);
\foreach \i in {1,...,8}  { \fill (T\i) circle (6pt); \fill (B\i) circle
(6pt); }
\draw (B1) node[white] {\bf +}; \draw (B2) node[white] {\bf +}; \draw (B3)
node[white] {\bf -}; \draw (B4) node[white] {\bf +}; \draw (B5) node[white]
{\bf +}; \draw (B6) node[white] {\bf -}; \draw (B7) node[white] {\bf
+};\draw (B8) node[white] {\bf +};
\draw (T1) node[white] {\bf +};\draw (T2) node[white] {\bf -};\draw (T3)
node[white] {\bf -};\draw (T4) node[white] {\bf -};\draw (T5) node[white]
{\bf +};\draw (T6) node[white] {\bf -};
\draw (T7) node[white] {\bf +};\draw (T8) node[white] {\bf -};
\draw (-.25,0) node {$d_2=$};
\end{tikzpicture}
\end{array}
=
\begin{array}{c}
 \begin{tikzpicture}[scale=.5,line width=1pt]
\foreach \i in {1,...,8}  { \path (\i,1) coordinate (T\i); \path (\i,-1)
coordinate (B\i); }
\filldraw[fill= black!10,draw=black!10,line width=4pt]  (T1) -- (T8) --
(B8) -- (B1) -- (T1);
\draw[black,line width = 2 pt,fill=black!60]
(T1) .. controls +(-90:0.5)
and +(-90:0.5) ..  (T2) .. controls +(-90:0.5) and +(-90:0.5) .. (T4) ..
controls +(-90:0.5) and +(-90:0.5) .. (T5)  .. controls +(-90:1.15) and
+(-120:1.15) .. (T8).. controls +(-90:0.5) and +(45:0.5) .. (B6)   ..
controls
+(90:0.5) and +(90:0.5) ..  (B3)  .. controls
+(90:0.5) and +(-90:0.5) .. (T1);
\foreach \i in {1,...,8}  { \fill (T\i) circle (6pt); \fill (B\i) circle
(6pt); }
\draw (T1) node[white] {\bf -}; \draw (T2) node[white] {\bf -}; \draw (T3)
node[white] {\bf +}; \draw (T4) node[white] {\bf -}; \draw (T5) node[white]
{\bf -}; \draw (T6) node[white] {\bf +}; \draw (T7) node[white] {\bf
+};\draw (T8) node[white] {\bf -};
\draw (B1) node[white] {\bf +}; \draw (B2) node[white] {\bf +}; \draw (B3)
node[white] {\bf -}; \draw (B4) node[white] {\bf +}; \draw (B5) node[white]
{\bf +}; \draw (B6) node[white] {\bf -}; \draw (B7) node[white] {\bf
+};\draw (B8) node[white] {\bf +};
\draw (9.25,0) node {$=d_3$.};
\end{tikzpicture}
\end{array}
\]
 \end{remark}   
\begin{subsection}{Irreducible modules for $\ZZ_k({\C_n})$} \end{subsection} 

 For $\ell \in \Lambda_k(\C_n)$,  set 
\begin{equation}\label{eq:Zka} \Z_k^{(\ell)} :  = 
 \spann_\CC \{\v_{\mathsf{r}} \in \VV^{\otimes k} \ \big | \  k-2 \mathsf{\vert r \vert}  \equiv  
\ell  \modd n \} =   \spann_\CC \{\v_{\mathsf{r}} \in \VV^{\otimes k} \ \big | \  \mathsf{\vert r \vert}  \equiv  
a_\ell  \modd \tilde n \}, \end{equation}
where $a_\ell$ is as in \eqref{eq:lamkCn}.  
{\it  When we need to emphasize that we are working with the group $\C_n$,  we will write this as $\Z_k(\C_n)^{(\ac)}$.} 
Now the generator $g$ of $\C_n$ acts as the scalar $\zeta^{\ell}$ on $\Z_k^{(\ac)}$,  and these scalars are distinct for different values of $\ell
\in \{0,1,\dots, n-1\}$.    Therefore, 
\[\VV^{\otimes k}  = \bigoplus_{\ell \in \Lambda_k(\C_n)}  \Z_k^{(\ell)}\]
is a decomposition of $\VV^{\ot k}$  into $\C_n$-modules.   
 
The mappings $ \mathsf{E_{r,s}}$ with $\mathsf{r,s} \in \{-1,1\}^k$ and  $\mathsf{\vert r\vert \equiv\vert s \vert} \equiv a_\ell  \modd \tilde n$ act as 
matrix units on  $\Z_k^{(\ac)}$ and act trivially on $ \Z_k^{(m)}$ for  $m \in\Lambda_k(\C_n), \ m \neq \ac$.   In addition,   
\[\mathsf{span}\{ \mathsf{E_{r,s}} \mid \mathsf{\vert r\vert \equiv\vert s \vert} \equiv a_\ell  \, \modd \tilde n\} = \End(\Z_k^{(\ell)})
= \End_{\C_n}(\Z_k^{(\ell)}).\]    As a consequence,  we have
that the spaces $\Z_k^{(\ell)}$   are modules for $ \Z_k({\C_n})$.  Since they are also invariant under the action of
$\C_n$,  they are modules both  for $\C_n$ and for  $\Z_k(\C_n)$.   It is apparent that $\Z_k^{(\ell)}$  is irreducible as a $\Z_k(\C_n)$-module from the fact that the natural module for a matrix algebra is its unique irreducible module.   \msk

\begin{examples}  For any $m \geq 1$, let $\zeta_m$ be a primitive $m$th root of unity.  Assume $k = 5$ and $n = 12$, so $\tilde n = 6$.   Then $\Z_5(\C_{12})$ has
6 irreducible modules  $\Z_5(\C_{12})^{(\ell)}$ for  $\ell= 1,3,5,7,9,11$.   On them,  the generator
$g$ of $\C_{12}$ acts by the scalars $ \zeta_{12}, \,\zeta_{12}^{3},\,\zeta_{12}^{5}, \zeta_{12}^7, \,\zeta_{12}^9,\,\zeta_{12}^{11}$, respectively.   

The  algebra  $\Z_5(\C_6)$ has 3 irreducible modules $\Z_5(\C_{6})^{(\ell)}$ for $\ell =1,3,5$,  on
which the generator $g'$ of $\C_{6}$ acts by the scalars  $\zeta_6^1,\, \zeta_6^3,\, \zeta_6^5$, respectively.

The algebra  $\Z_5(\C_3)$ also has 3 irreducible modules $\Z_5(\C_{3})^{(\ell)}$ for $\ac = 0,1,2$,   on
which the generator $g''$ of $\C_{3}$ acts by the scalars  $1,\, \zeta_3,\zeta_3^2$, respectively .     

The vectors  $\{\v_{\mathsf r} \mid \mathsf{r} \in \{-1,1\}^5,  \  k-2\mathsf{\vert r \vert }\equiv  \ell \modd 3\}$
form a basis for  $\Z_5(\C_{6})^{(1)}$ and $\Z_5(\C_{3})^{(1)}$ when $\ell \equiv 1 \modd 3$;   for 
$\Z_5(\C_{6})^{(3)}$ and $\Z_5(\C_{3})^{(0)}$ when $\ell  \equiv 0 \modd 3$;  and for 
$\Z_5(\C_{6})^{(5)}$ and $\Z_5(\C_{3})^{(2)}$ when $\ell  \equiv  2 \modd 3$. 
\end{examples}

The next result gives an expression for the dimension of the module $\Z_k^{(\ell)} = \Z_k(\C_n)^{(\ell)}$.   

\begin{thm} \label{thm:dims}  With $\tilde n$ as in \eqref{eq:defnprime},  suppose $\ell \in \Lambda_k(\C_n)$ and
$k-2a_\ell \equiv \ell \modd n$ as in \eqref{eq:lamkCn}.
Then  $\Z_k^{(\ell)} =  \spann_\CC \{\v_{\mathsf{r}} \in \VV^{\otimes k} \ \big | \  \mathsf{\vert r \vert}  \equiv  
a_\ell \modd \tilde n \}$  is an irreducible $\Z_k(\C_n)$-module and the following hold:
\begin{itemize}
\item[{\rm (i)}] $\dimm \Z_k^{(\ell)}  = \displaystyle{ \sum_{{0 \leq b  \leq k} \atop {b \equiv a_\ell \modd \tilde n }
}  {k  \choose  b}},$ which is the coefficient of
$z^{a_\ell}$ in   $(1+z)^{k}\big |_{z^{\tilde n} = 1}$
\item[{\rm (ii)}]   As a bimodule for $ \C_n \times \Z_k({\C_n})$
\begin{equation}\label{eq:cycbimod} \VV^{\otimes k} \cong  \bigoplus_{\ac  \in \Lambda_k(\C_n)}  \left(\C^{(\ac)}_n \otimes  \Z_k^{(\ac)}\right). \end{equation} {\rm (}Here $ \C_n$ acts only on the factors $\C^{(\ac)}_n$  and $ \Z_k({\C_n})$ on the factors $  \Z_k^{(\ac)}$. 
Since $\Z_k^{(\ac)}$
is also both a $\C_n$ and a $\Z_k(\C_n)$-module and the actions commute,  the decomposition 
$\VV^{\ot k} = \bigoplus_{\ac \in \Lambda_k(\C_n)}  \Z_k^{(\ac)}$ is also a   $\C_n \times \Z_k({\C_n})$-bimodule
decomposition.{\rm )}
\item[{\rm (iii)}]   $\C^{(\ac)}_n$ occurs as a summand in the $\C_n$-module
$\VV^{\otimes k}$  with multiplicity  
\[m^{(\ell)}_k =  
 \displaystyle{ \sum_{  {0 \leq b \leq k} \atop {b\equiv a_\ell \modd \tilde n}}  {k  \choose  b} = \dimm \Z_k^{(\ac)}}.\]
\item[{\rm (iv)}]   The   number of walks on the
affine Dynkin diagram of type $\mathsf{\hat {A}}_{n\text{-}1}$ starting at node 0
and ending at node $\ell$ and taking $k$ steps is the coefficient of
$z^{a_\ell}$ in $(1+z)^{k}\big |_{z^{\tilde n} = 1}$, which is  $m^{(\ell)}_k  =  \dimm \Z_k^{(\ac)}.$
\end{itemize}
\end{thm} 

\begin{proof}  Part (i) follows readily from the fact that a basis for 
$\Z_k^{(\ell)}$  consists of the vectors $\v_{\mathsf r}$ labeled by the
 tuples $\mathsf{r} \in \{-1,1\}^k$ with  $\mathsf{\vert r \vert } = b \equiv a_\ell \modd \tilde n$
(see \eqref{eq:Zka}), and the number
 of $k$-tuples $\mathsf r$ with $b$ components equal to $-1$  
 is ${k \choose b}$.      Each such vector $\v_\mathsf{r}$ satisfies 
 $g \v_{\mathsf r} = \zeta^{k-2\vert\mathsf{r}\vert} \v_{\mathsf r} = \zeta^\ell \v_{\mathsf r}$; 
 hence $\CC \v_{\mathsf r} \cong  \C_n^{(\ell)}$ as a $\C_n$-module.
 The other statements are apparent from these.     \end{proof}

 \begin{examples}\label{exs: smallk}  Consider the following special cases for $\C_n$, where
 $\tilde n$ is as in \eqref{eq:defnprime}.
 \begin{itemize}
 \item[{\rm (i)}] If  $k < \tilde n$,  then 
$\dimm  \Z_k({\C_n}) = \displaystyle{\sum_{a =0}^k   {k \choose a}^2 = {2k \choose k}.}$
 \item[{\rm (ii)}]
If $k = \tilde n$,  and $\ell \in \{0,1,\dots, n-1\}$ such that $k-2 \cdot 0 \equiv \ell \modd n$,  then $\ell = 0$ if $n$ is odd and $\ell = \tilde n = \half n$ if $n$ is even. For this value of $\ell$, we have  $\dimm \Z_{\tilde n}^{(\ell)}  = \displaystyle{{\tilde n \choose 0} + {\tilde n \choose \tilde n} = 2}$.
Thus,
$\displaystyle{\dimm  \Z_{\tilde n}({\C_n}) = \sum_{a =1}^{\tilde n-1}   {\tilde n \choose a}^2  + 2^2 =  {2\tilde n \choose \tilde n} + 2.}$
\end{itemize}  \end{examples} 

\begin{example} {Suppose $k = 6$ and $n =8$, so $\tilde n = 4$.  The
irreducible $\C_8$-modules  $\C_8^{(\ell)}$ occurring in $\VV^{\ot 6}$ have $\ell =0,2,4,6$,
and $a_\ell = 3,2,1,0$, respectively, where $k-\ell \equiv 2a_\ell  \modd 8$.  We have the following expressions for the number  $m_6^{(\ell)}$  of times  that
$\C_8^{(\ell)}$ occurs as a summand  of $\VV^{\ot 6}$.  

\[{\begin{array}[t]{|c|c|c|}
\hline 
\ell &a_\ell & \quad \displaystyle{ m_6^{(\ell)}}  \\
\hline  \hline
0 & \quad  3 &  \hspace{-.35truein}  \displaystyle{m_6^{(0)} = {6 \choose 3}  = 20 = \dimm  \Z_6^{(0)}} \\
\hline 
2 & \quad 2  & \displaystyle{\quad   m_6^{(2)} =  {6 \choose 2} + {6 \choose 6}   =  16 = \dimm \Z_6^{(2)}} \\
\hline
4 & \quad 1 &\displaystyle{ \quad m_6^{(4)}= {6 \choose 1} + {6 \choose 5}   = 12 = \dimm  \Z_6^{(4)}} \\
\hline 
6 &  \quad 0 & \displaystyle{\quad m_6^{(6)} = {6 \choose 0} + {6 \choose 4}   = 16 = \dimm  \Z_6^{(6)}} \\ 
\hline
\end{array}}
\]
 \noindent In particular
\[\dimm  \Z_6(\C_8) = \sum_{\ell \in  \Lambda_6(\C_8)}\left(\dimm  \Z_6^{(\ell)}\right)^2 = 
20^2+ 16^2 + 12^2 + 16^2 = 1056,\] exactly as in Example \ref{ex:cyc}. 

The number of walks on the Dynkin diagram of type $\mathsf{\hat {A}}_{7}$  with 
$6$ steps starting and ending at  $0$ is the coefficient of $z^3$ in 
\[ (1+z)^6 \big |_{z^{4} = 1}  = \  \quad  1 + 6 z \, +\, 15 z^2 + 20 z^3
+ 15 + 6 z \, + \, z^2,\]
which is 20.   The number of walks starting at $0$ and ending at $4$
is the coefficient of $z$ in this expression, which is 6+6 = 12.}
\end{example} 

\subsection{The cyclic subgroup $\C_{\infty}$}

Let  $\C_\infty$ denote the cyclic subgroup of $\mathsf{SU}_2$ generated by 
\[g = \left (\begin{array}{cc} \zeta^{-1}  & 0 \\ 0 & \zeta \end{array}\right ) \in \mathsf{SU}_2,\]
where $\zeta = e^{i\theta}$  for any  $\theta \in \mathbb R$  such that $\zeta$ is not a root of unity.    Then  $\C_\infty$ has a natural
action on  $\VV^{\ot k}$, and the   
irreducible  $\C_\infty$-modules occurring in the modules  $\VV^{\ot k}$
are all one-dimensional and are given by  $\C_\infty^{(\ac)}= \CC\v_\ac$ for some $\ac \in \mathbb Z$,  where
$g \v_\ac = \zeta^\ac \v_\ac$  and  $\C_\infty^{(\ac)} \otimes \C_\infty^{(m)}  \cong \C_\infty ^{(\ac+m)}$. 
In particular,  $\VV = \C_\infty^{(-1)}  \oplus \C_\infty^{(1)}$, and $\C_\infty^{(\ell)} \ot \VV = 
\C_{\infty}^{(\ell-1)} \oplus \C_{\infty}^{(\ell+1)}$ for all $\ell$.   Thus, the representation graph $\mathcal R_{\VV}(\C_\infty)$
is the Dynkin diagram ${\mathsf A}_{\infty}$. 

\begin{equation}\label{AinfinityDynkin}
\begin{array}{rll}
\C_\infty: & \begin{array}{c} \includegraphics[scale=.7,page=7]{mckay-diagrams.pdf}  \end{array} &  ({\mathsf A}_\infty) \\
\end{array}
\end{equation}

Now  $g \v_{\mathsf r}  =  \zeta^{k-2 \mathsf{\vert r \vert}} \v_{\mathsf r},$
for all ${\mathsf r} \in  \{-1,1\}^k$, where $\vert \mathsf r \vert$ is as
in \eqref{eq:sigvert}.  The arguments in the previous section can be easily adapted  to show the following.

\begin{thm}\label{thm:cycinf}  Let  $\Z_k = \Z_k(\C_{\infty})  = \End_{\C_{\infty}}(\VV^{\ot k})$.    Then
 \begin{itemize}  
 \item[{\rm (a)}] $\mathcal B^k(\C_{\infty}) = \left \{ \mathsf{E_{r,s}} \mid  \mathsf{r,s} \in \{-1,1\}^k,  \mathsf{\vert r \vert = \vert s \vert} \right \}$ is a basis for
 $\Z_k$, where $\mathsf{E_{r,s}}\v_{\mathsf t} = \delta_{\mathsf{s},\mathsf{t}} \v_{\mathsf r}$ and $\mathsf{E_{r,s}} \mathsf{E_{t,u}} = \delta_{\mathsf s, \mathsf t} \mathsf{E_{r,u}}$  for all $\mathsf{r,s,t,u} \in 
 \{-1,1\}^k$.
 \item[{\rm (b)}]  The irreducible modules for $\Z_k$ are labeled by $\Lambda_k(\C_{\infty}) =
 \{ k-2a \mid  a = 0,1,\dots, k\}$.  A  basis for the irreducible $\Z_k$-module $\Z_k^{(k-2a)}$ is 
 $\{\v_{\mathsf r} \mid {\mathsf r} \in \{-1,1\}^k, \  \mathsf{\vert r \vert} = a\},$
and  $\dimm \Z_k^{(k-2a)} = \displaystyle{{k \choose a}}$.  The module  $\Z_k^{(k-2a)}$  is also a $\C_{\infty}$-module,
 hence a $(\C_{\infty} \times \Z_k)$-bimodule. 
  \item[{\rm (c)}]  $\dimm \Z_k = \displaystyle{ \sum_{a=0}^k  {k \choose a}^2 =  {2k \choose k}} =
  \hbox{\rm coefficient of } \ z^k  \ \hbox{\rm in} \  (1+z)^{2k}$. 
  \item[{\rm (d)}]  $ \Z_k$ is isomorphic to the planar rook algebra $\mathsf {PR}_k$.
  \end{itemize}
  \end{thm}  
 \begin{remark}\label{rem:prook}  A word of explanation about part (d) is in order.   The planar rook algebra $\mathsf {PR}_k$ was studied in \cite{FHH},
  where it was shown (see \cite[Prop.~3.3]{FHH}) to have a basis of matrix units  $\{\mathsf{X_{R,S}} \mid \mathsf{R,S} \subseteq \{1,\dots, k\}, \ 
  |\mathsf{R}| = |\mathsf{S}|\}$ such that $\mathsf{X_{R,S}}\mathsf{X_{T,U}} = \delta_{\mathsf{S},\mathsf{T}} \mathsf{X_{R,U}}$. 
Identifying the subset $\mathsf{R}$ of $\{1,\dots, k\}$ with the $k$-tuple $\mathsf{r} = (r_1,\dots, r_k) \in \{-1,1\}^k$ such $r_j = -1$ if 
 $j \in \mathsf{R}$ and $r_j = 1$ if $j \not \in \mathsf{R}$, it is easy to see that $\Z_k(\C_{\infty})  \cong \mathsf {PR}_k$ via the correspondence
  $\mathsf{E_{r,s}} \mapsto \mathsf{X_{R,S}}$.   It follows from Theorem \ref{T:centcycbasis}  and Example 2.19\,(i) that when $k < \tilde n$,  the centralizer algebra $\Z_k(\C_n)$ is also isomorphic to  $\mathsf {PR}_k$.  \end{remark} 
  
\begin{remark}\label{rem:circ}  As a module for the circle subgroup (maximal torus)  
$S^1 = \left \{\left (\begin{smallmatrix}e^{-i t}  & 0 \\ 0 & e^{it} \end{smallmatrix}\right )\, \bigg | \, t \in \mathbb R\right\} $
of $\SU_2$,   $\VV$ has the same decomposition $\VV = \CC \v_{-1} \oplus \CC \v_1$ into submodules (common eigenspaces) as it has for $\C_\infty$.   Thus, the
centralizer algebra $\Z_k(S^1) \cong \Z_k(\C_\infty)$, and its structure and representations are also given by
this theorem.  \end{remark}

\begin{section} {The Binary Dihedral Subgroups} \end{section}
 
Let  $\DD_n$ denote the binary dihedral subgroup of $\mathsf{SU_2}$ of order $4n$  generated by 
the elements $g, h \in \mathsf{SU_2}$, where 
\begin{equation}\label{eq:defngh} g = \left (\begin{array}{cc} \zeta^{-1}  & 0 \\ 0 & \zeta \end{array}\right),  \qquad
h = \left (\begin{array}{cc} 0  & i \\  i & 0 \end{array}\right),\end{equation}
$\zeta = \zeta_{2n}$, a {\it primitive $(2n)$th root of unity} in $\CC$,  and $i = \sqrt{-1}$.  The defining relations for $\DD_n$ are
\begin{equation}\label{eq:defreln} g^{2n} = 1,  \quad  g^n = h^2,  \quad   h^{-1} g h = g^{-1}. \end{equation} 

Each of the nodes $\ac = 0,0',1,2,\dots,n-1,n, n'$ of the affine Dynkin diagram of
type $\hat{\mathsf{ D}}_{\mathsf{n+2}}$  (see Section \ref{subsec:Ddiag}) corresponds to  an irreducible $\DD_n$-module.  
For $\ac= 1,\dots, n-1$,  let  $\DD_n^{(\ac)}$ denote the two-dimensional $\DD_n$-module on
which the generators $g,h$ have the following matrix representations,  
\[g  \ \rightarrow \left (\begin{array}{cc} \zeta^{-\ac}  & 0 \\ 0 & \zeta^\ac\end{array}\right),  \qquad
h \  \rightarrow \left (\begin{array}{cc} 0  & i^\ac \\  i^\ac& 0 \end{array}\right),\]
relative to the basis $\{\vf_{-\ac}, \vf_\ac\}$.  For  $\ac = 0,0',n,n'$, let the one-dimensional 
$\DD_n$-module  $\DD_n^{(\ac)}$ be as follows: 
\begin{equation}  
\begin{array}{lcllll}
\DD_n^{(0)}  &=&  \CC \vf_0,   \qquad  \  &g\vf_0 = \vf_0, \quad \ & h\vf_0 = \vf_0  \\
\DD_n^{(0')} &=& \CC \vf_{0'},  \qquad   &g\vf_{0'} = \vf_{0'}, \quad &h\vf_{0'} = -\vf_{0'}  \\ 
\DD_n^{(n)}  &=& \CC \vf_n,   \qquad   &g\vf_{n} = -\vf_{n}, \quad  &h\vf_{n} = i^n \vf_{n}  \\ 
\DD_n^{(n')} &=& \CC \vf_{n'},  \quad  \  &g\vf_{n'} = -\vf_{n'}, \quad &h\vf_{n'} = -i^n\vf_{n'}.  
\end{array}
\end{equation} 
In each case,  we refer to the given basis  as the ``standard basis" for $\DD_n^{(\ac)}$.   
The modules $\DD_n^{(\ac)}$ for  $\ac= 0,0',1,2,\dots,n-1,n,n'$  give
the  complete  list  of irreducible $\DD_n$-modules up to isomorphism.    

Relative to the standard basis $\{\vf_{-1},\vf_1\}$ for the module $\VV := \DD_n^{(1)}$,
 $g$ and $h$ have the matrix realizations displayed
in \eqref{eq:defngh}, and $\VV$ is the natural $\DD_n$-module of $2 \times 1$ column vectors. \msk
 
 \begin{prop}\label{prop:Vtens} Tensor products of $\VV$ with the irreducible modules
$\DD_n^{(\ac)}$ are given as follows: 
\begin{itemize} 
\item[{\rm (a)}]  
 $\DD_n^{(\ac)} \otimes \VV   \cong   \DD_n^{(\ac-1)} \oplus  \DD_n^{(\ac+1)}$ \ \  for $1< \ac < n-1$; 
\item[{\rm (b)}]  $\DD_n^{(1)} \ot \VV  \cong  \DD_n^{(0')} \oplus  \DD_n^{(0)} \oplus \DD_n^{(2)}$;
\item[{\rm (c)}]  $\DD_n^{(n-1)} \ot \VV  \cong \DD_n^{(n-2)} \oplus \DD_n^{(n)} \oplus  \DD_n^{(n')}$;
\item[{\rm (d)}]  $\DD_n^{(0)} \ot \VV  \cong\DD_n^{(1)} =   \VV$,  \quad   $\DD_n^{(0')} \ot \VV  \cong \DD_n^{(1)} =   \VV$;
\item[{\rm (e)}]  $\DD_n^{(n)} \ot \VV  \cong \DD_n^{(n-1)}$,  \quad  $\DD_n^{(n')} \ot \VV = \DD_n^{(n-1)}$.
\end{itemize}  
   \end{prop}
 
This can be readily checked using the standard bases above. They are exactly the tensor product
rules given  by the McKay correspondence.  

Assume $\mathsf{s} = (s_1, \dots, s_k) \in \{-1,1\}^k$, and let  $\mathsf{\vert s \vert} = \left |\left \{s_j  \mid s_j= -1\right \}\right |$ as in \eqref{eq:sigvert}.  On  
the vector $\vf_{\mathsf{s}} = \vf_{s_1} \otimes \cdots \otimes \vf_{s_k} \in \VV^{\otimes k}$,  the generators
$g, h$ have the following action:
\begin{equation}\label{eq:gh}  g \vf_{\mathsf{s}} =  \zeta^{k-2\mathsf{\vert s\vert}} \vf_{\mathsf{s}}, \qquad h\vf_{\mathsf{s}} =  i^k  \vf_{\mathsf{-s}}. \end{equation}   
For two such $k$-tuples $\mathsf{r}$ and $\mathsf{s}$,  \  
\[k-2 \mathsf{\vert r\vert}\equiv k-2\mathsf{\vert s \vert}  \modd 2n    \ \iff \  \mathsf{\vert r\vert \equiv\vert s \vert}  \modd n.\] 

  Let  
\begin{eqnarray} \label{eq:lamsDn} \Lambda_k^\bullet(\DD_n) &=& \left \{\ell \in \{0,1,\dots, n\} \ \big | \  \ell \equiv k-2a_\ell\, \modd 2n  \ \ \hbox{\rm for some} \ \ a_\ell \in \{0,1,\dots, k\}\right\} \\  
\Lambda_k(\DD_n) &=& \Lambda_k^\bullet(\DD_n) \cup \left \{ \ell'  \ \big | \ \  \ell \in \Lambda_k^\bullet(\DD_n) \cap \{0,n\}\right \}.
\end{eqnarray}  {\it We will always assume $a_\ell$ is minimal with that property.}  The set $\Lambda_k(\DD_n)$ indexes the irreducible $\DD_n$-modules in $\VV^{\ot k}$. 
  In particular, 
\begin{equation}\label{eq:klcond} \hbox{\rm for} \ \ \DD_n^{(\ell)} \ \hbox{\rm to occur in } \ \VV^{\ot k} \ \ \hbox{\rm it is necessary that}  \ \ k - \ell \equiv 0 \modd 2, \end{equation}
and when $k - \ell \equiv 0 \modd 2$  holds,  then  $i^{k-\ell} = i^{\ell-k} =1$ or $-1$  depending on 
whether $k-\ell$ is $0 \modd 4$ or $2 \modd 4$.    Note also for $\ell = 0,n$ that 
$\DD_n^{(\ell')}$ occurs in $\VV^{\ot k}$ with the same multiplicity as $\DD_n^{(\ell)}$.

 \begin{subsection}{The centralizer algebra $\ZZ_k(\DD_n)$} \end{subsection}  
 
In this section,  we investigate the centralizer algebra 
 $\ZZ_k(\DD_n)= \mathsf{End}_{\DD_n}(\VV^{\otimes k})$ for $\VV = \DD_n^{(1)} = \CC \vf_{-1} \oplus \CC \vf_1$.
 The element  $g$ in \eqref{eq:defngh}   generates a cyclic subgroup $\C_{2n}$ of order $2n$, which implies that
 that $\mathsf{End}_{\DD_n}(\VV^{\otimes k}) = \ZZ_k(\DD_n) \subseteq \ZZ_k(\C_{2n}) =  \mathsf{End}_{\C_{2n}}(\VV^{\otimes k})$.   We will exploit that fact in our considerations.   
 
We impose the following order on $k$-tuples in $\{-1,1\}^k$.\msk

\begin{definition}\label{def:lexorder}  Say $\mathsf{r} \succeq \mathsf{s}$ if
  $\mathsf{\vert r \vert \leq\vert s \vert}$,  and if  $\mathsf{\vert r \vert = \vert s \vert}$,  then $\mathsf{r}$ is greater
  than or equal to $\mathsf{s}$ in the lexicographic order  coming from the relation $1 > -1$.  \end{definition}

\begin{example}  $(1,1,-1,-1,-1,1) \succ (1,-1,1,-1,1,-1)$.   \end{example} 
 
\begin{subsection}{ A basis for $\Z_k(\DD_n)$}  \end{subsection}  

We determine the dimension of $\Z_k(\DD_n)$ and a basis for it. 
It follows from Theorem \ref{T:centcycbasis}\,(a) that a basis for $\Z_k(\C_{2n})$ is given by
\begin{equation}\label{eq:cybasis} \mathcal B^k(\C_{2n}) =  \mathsf{\left \{E_{r,s}  \mid  r,s \in \{-1,1\}^k,  \  \vert r \vert   \equiv \vert{\mathsf{s}}\vert  \modd n\right \},} \end{equation}
where $\mathsf{\vert r \vert}, \mathsf{\vert s \vert} \in \{0,1,\dots, k\}$.   Now 
\begin{equation}\label{eq:rsneg} \mathsf{\vert r \vert = \vert s \vert}  \modd n \iff  \mathsf{\vert -r \vert} = k-\mathsf{\vert r \vert}  \equiv k- \mathsf{\vert s \vert
=\vert -s \vert}
 \modd n,\end{equation}
so  $\mathsf{E_{r,s}} \in \mathcal B^k(\C_{2n})$ if and only if 
$\mathsf{E}_{-\mathsf{r},-\mathsf{s}} \in \mathcal B^k(\C_{2n})$. \msk

\begin{thm}\label{thm:dicentbasis}   \begin{itemize}\item [{\rm (a)}] $\Z_k(\DD_n) = \{ X \in \Z_k(\C_{2n}) \mid  hX = Xh\}$,
where $h$ is as in \eqref{eq:defngh}.
\item [{\rm (b)}]  A basis for $\Z_k(\DD_n) =  \mathsf{End}_{\DD_n}(\VV^{\otimes k})$ is  the set 
\begin{equation}\label{eq:dibasis} \mathcal B^k(\DD_n) = \mathsf{\left \{E_{r,s} + E_{-r,-s} \,\big | \,  r,s \in \{-1,1\}^k, \  r \succ -r,  \  \vert r \vert   \equiv\vert{s}\vert  \modd n\right \}. } \end{equation}
\item[{\rm (c)}]  The dimension of $\Z_k(\DD_n)$ is given by
\begin{eqnarray}\label{eq:didim} 
\dimm \Z_k(\DD_n) &=& \half \dimm \Z_k(\C_{2n}) =  \half  \sum_{{0 \leq a,b\leq k} \atop {a \equiv b \modd n}}  {k \choose a}{k \choose b} \\
&=& \half \left(\hbox{\rm coefficient of } \ z^k  \ \hbox{\rm in} \  (1+z)^{2k} \mid_{z^n = 1} \right). \nonumber \end{eqnarray}  
\item[{\rm (d)}]     The number of walks of $2k$ steps on the affine Dynkin diagram of type 
$\hat{\mathsf D}_{n+2}$ starting and ending  at node $0$ is \  $\displaystyle{\half  \sum_{{0 \leq a,b\leq k} \atop {a \equiv b \modd n}}  {k \choose a}{k \choose b}}$.
\end{itemize} 
\end{thm}  \smallskip
 
\begin{proof}  Since $\ZZ_k(\DD_n) \subseteq \ZZ_k(\C_{2n})$, we may assume that $X \in \Z_k(\DD_n)$ can be written
as \[X = \sum_{\mathsf {\vert r \vert = \vert s \vert} \modd n}  X_{\mathsf{r,s}} \mathsf{E_{r,s}},\] 
and that $X$ commutes with the generator $g$ of $\DD_n$  in  \eqref{eq:defngh}.   In order for $X$ to belong to  $\ZZ_k(\DD_n)$,
$hX = Xh$ must hold, as asserted in (a).      Applying both sides of  $hX = Xh$  to $\vf_{\mathsf{s}}$, we obtain 
\[\sum_{\mathsf {\vert r \vert = \vert s\vert} \modd n}    i^k X_{\mathsf{r,s}} \vf_{\mathsf{-r}}  =\sum_{\mathsf {\vert t \vert = \vert -s \vert} \modd n}    i^k X_{\mathsf{t,-s}}\vf_{\mathsf{t}}.\]    
The coefficient $i^k X_{\mathsf{r,s}}$ of $\vf_{\mathsf{-r}}$ on the left  is nonzero  if and only if 
the coefficient $i^k X_{\mathsf{-r,-s}}$ of  $\vf_{\mathsf{-r}}$ on the right is nonzero, and they are equal.
Hence,  $Xh = hX$ if and only if $X_{\mathsf{-r,-s}} = X_{\mathsf{r,s}}$ for all $\mathsf {\vert r \vert = \vert s \vert} \modd n.$ 
Therefore,  
we have  
\[X =  \sum_{{\mathsf {\vert r \vert = \vert s \vert} \modd n} \atop {\mathsf{r} \succ \mathsf{-r}}}   X_{\mathsf{r,s}}( \mathsf{E_{r,s} + E_{-r,-s}}).\]
  Thus,  the set $\mathcal B^k(\DD_n)$ in \eqref{eq:dibasis}  spans $\ZZ_k({\DD_n})$, and
since it is clearly linearly independent,   it is a basis for $\ZZ_k({\DD_n})$.  

Part (c)  is apparent from \eqref{eq:cybasis},  part (b), and  Theorem \ref{T:centcycbasis}\,(c),  which says that 
 $\dimm \Z_k(\C_{2n})$  is the coefficient of
 $z^k$ in $(1+z)^{2k} \mid_{z^n = 1}$.    Part (d) follows directly from (c) and \eqref{eq:evid}.
 \end{proof}   
 
\begin{example}\label{ex:dih2} {\rm Assume $k = 4$ and $n=5$.   Then    
$\dimm \Z_k(\DD_n)$ is $\half$ the coefficient of $z^4$ in 
\[(1+z)^8\mid_{z^5 = 1}  = 1 + 8z + 28 z^2 + 56 z^3 + 70 z^4 + 56 + 28 z + 8 z^2 + z^3,\] 
so that $\dimm \Z_{4}(\DD_5) = \half \cdot 70 = 35$.   Since $z^4$ appears only once in  
$(1+z)^8\mid_{z^n = 1}$ for $n \geq 5$,  in fact  $\dimm \Z_{4}(\DD_n) = 35$ for all $n \geq 5$.  

Now when $n=4=k$,     
$\dimm \Z_k(\DD_n)$ is $\half$ the coefficient of $z^4 = z^0 = 1$ in 
\[(1+z)^8\mid_{z^4 = 1}  = 1 + 8z + 28 z^2 + 56 z^3 + 70 + 56z + 28 z^2 + 8 z^3 + 1.\] 
Thus,  $ \dimm \mathsf{Z}_4(\DD_4) = \half(1+70+1) = 36$.} 
\end{example} 

\begin{subsection}{Copies of $\DD_n^{(\ac)}$  in $\VV^{\ot k}$ for $\ac \in \{1,\dots, n-1\}$}
\end{subsection}   

The next result locates copies of  $\DD_n^{(\ac)}$ inside $\VV^{\ot k}$ when  $1\leq \ac \leq n-1$.\msk
 
\begin{thm}\label{thm:mj} Assume $\ell \in \Lambda_k(\DD_n)$ and $1 \leq \ac \leq n-1$.   Set
\begin{equation}\label{eq:Kdef} {\kl}_\ac  =  \left \{ \mathsf{r}\in \{-1,1\}^k  \ \Big  | \   
k-2 \mathsf{\vert r \vert} \equiv \ac \modd 2n,  \  \mathsf{r \succ -r}  \right \}.  \end{equation}
\begin{itemize} 
\item[{\rm (i)}]  For each $\mathsf{r} \in {\kl}_{\ac}$,  the vectors 
$i^{\ac-k} \vf_{\mathsf {-r}}, \, \vf_{\mathsf {r}}$  determine  a standard basis for a copy of $\DD_n^{(\ac)}$.  
\item[{\rm (ii)}]   For  $\mathsf{r,s} \in {\kl}_{\ac}$, the transformation
$\mathsf{e_{r,s}} := \mathsf{E_{r,s} + E_{-r,-s}}\in \Z_k(\DD_n) =  \mathsf{End}_{\DD_n}(\VV^{\otimes k})$ satisfies 
\[ \mathsf{e_{r,s}}(\vf_{\mathsf t}) = \delta_{\mathsf{s,t}} \vf_{\mathsf{r}}, \qquad
\mathsf{e_{r,s}}(i^{\ac-k} \vf_{\mathsf {-t}})  =   \delta_{\mathsf{s,t}}  i^{\ac-k} \vf_{\mathsf {-r}}.\]
\item[{\rm (iii)}]  For  $\mathsf{r,s,t,u} \in{\kl}_{\ac}$, \
$\mathsf{e_{r,s}} \mathsf{e_{t,u}} = \delta_{\mathsf{s,t}} \mathsf{e_{r,u}}$ holds, so 
$\spann_{\mathbb C}\{\mathsf{e_{r,s}} \mid \mathsf{r,s} \in{\kl}_{\ac}\}$ can be 
identified with the matrix algebra $\mathcal M_{d_\ac}(\mathbb C)$, where $d_\ac  = |{\kl}_\ac |$.
\item[{\rm (iv)}]   $\mathsf{B}_{k}^{(\ac)}: = \spann_{\mathbb C}\{i^{\ac-k}\vf_{\mathsf{-r}}, \vf_{\mathsf r} \mid  \mathsf{r} \in {\kl}_{\ac}\}$
is an irreducible bimodule for  $\DD_n \times \mathcal M_{d_{\ac}}(\mathbb C)$.  
\end{itemize} 
\end{thm}  

\begin{proof} (i)  Since  $\mathsf{r} \in \kl_\ac$ satisfies   $k-2 \mathsf {\vert r\vert}  \equiv \ac \modd 2n$, we have
 for the generators $g,h$ in \eqref{eq:defngh},
\begin{eqnarray*} &&g (i^{\ac-k}\vf_{\mathsf{-r}})  = \zeta^{-\ac} ( i^{\ac-k} \vf_{\mathsf{-r}})\quad  \hbox{and} \quad
g \vf_{\mathsf r} =  \zeta^\ac \vf_{\mathsf r}, \\ 
&& h (i^{\ac-k} \vf_{\mathsf{-r}})   =  i^\ac  \vf_{\mathsf r} 
\quad  \hbox{and} \quad
h \vf_{\mathsf r} =  i^k  \vf_{\mathsf {-r}} =  i^\ac(i^{k-\ac}\vf_{\mathsf {-r}}) =
i^\ac (i^{\ac-k} \vf_{\mathsf {-r}}),  \end{eqnarray*}  
so $\{ i^{\ac-k}  \vf_{\mathsf {-r}}, \vf_{\mathsf r}\}$  forms a standard basis for a copy of  $\DD_n^{(\ac)}$. 

Part (ii) can be verified by direct computation.  

For (iii) note that if $\mathsf{s} \in {\kl}_\ell$, then $\mathsf{-s} \not \in  {\kl}_\ell$, as otherwise $-\ell = k-2\mathsf{\vert -s\vert}  \equiv \ell \modd 2n$,
which is impossible for $\ell \in \{1,\dots, n-1\}$.   Thus, for $\mathsf{r,s,t,u \in \kl_\ell}$,  
\[\mathsf{e_{r,s} e_{t,u} = \delta_{s,t} e_{r,u} + \delta_{-s,t}e_{-r,u} = \delta_{s,t}e_{r,u}.}\]

For  (iv),  assume $\mathsf{S} \neq 0$ is a sub-bimodule of $\mathsf B_k^{(\ac)}$, and $0 \neq u = \sum_{\mathsf{r}  \in {\kl}_\ac} \left(  \gamma_{-\mathsf{r}} i^{\ac-k}\vf_{-\mathsf{r}} + \gamma_{\mathsf{r}}\vf_{\mathsf r} \right)  \in \mathsf{S}$.  
Then since $k-2\mathsf{\vert r \vert}  \equiv \ac \modd 2n$,  \  $gu - \zeta^{-\ac}u = (\zeta^\ac - \zeta^{-\ac}) \sum_{\mathsf{r}  \in {\kl}_\ac}   \gamma_{\mathsf{r}}\vf_{\mathsf r} \in \mathsf{S}$.  As  $\zeta^\ac - \zeta^{-\ac} \neq 0$ for $\ac=1,\dots, n-1$,  we have that
$w: = \sum_{\mathsf{r}  \in {\kl}_\ac}   \gamma_{\mathsf{r}} \vf_{\mathsf r} \in \mathsf{S}$.
If  $\gamma_{\mathsf{u}} \neq 0$ for some $\mathsf{u} \in {\kl}_\ac$,
then  $\mathsf{e_{t,u}}w =  \gamma_{\mathsf u} \vf_{\mathsf t} \in \mathsf{S}$
for all  $\mathsf{t} \in {\kl}_\ac$.   But then $h \vf_{\mathsf{t}}  = i^k \vf_{-\mathsf{t}} \in
\mathsf{S}$ for all such $\mathsf{t}$ as well.  This implies $\mathsf{S} = \mathsf{B}_{k}^{(\ac)}$.
If instead $\gamma_{\mathsf{u}} = 0$
for all $\mathsf{u} \in {\kl}_\ac$, then 
$u = \sum_{\mathsf{r}  \in {\kl}_\ac}  \gamma_{-\mathsf{r}} i^{\ac-k} \vf_{-\mathsf{r}}$.
There is some   $\gamma_{\mathsf{-s}} \neq 0$, and applying
$\mathsf{e_{t,s}}$ to $u$ shows that $\vf_{-\mathsf{t}} \in \mathsf{S}$ for
all $\mathsf{t}  \in {\kl}_\ac$.  Applying $h$ to those elements shows that
$\vf_{\mathsf{t}} \in \mathsf{S}$ for all $\mathsf{t} \in {\kl}_\ac$.  Hence
$\mathsf{S} = \mathsf{B}_{k}^{(\ac)}$ in this case also,  and  $\mathsf{B}_{k}^{(\ac)}$ must be irreducible.  \end{proof} 

\begin{subsection}{Copies of $\DD_n^{(\ac)}$ in $\VV^{\ot k}$ for $\ac \in \{0,0',n,n'\}$}

\end{subsection}  Throughout this section,  let  $\ac= 0,n$ and $\ac' = 0',n'$.  
Observe that if $k - 2\mathsf{\vert r \vert} \equiv \ac \modd 2n$, for $\ac= 0,n$,
 then $k - 2\mathsf{\vert -r \vert} \equiv -\ac \equiv \ac \modd 2n$.  \msk    
 
\begin{thm}\label{thm:m0n} Assume $\ac= 0$ or $n$ and $\ell \in \Lambda_k(\DD_n)$.   Let
\begin{equation}\label{eq:K0ndef} \ {\kl}_\ac 
 = \left \{ \mathsf{r}\in \{-1,1\}^k \ \Big  | \   
k-2\mathsf{\vert r \vert}  \equiv \ac \  \modd 2n, \ \mathsf{r \succ -r} \right \} \end{equation} 
\begin{itemize}
\item[{\rm (i)}]  $\CC(\vf_{\mathsf r}+ i^{\ac-k} \vf_{\mathsf {-r}}) \cong \DD_n^{(\ac)}$  and   $\CC (\vf_{\mathsf r}- i^{\ac-k} \vf_{\mathsf {-r}})\cong \DD_n^{(\ac')}$  for each $\mathsf{r} \in {\kl}_\ac$.  
\item[{\rm (ii)}]   For  $\mathsf{r,s} \in {\kl}_\ac$, the transformations 
\[\mathsf{e^{\pm}_{r,s}} := \frac{1}{2}\Big((\mathsf{E_{r,s} + E_{-r,-s}}) \pm i^{\ac-k}
(\mathsf{E_{r,-s} + E_{-r,s}})\Big) \in \Z_k(\DD_n)\] satisfy 
\begin{eqnarray*} \mathsf{e^{\pm}_{r,s}}(\vf_{\mathsf t} \pm i^{\ac-k} \vf_{\mathsf {-t}}) &=& 
\delta_{\mathsf{s,t}}(\vf_{\mathsf r} \pm i^{\ac-k} \vf_{\mathsf {-r}}) \\
\mathsf{e^{\pm}_{r,s}}(\vf_{\mathsf t} \mp i^{\ac-k} \vf_{\mathsf {-t}}) &=& 
0. \end{eqnarray*}  
\item[{\rm (iii)}]  For  $\mathsf{r,s,t,u} \in {\kl}_\ac$,  $\mathsf{e^{\pm}_{r,s}}  \mathsf{e^{\pm}_{t,u}} = \delta_{\mathsf{s,t}} \mathsf{e^{\pm}_{r,u}}$ and  $\mathsf{e^{\pm}_{r,s}}  \mathsf{e^{\mp}_{t,u}} = 0$  hold. Therefore
$\spann_{\mathbb C}\{\mathsf{e^\pm_{r,s}}\}$ can be identified with the matrix algebra  $\mathcal M_{d_\ac}(\mathbb C)^{\pm}$,  where $d_\ac = |{\kl}_\ac |$.  
\item[{\rm (iv)}]   $\mathsf{B}_{k}^{(\ac)}: = \mathsf{span}_{\mathbb C}\{\vf_{\mathsf r}+ i^{\ac-k} \vf_{\mathsf{-r}} \mid  \mathsf{r} \in {\kl}_\ac\}$ and $\mathsf{B}_{k}^{(\ac')}: = \mathsf{span}_{\mathbb C}\{\vf_{\mathsf r}- i^{\ac-k} \vf_{\mathsf{-r}} \mid  \mathsf{r} \in {\kl}_\ac\}$
are irreducible bimodules for $\DD_n \times \mathcal M_{d_\ac}(\mathbb C)^\pm$.  

\end{itemize} 
\end{thm}

\begin{proof}  Part (i) follows directly from \eqref{eq:gh}, and parts (ii) and (iii) can be verified by 
direct computation using the fact that if $\mathsf{r,s} \in {\kl}_\ac$,  then it cannot be the case that $\mathsf{r = -s}$.
Indeed, if  $\mathsf{r = -s}$,  then $\mathsf{-s = r \succ -r = s}$,  which contradicts that $\mathsf{s}$ belongs to ${\kl}_\ac$.   Part (iv)  can be deduced from the earlier parts  and the
fact that the maps $\mathsf{e^\pm_{r,s}}$ belong to $\Z_k(\DD_n)$, hence commute with
the $\DD_n$-action.  The argument follows the proof of  (iv) of Theorem \ref{thm:mj}
and is left to the reader. 
\end{proof} 

\begin{cor}\label{cor:matrixbasis} \begin{itemize} \item[{\rm (a)}]  The following set is a basis for $\Z_k(\DD_n)$ which gives its decomposition
into matrix summands:
\begin{equation} {\mathcal B}^k_{\mathsf{mat}}(\DD_n): = \bigcup_{\ac =1,\dots,n-1} \left \{ \mathsf{e_{r,s}} \mid \mathsf{r,s}
\in {\kl}_\ac\right \} \ \cup \  \bigcup_{\ac =0,n}\left \{ \mathsf{e^\pm_{t,u}} \mid \mathsf{t,u}
\in {\kl}_\ac\right \}, \end{equation} where  $\ell \in \Lambda_k^\bullet(\DD_n)$ {\rm (}see \eqref{eq:lamsDn}{\rm )};   
$\mathsf{e_{r,s}} =
\mathsf{E_{r,s} + E_{-r,-s}}$ for $\mathsf{r,s} \in {\kl}_\ac, \ \ac =1,\dots,n-1$, and ${\kl}_\ac$ is as in \eqref{eq:Kdef};  and 
$\mathsf{e^{\pm}_{t,u}} := \frac{1}{2}\Big((\mathsf{E_{t,u} + E_{-t,-u}}) \pm i^{\ac-k}
(\mathsf{E_{t,-u} + E_{-t,u}})\Big)$ for $\mathsf{t,u} \in {\kl}_\ac$,  $\ac= 0,n$,  and ${\kl}_\ac$ is as in \eqref{eq:K0ndef}.
\item[{\rm (b)}]  $\left \{z_\ell \mid  \ell \in \{1,\dots, n-1\}\right \} \cup \left \{z_{\ell}^{\pm} \mid \ell = 0,n \right\}$, \ $\ell \in \Lambda_k^{\bullet}(\DD_n)$,   is a 
basis for the center of $\Z_k(\DD_n)$,  where
\begin{eqnarray} z_\ell \  &=& \sum_{\mathsf{r} \in \kl_{\ac}, \ \mathsf{r \succ r'}}  \mathsf{e_{r,r}}\qquad \hbox{\rm for} \ \ \ell \in \{1,\dots, n-1\}, \\ 
z_\ell^\pm &=& \ \  \sum_{\mathsf{r} \in \kl_{\ac}}  \mathsf{e_{r,r}^\pm}
\qquad \hbox{\rm for} \ \ \ell \in \{0,n\}. \end{eqnarray} 
\end{itemize} 
\end{cor}
 
\begin{subsection}{Example $\VV^{\ot 4}$} \end{subsection}
 
 \begin{example}\label{ex:di1}{\rm In this example, we decompose $\VV^{\otimes 4}$  for  $n \geq 4$.  
 
 {\it First, suppose $n \geq 5$.} \ \    Then using the results in Proposition \ref{prop:Vtens},  we have  
 \[\VV^{\ot 4}  =  3 \cdot \DD_n^{(0')} \oplus 3 \cdot \DD_n^{(0)} \oplus 4 \cdot \DD_n^{(2)} \oplus \DD_n^{(4)}.\]
 Correspondingly,  $\ZZ_k({\DD_n})$ decomposes into matrix blocks according to 
 \[\ZZ_k(\DD_n)\cong \mathcal M_{3}(\mathbb C) \oplus \mathcal M_{3}(\mathbb C) \oplus
 \mathcal M_{4}(\mathbb C) \oplus  \mathcal M_{1}(\mathbb C),\] and $\dimm \Z_k(\DD_n) = 3^2 + 3^2 + 4^2 + 1^2 = 35$
 as in Example \ref{ex:dih2}.   More explicitly we have for $n \geq 5$ the following: 
\begin{itemize}
\item[{\rm (i)}]  Let $\ac=k=4$,  and set  $\mathsf{q} = (1,1,1,1)$.   Then $i^{\ac-k} = 1$, so  $\{\vf_{\mathsf {-q}}, \vf_{\mathsf q}\}$ gives a standard basis for a copy of $\DD_n^{(4)}$,
and $\{\mathsf{e_{q,q}}:= \mathsf{E_{q,q}} + \mathsf{E_{-q,-q}}\}$ is a basis for  $\mathcal M_{1}(\mathbb C)$.  
\item[{\rm (ii)}]   Let $\ac =2$, and set  $\mathsf{r} = (1,1,1,-1)$,  $\mathsf{s} = (1,1,-1,1)$, $\mathsf{t} = (1,-1,1,1)$, and $\mathsf{u} = (-1,1,1,1)$.
Then the pair  $\{-\vf_{{-\lambda}}, \vf_{{ \lambda}}\}$ for ${\lambda} = \mathsf{r,s,t,u}$ determines a standard basis
for a copy of $\DD_n^{(2)}$. The maps  $\mathsf{e}_{{\lambda,\mu}} := \mathsf{E_{{\lambda,\mu}} +
E_{{-\lambda,-\mu}}}$ for ${\lambda, \mu} \in \{ \mathsf{r,s,t,u}\}$ give a matrix unit basis for $\mathcal M_4(\mathbb C)$.
\item[{\rm (iii)}]  Let $\ac =0$, and set   $\mathsf{v} = (1,1,-1,-1)$,  $\mathsf{w} = (1,-1,1,-1)$, and $\mathsf{x} = (1,-1,-1,1)$.  Then since $i^{\ac-k} = 1$,  
 the vector $\vf_{ \lambda} - \vf_{-\lambda}$
gives a basis for a copy of $\DD_n^{(0')}$ for $\lambda \in \{\mathsf{v,w,x}\}$.     The transformations $\mathsf{e_{\lambda,\mu}^- = \frac{1}{2}\Big( (\mathsf{E_{{\lambda,\mu}} +
E_{{-\lambda,-\mu}}}) - (E_{{\lambda,-\mu}}}
+ \mathsf{E_{{-\lambda,\mu}}}) \Big)$ for ${ \lambda,\mu} \in \{\mathsf{v,w,x}\}$ form a matrix
unit basis for $\mathcal M_3(\mathbb C)^- \cong \mathcal M_3(\mathbb C)$ (in the notation of Theorem \ref{thm:m0n}).  
Similarly,  the vector $\vf_{ \lambda} + \vf_{- \lambda}$
gives a basis for a copy of $\DD_n^{(0)}$, and the maps 
 $\mathsf{e_{\lambda,\mu}^+ = \frac{1}{2}\Big((\mathsf{E_{{\lambda,\mu}} +
E_{{-\lambda,-\mu}}})+(E_{{\lambda,-\mu}}}
+ \mathsf{E_{{-\lambda,\mu}}})\Big)$ for ${ \lambda,\mu} \in \{\mathsf{v,w,x}\}$ form a matrix
unit basis for $\mathcal M_3(\mathbb C)^+ \cong \mathcal M_3(\mathbb C)$.  
  
\end{itemize}  In (i) and (ii),  the space $\mathsf {B}_{4}^{(\ac)} = \spann_\mathbb C\{\vf_{\pm \lambda}\}$, as $\lambda$ ranges over the appropriate indices, is an irreducible  bimodule for $\DD_n \times \ZZ_k(\DD_n)$.  
In {\rm (iii)}, $\mathsf {B}_{4}^{(0)} = \spann_\mathbb C\{\vf_\lambda + \vf_{- \lambda}\}$ and 
 $\mathsf {B}_{4}^{(0')} = \spann_\mathbb C\{\vf_\lambda - \vf_{- \lambda}\}$
are irreducible  $(\DD_n \times \ZZ_k(\DD_n))$-bimodules in $\VV^{\ot 4}$. Therefore,
as a  $\DD_n \times \ZZ_k(\DD_n)$-bimodule  $\VV^{\ot 4}$ decomposes into irreducible bimodules  of dimensions $2,8,3,3$.

{\it Suppose now that $n = 4$.} \ \   Then $\VV^{\ot 4}  = 3 \cdot \DD_4^{(0')} \oplus 3 \cdot \DD_4^{(0)} \oplus 4 \cdot \DD_4^{(2)} \oplus \DD_4^{(4)} \oplus  \DD_4^{(4')},$  and  $\ZZ_k(\DD_4) \cong \mathcal M_{3}(\mathbb C) \oplus \mathcal M_{3}(\mathbb C) \oplus
 \mathcal M_{4}(\mathbb C) \oplus  \mathcal M_{1}(\mathbb C) \oplus  \mathcal M_{1}(\mathbb C)$, which has
 dimension 36 (compare Example \ref{ex:dih2}). 
The only change is in {\rm (i)},   where $\{\vf_{\mathsf q} -\vf_{-\mathsf q}\}$ is a basis for a copy of   $\DD_4^{(4')}$
and $\{\vf_{\mathsf q}+\vf_{-\mathsf q}\}$ is a basis for a copy of  $ \DD_4^{(4)}$.  In the first case,
$\mathsf{e^-_{q,q}} = \frac{1}{2}\Big((\mathsf{E_{q,q}+E_{-q,-q}}) - (\mathsf{E_{q,-q}+E_{-q,q}})\Big)$ gives a basis for $\mathcal M_1(\mathbb C)^-$, while in the 
second,  $\mathsf{e^+_{q,q}} = \frac{1}{2}\Big((\mathsf{E_{q,q}+E_{-q,-q}})+(\mathsf{E_{q,-q}+E_{-q,q}})\Big)$
for $\mathcal M_1(\mathbb C)^+$.  } \end{example}

\begin{subsection}{A diagram basis for $\DD_n$} \end{subsection}
The basis $\{\mathsf{E_{r,s}+E_{-r,-s} \mid r,s\in \{-1,1\}^k, r \succ -r,  \vert r \vert \equiv \vert s\vert} \modd n\}$ of matrix units can be viewed diagrammatically.  
If $\rf = (r_1, \ldots, r_k) \in \{-1,1\}^k$ and $\sm = (s_{1'}, \ldots, s_{k'}) \in \{-1,1\}^k$, let $d_\mathsf{r,s}$ be the diagram on two rows of $k$ vertices labeled by $1, 2, \ldots, k$ on top and $1', 2', \ldots, k'$ on bottom.  Mark the $i$th vertex on top with ${\bf +}$ if $r_i = 1$ and with ${\bf -}$ if $r_i = -1$. Similarly, mark the $j'$th vertex on the bottom with ${\bf +}$ if $s_j =  1$ and with ${\bf -}$ if $s_j = -1$.  The positive and negative vertices correspond to a set partition of $\{1,\dots, k, 1',\dots, k'\}$ into two parts.  Draw edges in the diagram so that the ${\bf +}$ vertices form a connected component and the ${\bf -}$ vertices form a connected component.  For example, if 
\[\rf = (-1,-1,1,-1,-1,1,1,-1) \quad \hbox{\rm and} \quad \sm = (1,-1,-1,-1,1,-1,1, -1),\]  then  the corresponding diagram and set partition are
\[
\begin{array}{c} \begin{tikzpicture}[scale=.65,line width=1pt]
\foreach \i in {1,...,8}  { \path (\i,1) coordinate (T\i); \path (\i,-1)
coordinate (B\i); }
\filldraw[fill= black!10,draw=black!10,line width=4pt]  (T1) -- (T8) --
(B8) -- (B1) -- (T1);
\draw[black,line width = 2 pt,fill=black!60, fill opacity=0.5]   
(B1) -- (T3) .. controls +(-90:0.5) and +(-90:0.5) ..  
(T6) .. controls +(-90:0.5) and +(-90:0.5) .. (T7)  -- 
(B7).. controls +(90:0.5) and +(90:0.75) ..  (B5)
.. controls +(90:0.75) and +(90:0.5) ..  (B1);
\draw[black,line width = 2 pt,fill=black!60,fill opacity=0.5]   
(T1) .. controls +(-90:1.00) and +(-90:1.00) ..  
(T2) .. controls +(-90:1.00) and +(-90:1.00) .. 
(T4) .. controls +(-90:1.00) and +(-90:1.00) .. 
(T5)  .. controls +(-90:1.15) and +(-90:1.15) .. 
(T8) -- 
(B8)   .. controls +(90:1.00) and +(90:1.00) .. 
(B6) .. controls +(90:1.00) and +(90:1.00) .. 
(B4) .. controls +(90:1.00) and +(90:1.00) .. 
(B3) .. controls +(90:1.00) and +(90:1.00) .. 
(B2) .. controls +(90:1.00) and +(-90:1.00) .. (T1);
\foreach \i in {1,...,8}  { \fill (T\i) circle (6pt); \fill (B\i) circle
(6pt); }
\draw  (T1)  node[above=0.1cm]{$\scriptstyle{1}$};\draw  (B1)  node[below=0.1cm]{$\scriptstyle{1'}$};
\draw  (T2)  node[above=0.1cm]{$\scriptstyle{2}$};\draw  (B2)  node[below=0.1cm]{$\scriptstyle{2'}$};
\draw  (T3)  node[above=0.1cm]{$\scriptstyle{3}$};\draw  (B3)  node[below=0.1cm]{$\scriptstyle{3'}$};
\draw  (T4)  node[above=0.1cm]{$\scriptstyle{4}$};\draw  (B4)  node[below=0.1cm]{$\scriptstyle{4'}$};
\draw  (T5)  node[above=0.1cm]{$\scriptstyle{5}$};\draw  (B5)  node[below=0.1cm]{$\scriptstyle{5'}$};
\draw  (T6)  node[above=0.1cm]{$\scriptstyle{6}$};\draw  (B6)  node[below=0.1cm]{$\scriptstyle{6'}$};
\draw  (T7)  node[above=0.1cm]{$\scriptstyle{7}$};\draw  (B7)  node[below=0.1cm]{$\scriptstyle{7'}$};
\draw  (T8)  node[above=0.1cm]{$\scriptstyle{8}$};\draw  (B8)  node[below=0.1cm]{$\scriptstyle{8'}$};
\draw (-.25,0) node {$\mathsf{d}_{\rf,\sm}=$};
\draw (T1) node[white] {\bf -}; \draw (T2) node[white] {\bf -}; \draw (T3)
node[white] {\bf +}; \draw (T4) node[white] {\bf -}; \draw (T5) node[white]
{\bf -}; \draw (T6) node[white] {\bf +}; \draw (T7) node[white] {\bf +};
\draw (T8) node[white] {\bf -};\draw (B1) node[white] {\bf +};\draw (B2) node[white] {\bf -};\draw (B3)
node[white] {\bf -};\draw (B4) node[white] {\bf -};\draw (B5) node[white]
{\bf +};\draw (B6) node[white] {\bf -};
\draw (B7) node[white] {\bf +};\draw (B8) node[white] {\bf -}; 
\end{tikzpicture}  \end{array}
= \bigg\{ \{1,2,4,5,8,2',3',4',6',8'\}, \{3,6,7,1',5',7'\}\bigg\}.
\]
Two diagrams are equivalent if they correspond to the same underlying set partition.  Furthermore, switching the $+$ signs  and $-$ signs give the same diagram, since
$\mathsf{d}_{\rf,\sm} = \mathsf{d}_{-\rf,-\sm}$.

If $d$ is any diagram corresponding to a set partition of $\{1, \ldots, k, 1', \ldots, k'\}$ into two parts, then $d$ acts on  $\VV^{\otimes k}$ by $d.\v_{\mathsf u} = 
\sum_{\mathsf t \in \{-1,1\}^k}  d^{\mathsf{t}}_{\mathsf{u}} \, \v_{\mathsf t}$ for $\mathsf{t} = (u_1,\dots, u_k)$ and
$\mathsf{u} = (u_{1'},\dots, u_{k'}) \in \{-1,1\}^k$ with
\begin{equation}
d^{\mathsf{t}}_{\mathsf u} 
= \begin{cases}
1 &\quad \hbox{if $u_a = u_b$  \hbox{\it if and only if} $a$ and $b$ are in the same block of $d$,}
\\ 
0 &\quad \hbox{otherwise}
\end{cases}
\end{equation}
for  $a,b \in \{1,\dots,k, 1', \dots, k'\}$.  
In this notation $d_{\mathsf{u}}^{\mathsf{t}}$ denotes the $({\mathsf{t}},{\mathsf{u}})$-entry of the matrix of $d$ on $\VV^{\otimes k}$ with respect to the basis of simple tensors.
 It follows from this construction that
$  \mathsf{e}_\mathsf{r,s}=\mathsf{E}_{\rf,\sm} + \mathsf{E}_{-\rf,-\sm}$  and  the set partition diagram $ \mathsf{d}_{\rf,\sm}
$  have the same action on $\VV^{\ot k}$ and so are equal. 
Note that in the example above, we have
\[
\mathsf{d}_{\rf,\sm}.(\v_a \otimes \v_b \otimes \v_b \otimes \v_b \otimes \v_a \otimes \v_b \otimes \v_a \otimes \v_b) =
\v_b \otimes \v_b \otimes \v_a \otimes \v_b \otimes \v_b \otimes \v_a \otimes \v_a \otimes \v_b,
\]
for $\{a,b\} = \{-1,1\}$ and $\mathsf{d}_{\rf,\sm}$ acts as 0 on all other simple tensors.

The  diagrams multiply as matrix units.  Let $b(d)$ and $t(d)$ denote the set partitions imposed by $d$ on the bottom and top rows of $d$, respectively. If $d_1$ and $d_2$ are diagrams,  then
\[
d_1 d_2 = \delta_{b(d_1),t(d_2)} d_3,
\]
where $d_3$ is the  diagram given by placing $d_1$ on top of $d_2$, identifying $b(d_1)$ with $t(d_2)$, and taking $d_3$ to be the connected components of the resulting diagram.  For example,
\[
\begin{array}{c}
\begin{tikzpicture}[scale=.6,line width=1pt]
\foreach \i in {1,...,8}  { \path (\i,1) coordinate (T\i); \path (\i,-1)
coordinate (B\i); }
\filldraw[fill= black!10,draw=black!10,line width=4pt]  (T1) -- (T8) --
(B8) -- (B1) -- (T1);
\draw[black,line width = 2 pt,fill=black!60, fill opacity=0.5]   
(B1) -- (T3) .. controls +(-90:0.5) and +(-90:0.5) ..  
(T6) .. controls +(-90:0.5) and +(-90:0.5) .. (T7)  -- 
(B7).. controls +(90:0.5) and +(90:0.75) ..  (B5)
.. controls +(90:0.75) and +(90:0.5) ..  (B1);
\draw[black,line width = 2 pt,fill=black!60,fill opacity=0.5]   
(T1) .. controls +(-90:1.00) and +(-90:1.00) ..  
(T2) .. controls +(-90:1.00) and +(-90:1.00) .. 
(T4) .. controls +(-90:1.00) and +(-90:1.00) .. 
(T5)  .. controls +(-90:1.15) and +(-90:1.15) .. 
(T8) -- 
(B8)   .. controls +(90:1.00) and +(90:1.00) .. 
(B6) .. controls +(90:1.00) and +(90:1.00) .. 
(B4) .. controls +(90:1.00) and +(90:1.00) .. 
(B3) .. controls +(90:1.00) and +(90:1.00) .. 
(B2) .. controls +(90:1.00) and +(-90:1.00) .. (T1);
\foreach \i in {1,...,8}  { \fill (T\i) circle (6pt); \fill (B\i) circle
(6pt); }
\draw (-.25,0) node {$d_1=$};
\draw (T1) node[white] {\bf -}; \draw (T2) node[white] {\bf -}; \draw (T3)
node[white] {\bf +}; \draw (T4) node[white] {\bf -}; \draw (T5) node[white]
{\bf -}; \draw (T6) node[white] {\bf +}; \draw (T7) node[white] {\bf +};
\draw (T8) node[white] {\bf -};\draw (B1) node[white] {\bf +};\draw (B2) node[white] {\bf -};\draw (B3)
node[white] {\bf -};\draw (B4) node[white] {\bf -};\draw (B5) node[white]
{\bf +};\draw (B6) node[white] {\bf -};
\draw (B7) node[white] {\bf +};\draw (B8) node[white] {\bf -}; 
\end{tikzpicture}  \\
 \begin{tikzpicture}[scale=.6,line width=1pt]
\foreach \i in {1,...,8}  { \path (\i,1) coordinate (T\i); \path (\i,-1)
coordinate (B\i); }
\filldraw[fill= black!10,draw=black!10,line width=4pt]  (T1) -- (T8) --
(B8) -- (B1) -- (T1);
\draw[black,line width = 2 pt,fill=black!60,fill opacity = 0.5]
(T1)-- (B1) .. controls +(90:0.85) and +(90:0.85) .. (B2).. controls +(90:0.85) and +(90:0.85) .. (B4)
.. controls +(90:0.5) and +(90:0.5) .. (B5) .. controls +(90:0.85) and +(90:0.85) .. (B7)
.. controls +(90:0.5) and +(90:0.5) .. (B8) .. controls +(90:0.85) and +(-90:0.85) .. (T7)
.. controls +(-90:0.5) and +(-90:0.5) .. (T5) .. controls +(-90:1.2) and +(-90:1.2) .. (T1);
\draw[black,line width = 2 pt,fill=black!60,fill opacity = 0.5]
(B3) .. controls +(+90:0.85) and +(+90:0.85) .. (B6)  .. controls +(45:0.85)
and +(-90:0.85) .. (T8)   .. controls
+(-90:0.85) and +(-90:1.20) .. (T6)
.. controls +(-90:1.20) and +(-90:0.85) .. (T4) .. controls +(-90:0.5) and
+(-90:0.5) .. (T3) .. controls +(-90:0.5) and +(-90:0.5) .. (T2) .. controls
+(-90:0.85) and +(+90:0.85) .. (B3);
\foreach \i in {1,...,8}  { \fill (T\i) circle (6pt); \fill (B\i) circle
(6pt); }
\draw (-.25,0) node {$d_2=$};
\draw (T1) node[white] {\bf -}; \draw (T2) node[white] {\bf +}; \draw (T3)
node[white] {\bf +}; \draw (T4) node[white] {\bf +}; \draw (T5) node[white]
{\bf -}; \draw (T6) node[white] {\bf +}; \draw (T7) node[white] {\bf -};
\draw (T8) node[white] {\bf +};\draw (B1) node[white] {\bf -};\draw (B2) node[white] {\bf -};\draw (B3)
node[white] {\bf +};\draw (B4) node[white] {\bf -};\draw (B5) node[white]
{\bf -};\draw (B6) node[white] {\bf +};
\draw (B7) node[white] {\bf -};\draw (B8) node[white] {\bf -}; 
\end{tikzpicture}
\end{array}
=
\begin{array}{c}
 \begin{tikzpicture}[scale=.6,line width=1pt]
\foreach \i in {1,...,8}  { \path (\i,1) coordinate (T\i); \path (\i,-1)
coordinate (B\i); }
\filldraw[fill= black!10,draw=black!10,line width=4pt]  (T1) -- (T8) --
(B8) -- (B1) -- (T1);
\draw[black,line width = 2 pt,fill=black!60,fill opacity=0.5]
(T3) .. controls +(-120:0.85) and +(90:1.5) .. (B1) .. controls +(90:0.85) and +(90:0.85) .. (B2).. controls +(90:0.85) and +(90:0.85) .. (B4)
.. controls +(90:0.5) and +(90:0.5) .. (B5) .. controls +(90:0.85) and +(90:0.85) .. (B7)
.. controls +(90:0.5) and +(90:0.5) .. (B8) .. controls +(90:0.85) and +(-90:0.85) .. (T7)
.. controls +(-90:0.5) and +(-90:0.5) .. (T6) .. controls +(-90:0.5) and +(-90:0.5) .. (T3);
\draw[black,line width = 2 pt,fill=black!60,fill opacity=0.5]
(T1) .. controls +(-90:0.5)
and +(-90:0.85) ..  (T2) .. controls +(-90:0.85) and +(-90:0.85) .. (T4) ..
controls +(-90:1.0) and +(-90:1.0) .. (T5)  .. controls +(-90:1.15) and
+(-120:1.15) .. (T8).. controls +(-90:0.85) and +(45:0.85) .. (B6)   ..
controls
+(90:0.85) and +(90:0.85) ..  (B3)  .. controls
+(90:0.85) and +(-90:0.85) .. (T1);
\foreach \i in {1,...,8}  { \fill (T\i) circle (6pt); \fill (B\i) circle
(6pt); }
\draw (9.25,0) node {$=d_3$.};
\draw (T1) node[white] {\bf -}; \draw (T2) node[white] {\bf -}; \draw (T3)
node[white] {\bf +}; \draw (T4) node[white] {\bf -}; \draw (T5) node[white]
{\bf -}; \draw (T6) node[white] {\bf +}; \draw (T7) node[white] {\bf +};
\draw (T8) node[white] {\bf -};\draw (B1) node[white] {\bf +};\draw (B2) node[white] {\bf +};\draw (B3)
node[white] {\bf -};\draw (B4) node[white] {\bf +};\draw (B5) node[white]
{\bf +};\draw (B6) node[white] {\bf -};\draw (B7) node[white] {\bf +};\draw (B8) node[white] {\bf +}; 
\end{tikzpicture}
\end{array}
\]
If the set partitions in the bottom of $d_1$ do not match \emph{exactly} with the set partitions of the top of $d_2$ then this product is 0. Note that the set partitions must match, but the $+$ and $-$ labels might be 
reversed.

  In order for a set partition diagram  $\mathsf{d}_{\rf,\sm}$ to belong to the centralizer algebra $\Z_k(\DD_n)$, we must have  $\mathsf{r,s} \in \mathsf{K}_\ell$.   Now  $\mathsf{r,s} \in \mathsf{K}_\ell$ if and only if $|\rf| \equiv |\sm| \equiv \frac{1}{2}(k-\ell) \modd n$,  which is equivalent to saying 
\begin{equation}
t(B) \equiv b(B) \modd n \quad\hbox{ for each block $B$ in $\mathsf{d}_{\rf,\sm}$},\end{equation} 
where $t(B) = |B \cap \{1, \ldots, k\}|$ and $b(B) =| B \cap \{1', \ldots, k'\}|$.
\medskip

\begin{example} Recall from Example \ref{ex:dih2} that $\dimm \Z_4(\DD_4) = 36$.   The diagrammatic  basis is given by diagrams on two rows of 4 vertices each that partition the vertices into $\le 2$ blocks $B$ that satisfy $t(B) \equiv b(B)\modd 4.$  They are  
\begin{enumerate}
\item[(a)]
$
\begin{array}{c}
 \begin{tikzpicture}[scale=.5,line width=1pt]
\foreach \i in {1,...,4}  { \path (\i,1) coordinate (T\i); \path (\i,-.5) coordinate (B\i); }
\filldraw[fill= black!10,draw=black!10,line width=4pt]  (T1) -- (T4) --(B4) -- (B1) -- (T1);
\draw[black,line width = 2 pt,fill=black!60,fill opacity=0.5]
(T1) .. controls +(-90:0.65) and +(-90:0.65) ..  (T2) .. controls +(-90:0.65) and +(-90:0.65) .. (T3) .. controls +(-90:0.65) and +(-90:0.65) .. (T4) -- (B4)
.. controls +(+90:0.65) and +(+90:0.65) ..  (B3) .. controls +(+90:0.65) and +(+90:0.65) ..  (B2) .. controls +(+90:0.65) and +(+90:0.65) ..  (B1) -- (T1);
\foreach \i in {1,...,4}  { \fill (T\i) circle (5pt); \fill (B\i) circle (5pt); }
\end{tikzpicture}
\end{array}
$ has one block $B$ with $t(B) = b(B) = 4.$
\item[(b)]
$
\begin{array}{c}
 \begin{tikzpicture}[scale=.5,line width=1pt]
\foreach \i in {1,...,4}  { \path (\i,1) coordinate (T\i); \path (\i,-.5) coordinate (B\i); }
\filldraw[fill= black!10,draw=black!10,line width=4pt]  (T1) -- (T4) --(B4) -- (B1) -- (T1);
\draw[black,line width = 2 pt]
(T1) .. controls +(-90:0.65) and +(-90:0.65) ..  (T2) .. controls +(-90:0.65) and +(-90:0.65) .. (T3) .. controls +(-90:0.65) and +(-90:0.65) .. (T4);
\draw[black,line width = 2 pt] 
(B4) .. controls +(+90:0.65) and +(+90:0.65) ..  (B3) .. controls +(+90:0.65) and +(+90:0.65) ..  (B2) .. controls +(+90:0.65) and +(+90:0.65) ..  (B1);
\foreach \i in {1,...,4}  { \fill (T\i) circle (5pt); \fill (B\i) circle (5pt); }
\end{tikzpicture}
\end{array}
$ has  a top block $B_1$ with  $t(B_1) = 4 \equiv 0  = b(B_1) \modd 4$ and a bottom block $B_2$ with $t(B_2) = 0 \equiv 4  = b(B_2) \modd 4$.
\item[(c)] There are 16 diagrams with two blocks $d = B_1 \sqcup B_2$ in which $B_1$ has 3 vertices in each row and $B_2$ has 1 vertex in each row:
\[
\begin{array}{c}
 \begin{tikzpicture}[scale=.5,line width=1pt]
\foreach \i in {1,...,4}  { \path (\i,1) coordinate (T\i); \path (\i,-.5) coordinate (B\i); }
\filldraw[fill= black!10,draw=black!10,line width=4pt]  (T1) -- (T4) --(B4) -- (B1) -- (T1);
\draw[black,line width = 2 pt,fill=black!60,fill opacity=0.5]
(T1) .. controls +(-90:0.65) and +(-90:0.65) ..  (T2) .. controls +(-90:0.65) and +(-90:0.65) .. (T3)  -- (B3)
.. controls +(+90:0.65) and +(+90:0.65) ..  (B2) .. controls +(+90:0.65) and +(+90:0.65) ..  (B1) -- (T1);
\draw[black,line width = 2 pt] (T4) -- (B4);
\foreach \i in {1,...,4}  { \fill (T\i) circle (5pt); \fill (B\i) circle (5pt); }
\end{tikzpicture}
\end{array}, \quad
\begin{array}{c}
 \begin{tikzpicture}[scale=.5,line width=1pt]
\foreach \i in {1,...,4}  { \path (\i,1) coordinate (T\i); \path (\i,-.5) coordinate (B\i); }
\filldraw[fill= black!10,draw=black!10,line width=4pt]  (T1) -- (T4) --(B4) -- (B1) -- (T1);
\draw[black,line width = 2 pt,fill=black!60,fill opacity=0.5]
(T1) .. controls +(-90:0.65) and +(-90:0.65) ..  (T2) .. controls +(-90:0.65) and +(-90:0.65) .. (T4)  .. controls +(-90:0.65) and +(+90:0.65) .. (B3)
.. controls +(+90:0.65) and +(+90:0.65) ..  (B2) .. controls +(+90:0.65) and +(+90:0.65) ..  (B1) -- (T1);
\draw[black,line width = 2 pt] (T3) .. controls +(-90:0.65) and +(+90:0.65) .. (B4);
\foreach \i in {1,...,4}  { \fill (T\i) circle (5pt); \fill (B\i) circle (5pt); }
\end{tikzpicture}
\end{array},
\quad
\begin{array}{c}
 \begin{tikzpicture}[scale=.5,line width=1pt]
\foreach \i in {1,...,4}  { \path (\i,1) coordinate (T\i); \path (\i,-.5) coordinate (B\i); }
\filldraw[fill= black!10,draw=black!10,line width=4pt]  (T1) -- (T4) --(B4) -- (B1) -- (T1);
\draw[black,line width = 2 pt,fill=black!60,fill opacity=0.5]
(T1) .. controls +(-90:0.65) and +(-90:0.65) ..  (T3) .. controls +(-90:0.65) and +(-90:0.65) .. (T4)  .. controls +(-90:0.65) and +(+90:0.65) .. (B3)
.. controls +(+90:0.65) and +(+90:0.65) ..  (B2) .. controls +(+90:0.65) and +(+90:0.65) ..  (B1) -- (T1);
\draw[black,line width = 2 pt] (T2) .. controls +(-90:0.65) and +(+90:0.65) .. (B4);
\foreach \i in {1,...,4}  { \fill (T\i) circle (5pt); \fill (B\i) circle (5pt); }
\end{tikzpicture}
\end{array},
\quad
\begin{array}{c}
 \begin{tikzpicture}[scale=.5,line width=1pt]
\foreach \i in {1,...,4}  { \path (\i,1) coordinate (T\i); \path (\i,-.5) coordinate (B\i); }
\filldraw[fill= black!10,draw=black!10,line width=4pt]  (T1) -- (T4) --(B4) -- (B1) -- (T1);
\draw[black,line width = 2 pt,fill=black!60,fill opacity=0.5]
(T2) .. controls +(-90:0.65) and +(-90:0.65) ..  (T3) .. controls +(-90:0.65) and +(-90:0.65) .. (T4)  .. controls +(-90:0.65) and +(+90:0.65) .. (B3)
.. controls +(+90:0.65) and +(+90:0.65) ..  (B2) .. controls +(+90:0.65) and +(+90:0.65) ..  (B1) .. controls +(90:0.65) and +(-90:0.65) .. (T2);
\draw[black,line width = 2 pt] (T1) .. controls +(-90:0.65) and +(+90:0.65) .. (B4);
\foreach \i in {1,...,4}  { \fill (T\i) circle (5pt); \fill (B\i) circle (5pt); }
\end{tikzpicture}
\end{array},
\ldots \text{(12 more)} \ldots
\]
\item[(d)] There are 18 diagrams with two blocks  $d = B_1 \sqcup B_2$ in which both blocks have 2 vertices in each row:
\[
\begin{array}{c}
 \begin{tikzpicture}[scale=.5,line width=1pt]
\foreach \i in {1,...,4}  { \path (\i,1) coordinate (T\i); \path (\i,-.5) coordinate (B\i); }
\filldraw[fill= black!10,draw=black!10,line width=4pt]  (T1) -- (T4) --(B4) -- (B1) -- (T1);
\draw[black,line width = 2 pt,fill=black!60,fill opacity=0.5]
(T1) .. controls +(-90:0.65) and +(-90:0.65) ..  (T2) .. controls +(-90:0.65) and +(+90:0.65) .. (B2)
.. controls +(+90:0.5) and +(+90:0.5) ..  (B1) .. controls +(+90:0.65) and +(-90:0.65) ..  (T1);
\draw[black,line width = 2 pt,fill=black!60,fill opacity=0.5]
(T3) .. controls +(-90:0.65) and +(-90:0.65) ..  (T4) .. controls +(-90:0.65) and +(+90:0.65) .. (B4)
.. controls +(+90:0.5) and +(+90:0.5) ..  (B3) .. controls +(+90:0.65) and +(-90:0.65) ..  (T3);
\foreach \i in {1,...,4}  { \fill (T\i) circle (5pt); \fill (B\i) circle (5pt); }
\end{tikzpicture}
\end{array}, \quad
\begin{array}{c}
 \begin{tikzpicture}[scale=.5,line width=1pt]
\foreach \i in {1,...,4}  { \path (\i,1) coordinate (T\i); \path (\i,-.5) coordinate (B\i); }
\filldraw[fill= black!10,draw=black!10,line width=4pt]  (T1) -- (T4) --(B4) -- (B1) -- (T1);
\draw[black,line width = 2 pt,fill=black!60,fill opacity=0.5]
(T1) .. controls +(-90:0.5) and +(-90:0.5) ..  (T3) .. controls +(-90:0.65) and +(+90:0.65) .. (B2)
.. controls +(+90:0.5) and +(+90:0.5) ..  (B1) .. controls +(+90:0.65) and +(-90:0.65) ..  (T1);
\draw[black,line width = 2 pt,fill=black!60,fill opacity=0.5]
(T2) .. controls +(-90:0.85) and +(-90:0.65) ..  (T4) .. controls +(-90:0.65) and +(+90:0.65) .. (B4)
.. controls +(+90:0.5) and +(+90:0.5) ..  (B3) .. controls +(+90:0.65) and +(-90:0.65) ..  (T2);
\foreach \i in {1,...,4}  { \fill (T\i) circle (5pt); \fill (B\i) circle (5pt); }
\end{tikzpicture}
\end{array}, \quad
\begin{array}{c}
 \begin{tikzpicture}[scale=.5,line width=1pt]
\foreach \i in {1,...,4}  { \path (\i,1) coordinate (T\i); \path (\i,-.5) coordinate (B\i); }
\filldraw[fill= black!10,draw=black!10,line width=4pt]  (T1) -- (T4) --(B4) -- (B1) -- (T1);
\draw[black,line width = 2 pt,fill=black!60,fill opacity=0.5]
(T1) .. controls +(-90:0.85) and +(-90:0.85) ..  (T4) .. controls +(-90:0.65) and +(+90:0.65) .. (B2)
.. controls +(+90:0.5) and +(+90:0.5) ..  (B1) .. controls +(+90:0.65) and +(-90:0.65) ..  (T1);
\draw[black,line width = 2 pt,fill=black!60,fill opacity=0.5]
(T2) .. controls +(-90:0.5) and +(-90:0.5) ..  (T3) .. controls +(-90:0.65) and +(+90:0.65) .. (B4)
.. controls +(+90:0.5) and +(+90:0.5) ..  (B3) .. controls +(+90:0.65) and +(-90:0.65) ..  (T2);
\foreach \i in {1,...,4}  { \fill (T\i) circle (5pt); \fill (B\i) circle (5pt); }
\end{tikzpicture}
\end{array},
\quad
\begin{array}{c}
 \begin{tikzpicture}[scale=.5,line width=1pt]
\foreach \i in {1,...,4}  { \path (\i,1) coordinate (T\i); \path (\i,-.5) coordinate (B\i); }
\filldraw[fill= black!10,draw=black!10,line width=4pt]  (T1) -- (T4) --(B4) -- (B1) -- (T1);
\draw[black,line width = 2 pt,fill=black!60,fill opacity=0.5]
(T1) .. controls +(-90:0.65) and +(-90:0.65) ..  (T2) .. controls +(-90:0.65) and +(+90:0.65) .. (B3)
.. controls +(+90:0.65) and +(+90:0.5) ..  (B1) .. controls +(+90:0.65) and +(-90:0.65) ..  (T1);
\draw[black,line width = 2 pt,fill=black!60,fill opacity=0.5]
(T3) .. controls +(-90:0.65) and +(-90:0.65) ..  (T4) .. controls +(-90:0.65) and +(+90:0.65) .. (B4)
.. controls +(+90:0.65) and +(+90:0.85) ..  (B2) .. controls +(+90:0.65) and +(-90:0.65) ..  (T3);
\foreach \i in {1,...,4}  { \fill (T\i) circle (5pt); \fill (B\i) circle (5pt); }
\end{tikzpicture}
\end{array},
\ldots \text{(5 more)} \ldots
\]
\vskip-.15in
\[
\begin{array}{c}
 \begin{tikzpicture}[scale=.5,line width=1pt]
\foreach \i in {1,...,4}  { \path (\i,1) coordinate (T\i); \path (\i,-.5) coordinate (B\i); }
\filldraw[fill= black!10,draw=black!10,line width=4pt]  (T1) -- (T4) --(B4) -- (B1) -- (T1);
\draw[black,line width = 2 pt,fill=black!60,fill opacity=0.5]
(T1) .. controls +(-90:0.65) and +(-90:0.65) ..  (T2) .. controls +(-90:0.65) and +(+90:0.65) .. (B4)
.. controls +(+90:0.5) and +(+90:0.5) ..  (B3) .. controls +(+90:0.65) and +(-90:0.65) ..  (T1);
\draw[black,line width = 2 pt,fill=black!60,fill opacity=0.5]
(T3) .. controls +(-90:0.65) and +(-90:0.65) ..  (T4) .. controls +(-90:0.65) and +(+90:0.65) .. (B2)
.. controls +(+90:0.5) and +(+90:0.5) ..  (B1) .. controls +(+90:0.65) and +(-90:0.65) ..  (T3);
\foreach \i in {1,...,4}  { \fill (T\i) circle (5pt); \fill (B\i) circle (5pt); }
\end{tikzpicture}
\end{array}, \quad
\begin{array}{c}
 \begin{tikzpicture}[scale=.5,line width=1pt]
\foreach \i in {1,...,4}  { \path (\i,1) coordinate (T\i); \path (\i,-.5) coordinate (B\i); }
\filldraw[fill= black!10,draw=black!10,line width=4pt]  (T1) -- (T4) --(B4) -- (B1) -- (T1);
\draw[black,line width = 2 pt,fill=black!60,fill opacity=0.5]
(T1) .. controls +(-90:0.5) and +(-90:0.5) ..  (T3) .. controls +(-90:0.65) and +(+90:0.65) .. (B4)
.. controls +(+90:0.5) and +(+90:0.5) ..  (B3) .. controls +(+90:0.65) and +(-90:0.65) ..  (T1);
\draw[black,line width = 2 pt,fill=black!60,fill opacity=0.5]
(T2) .. controls +(-90:0.85) and +(-90:0.65) ..  (T4) .. controls +(-90:0.65) and +(+90:0.65) .. (B2)
.. controls +(+90:0.5) and +(+90:0.5) ..  (B1) .. controls +(+90:0.65) and +(-90:0.65) ..  (T2);
\foreach \i in {1,...,4}  { \fill (T\i) circle (5pt); \fill (B\i) circle (5pt); }
\end{tikzpicture}
\end{array}, \quad
\begin{array}{c}
 \begin{tikzpicture}[scale=.5,line width=1pt]
\foreach \i in {1,...,4}  { \path (\i,1) coordinate (T\i); \path (\i,-.5) coordinate (B\i); }
\filldraw[fill= black!10,draw=black!10,line width=4pt]  (T1) -- (T4) --(B4) -- (B1) -- (T1);
\draw[black,line width = 2 pt,fill=black!60,fill opacity=0.5]
(T1) .. controls +(-90:0.85) and +(-90:0.85) ..  (T4) .. controls +(-90:0.65) and +(+90:0.65) .. (B4)
.. controls +(+90:0.5) and +(+90:0.5) ..  (B3) .. controls +(+90:0.65) and +(-90:0.65) ..  (T1);
\draw[black,line width = 2 pt,fill=black!60,fill opacity=0.5]
(T2) .. controls +(-90:0.5) and +(-90:0.5) ..  (T3) .. controls +(-90:0.65) and +(+90:0.65) .. (B2)
.. controls +(+90:0.5) and +(+90:0.5) ..  (B1) .. controls +(+90:0.65) and +(-90:0.65) ..  (T2);
\foreach \i in {1,...,4}  { \fill (T\i) circle (5pt); \fill (B\i) circle (5pt); }
\end{tikzpicture}
\end{array},
\quad
\begin{array}{c}
 \begin{tikzpicture}[scale=.5,line width=1pt]
\foreach \i in {1,...,4}  { \path (\i,1) coordinate (T\i); \path (\i,-.5) coordinate (B\i); }
\filldraw[fill= black!10,draw=black!10,line width=4pt]  (T1) -- (T4) --(B4) -- (B1) -- (T1);
\draw[black,line width = 2 pt,fill=black!60,fill opacity=0.5]
(T1) .. controls +(-90:0.65) and +(-90:0.65) ..  (T2) .. controls +(-90:0.65) and +(+90:0.65) .. (B4)
.. controls +(+90:0.65) and +(+90:0.5) ..  (B2) .. controls +(+90:0.65) and +(-90:0.65) ..  (T1);
\draw[black,line width = 2 pt,fill=black!60,fill opacity=0.5]
(T3) .. controls +(-90:0.65) and +(-90:0.65) ..  (T4) .. controls +(-90:0.65) and +(+90:0.65) .. (B3)
.. controls +(+90:0.65) and +(+90:0.85) ..  (B1) .. controls +(+90:0.65) and +(-90:0.65) ..  (T3);
\foreach \i in {1,...,4}  { \fill (T\i) circle (5pt); \fill (B\i) circle (5pt); }
\end{tikzpicture}
\end{array},
\ldots \text{(5 more)} \ldots
\]\end{enumerate} \end{example} 
 \begin{subsection}{Irreducible modules for $\Z_k(\DD_n)$} \end{subsection}

\begin{thm}\label{thm:simplemods}  Assume $\ell \in \Lambda_k^{\bullet}(\DD_n)$  as in \eqref{eq:lamsDn}, and let $a_\ell \in \{0,1\dots, k\}$ be minimal such that
$\ell \equiv k-2a_\ell \modd 2n$.  
\begin{itemize} 
\item[{\rm (i)}]  For $\ac  \in \{1,\dots, n-1\}$, 
 \begin{equation}\label{eq:wkjdef} \Z_{k}^{(\ac)} = \Z_k(\DD_n)^{(\ell)} := \spann_{\mathbb C}\{ \vf_{\mathsf{t}} \mid \mathsf{t} \in {\kl}_\ac \},\end{equation} 
where ${\kl}_\ac$ is as in \eqref{eq:Kdef},  
 is an irreducible $\Z_k(\DD_n)$-module,  and 
\begin{equation}\label{eq:dimwj} \dimm \Z_{k}^{(\ac)}= \sum_{{0 \leq b \leq k} \atop {b \equiv a_\ell \modd n}}{k \choose b} \ = \
\hbox{coefficient of} \ \ z^{a_\ell} \ \  \hbox{in} \ \  (1+z)^{k} \big |_{z^{n} = 1}.  \end{equation} 

\item[{\rm (ii)}] For  $\ac\in \{0,n\}$,  
\begin{eqnarray}\label{eq:wk0ndef} \Z_k^{(\ac)} &=& \Z_k(\DD_n)^{(\ell)}:= \spann_{\mathbb C}\{ \vf_{\mathsf{t}}+ i^{\ac-k}\vf_{\mathsf{-t}}\ \mid \mathsf{t} \in {\kl}_\ac\},\\
\Z_k^{(\ac')} &=& \Z_k(\DD_n)^{(\ell')}:= \spann_{\mathbb C}\{ \vf_{\mathsf{t}}- i^{\ac-k}\vf_{\mathsf{-t}}\ \mid \mathsf{t} \in {\kl}_\ac\}, \nonumber \end{eqnarray}   
where ${\kl}_\ac$ is as in \eqref{eq:K0ndef},  are irreducible $\Z_k(\DD_n)$-modules,  and 
\begin{equation}\label{eq:dimw0n} \dimm \Z_k^{(\ac')} = \dimm \Z_k^{(\ac)} = \frac{1}{2}\sum_{{0 \leq b \leq k} \atop {b \equiv a_\ac  \modd n}}{k \choose b} = 
\frac{1}{2}\left(\hbox{coefficient of} \ \ z^{a_\ac} \ \  \hbox{in} \ \  (1+z)^{k} \big |_{z^{n} = 1}\right).
\end{equation} 
\end{itemize}
Up to isomorphism, the modules in  {\rm (i)} and {\rm (ii)}  are the only irreducible  $\Z_k(\DD_n)$-modules.
\end{thm}  

\begin{proof} 
Assume initially that $\ac  \in \{1,\dots, n-1\}$ and $\ell \in \Lambda_k^{\bullet}(\DD_n)$.   Let 
$\Z_{k}^{(\ac)}$ be as in \eqref{eq:wkjdef}.   For all 
$\mathsf{r,s,t}  \in {\kl}_\ac$, \ 
$\mathsf{e_{r,s}} \vf_{\mathsf{t}} = \mathsf{\delta_{s,t}} \vf_{\mathsf{r}}$, where $\mathsf{e_{r,s} = E_{r,s} + E_{-r,-s}} \in \Z_k(\DD_n)$ is as in Theorem \ref{thm:mj} (ii).
Thus, $\Z_{k}^{(\ac)}$ is the unique irreducible module (up to isomorphism) for the matrix
subalgebra of $\Z_k(\DD_n)$  having basis the matrix units  $\mathsf{e_{r,s}}$, $\mathsf{r,s} \in {\kl}_\ac$.   Since the other basis
elements in the basis $\mathcal B_{\mathsf{mat}}^k(\DD_n)$ (in Corollary \ref{cor:matrixbasis})  act trivially on the
vectors in $\Z_{k}^{(\ac)}$,  we have that $\Z_{k}^{(\ac)}$ is an irreducible  $\Z_k(\DD_n)$-module.     Since $k-2\mathsf{\vert t \vert} \equiv 
\ell \modd 2n$ for all $\mathsf{t} \in {\kl}_\ell$, we have $\mathsf{\vert t \vert} \equiv a_\ell \modd n$. Therefore $\dimm \Z_{k}^{(\ac)}= \sum_{{0 \leq b \leq k} \atop {b \equiv a_\ac \modd n}}{k \choose b} \ = \
\hbox{coefficient of} \ \ z^{a_\ac} \ \  \hbox{in} \ \  (1+z)^{k} \big |_{z^{n} = 1},$  as claimed in (i). 
Observe  $\W_{k}^{(\ac)}: = \spann_{\mathbb C}\{ i^{\ac-k} \vf_{\mathsf{-t}} \mid \mathsf{t} \in {\kl}_\ac\}$
is also a $\Z_k(\DD_n)$-module isomorphic to $\Z_k^{(\ac)}$ by Theorem \ref{thm:mj}\,(ii),  
 and the $(\DD_n \times \Z_k(\DD_n))$-bimodule  $\mathsf{B}_{k}^{(\ac)}$ in that theorem is the sum  $\mathsf{B}_{k}^{(\ac)} =
\W_{k}^{(\ac)} \oplus  \Z_{k}^{(\ac)}.$

Now assume $\ac= 0,n$ and $\ell \in \Lambda_k^{\bullet}(\DD_n)$.   Let 
$\Z_k^{(\ac)}$ and $\Z_k^{(\ac')}$ be as in \eqref{eq:wk0ndef}.   
   Since according to Theorem \ref{thm:m0n},  
\[ \mathsf{e^{\pm}_{r,s}}(\vf_{\mathsf t} \pm i^{\ac-k} \vf_{\mathsf {-t}})\ = \ 
\delta_{\mathsf{s,t}}(\vf_{\mathsf r} \pm i^{\ac-k} \vf_{\mathsf {-r}}), \ \qquad  \
\mathsf{e^{\pm}_{r,s}}(\vf_{\mathsf t} \mp i^{\ell-k} \vf_{\mathsf {-t}}) \ = \
0\]  
for $\mathsf{e^{\pm}_{r,s}} := \frac{1}{2}\Big((\mathsf{E_{r,s} + E_{-r,-s}}) \pm i^{\ac-k}
(\mathsf{E_{r,-s} + E_{-r,s}})\Big)$ with $\mathsf{r,s} \in {\kl}_\ac$, and since
all other elements of the basis $\mathcal B_{\mathsf{mat}}^k(\DD_n)$ act trivially on $\Z_k^{(\ac)}$ and $\Z_k^{(\ac')}$,
we see that they are irreducible  $\Z_k(\DD_n)$-modules.      Moreover, since $\mathsf{\vert t\vert} \equiv a_\ell$ for all $\mathsf{t} \in \kl_\ell$, they have dimension 
\[ \frac{1}{2}\sum_{{0 \leq b \leq k} \atop {b \equiv a_\ac  \modd n}}{k \choose b} = 
\frac{1}{2}\left(\hbox{coefficient of} \ \ z^{a_\ac} \ \  \hbox{in} \ \  (1+z)^{k} \big |_{z^{n} = 1}\right).\]
 As the sum of the squares of the dimensions of the modules in (i) and (ii) adds up to
$\dimm \Z_k(\DD_n)$, these are all the irreducible  $\Z_k(\DD_n)$-modules
up to isomorphism.   \end{proof} 
The next result  is an immediate consequence of Theorem \ref{thm:simplemods}.  

 \begin{cor}\label{Cor:Dwalks} As in Theorem \ref{thm:simplemods},  assume $\ell \in \{0,1,\dots, n\}$ and $a_\ell \in \{0,1\dots, k\}$ is minimal such that
$\ell \equiv k-2a_\ell \modd 2n$. 
 \begin{itemize}
\item[{\rm (i)}]  For $\ac \in \{1,\dots, n-1\}$, the number of walks of $k$ steps from 0 to $\ac$ on the affine Dynkin diagram of type $\hat{\mathsf{D}}_{n+2}$ is  $\displaystyle{\sum_{{0 \leq b \leq k} \atop {b \equiv a_\ell \modd n}}{k \choose b}.}$ 
\item[{\rm (ii)}]  For $\ac \in \{0,n\}$, the number of walks of $k$ steps from 0 to $\ac$ (to $\ac'$)  on the affine Dynkin diagram of type $\hat{\mathsf{D}}_{n+2}$ is \ \, $\displaystyle{\half\sum_{{0 \leq b \leq k} \atop {b \equiv a_\ell \modd n}}{k \choose b}.}$ 
\end{itemize}
\end{cor} 

\begin{example}\label{ex:di4} 
{\rm Assume $k = 4$ and $n= 5$.      Then $\Lambda_k(\DD_n) = \{0,0',2,4\}$ and $a_\ac = \half(k-\ell) =2,1,0,$ respectively. 
Thus,
\[\dimm \Z_4^{(0')}  = \dimm \Z_4^{(0)} =  \frac{1}{2} \sum_{{0 \leq b \leq 4}\atop{b \equiv 2 \modd 5}} {4 \choose b}  = \frac{1}{2} {4 \choose 2} = 3.\]
\begin{eqnarray*} \dimm \Z_4^{(\ac)} &=& \begin{cases} \displaystyle{\sum_{{0 \leq b \leq 4}\atop{b \equiv 1 \modd 5} }} {4 \choose b}  = {4 \choose 1} = 4 & \qquad  \hbox{if} \ \  \ac = 2, \\
\displaystyle{\sum_{{0 \leq b \leq 4}\atop{b \equiv 0 \modd 5}}} {4 \choose b}  = {4 \choose 0} = 1 & \qquad  \hbox{if} \ \  \ac = 4. \end{cases} 
 \end{eqnarray*}
 Then $\dimm \Z_k(\DD_n)$ is the sum of the squares of these numbers which is 35, 
 as in Example \ref{ex:dih2}. 
 }  \end{example}   
 
 \begin{remark}  For $\ell \in \{1,\dots, n-1\}$ we have
 \begin{equation}  \dimm \Z_k(\C_{2n})^{(2n-\ell)} = \dimm \Z_k(\C_{2n})^{(\ell)} = \dimm \Z_k(\DD_n)^{(\ell)}. \end{equation}
 This can be seen from the observation that if  $2n-\ell \equiv k-2a \modd 2n$,  for some $a \in \{0,1,\dots, k\}$,  then $\ell \equiv k-2(k-a)  \modd n$,
 so that
\begin{eqnarray*} \dimm \Z_k(\C_{2n})^{(2n-\ell)} &=&\sum_{ {0 \leq b \leq k} \atop {b \equiv a \modd n}} {k \choose b} \ = \
 \sum_{ {0 \leq k-b \leq k}\atop {k-b \equiv  k-a \modd n}} {k \choose k-b} \ = \ \sum_{ {0 \leq c \leq k} \atop {c \equiv k-a \modd n}} {k \choose c}  \\   &=&  \dimm \Z_k(\C_{2n})^{(\ell)} =  \dimm \Z_k(\DD_n)^{(\ell)}. \end{eqnarray*}
For $\ell \in \{0,n\}$,  we have seen that 
\[\dimm \Z_k(\DD_n)^{(\ell')} = \dimm \Z_k(\DD_n)^{(\ell)} = \half  \dimm \Z_k(\C_{2n})^{(\ell)}.\]  Therefore,
\begin{eqnarray*} \dimm \Z_k(\DD_n) &=& \sum_{\ell = 1}^{n-1} \left(\dimm \Z_k(\DD_n)^{(\ell)}\right)^2 + \sum_{\ell \in \{0,n\}}\left( \left(\dimm \Z_k(\DD_n)^{(\ell)}\right)^2 +  \left(\dimm \Z_k(\DD_n)^{(\ell')}\right)^2 \right) \\
&=& \half \sum_{\ell =1, \ell \neq n}^{2n-1}  \left(\dimm \Z_k(\C_{2n})^{(\ell)}\right)^2 + \half\left(\dimm \Z_k(\C_{2n})^{(0)}\right)^2 +  \half\left(\dimm\Z_k(\C_{2n})^{(n)}\right)^2 \\
&=& \half \dimm \Z_k(\C_{2n}),  \end{eqnarray*}
as in Theorem \ref{thm:dicentbasis}(b).
\end{remark}

\subsection{The infinite binary dihedral subgroup $\DD_{\infty}$ of $\SU_2$} 
Let  $\DD_\infty$ denote the subgroup of $\SU_2$ generated by 
\begin{equation}\label{eq:defghinf} g = \left (\begin{array}{cc} \zeta^{-1}  & 0 \\ 0 & \zeta \end{array}\right),  \qquad
h = \left (\begin{array}{cc} 0  & i \\  i & 0 \end{array}\right),\end{equation}
where $\zeta =e^{i \theta}$, $\theta \in \mathbb R$, $i = \sqrt{-1}$, and $\zeta$ is not
a root of unity.   Then the following relations hold in $\DD_\infty$:
\begin{equation}\label{eq:defrel} h^4 = 1,  \quad   h^{-1} g h = g^{-1}, \quad g^n \neq 1 \ \ \hbox{\rm for} \ \ n \neq 0. \end{equation} 
 
For $\ac= 1,2, \dots$,  let  $\DD_\infty^{(\ac)}$ denote the two-dimensional $\DD_\infty$-module on
which the generators $g,h$ have the following matrix representations:
\[g  \ \rightarrow  \left (\begin{array}{cc} \zeta^{-\ac}  & 0 \\ 0 & \zeta^\ac\end{array}\right),  \qquad
h  \ \rightarrow   \left (\begin{array}{cc} 0  & i^\ac \\  i^\ac& 0 \end{array}\right).\]
relative to the basis $\{\vf_{-\ac}, \vf_\ac\}$.  In particular, $\VV = \DD_{\infty}^{(1)}$. 
For  $\ac = 0,0'$, let the one-dimensional 
$\DD_{\infty}$-module  $\DD_n^{(\ac)}$ be given by 
\begin{eqnarray}  
\DD_\infty^{(0)}  &=&  \CC \vf_0,  \qquad  \  g\vf_0 = \vf_0, \quad \ h\vf_0 = \vf_0  \\
\DD_\infty^{(0')} &=& \CC \vf_{0'},  \qquad   g\vf_{0'} = \vf_{0'}, \quad h\vf_{0'} = -\vf_{0'}.
\end{eqnarray}  
 
 \begin{prop}\label{prop:Vtensinf} Tensor products of $\VV$ with the irreducible modules
$\DD_\infty^{(\ac)}$ decompose as  follows: 
\begin{itemize} 
\item[{\rm (a)}]  
 $\DD_\infty^{(\ac)} \otimes \VV   \cong   \DD_\infty^{(\ac-1)} \oplus  \DD_\infty^{(\ac+1)}$ \ \  for $1< \ac < \infty$; 
\item[{\rm (b)}]  $ \DD_\infty^{(1)} \ot \VV  \cong   \DD_\infty^{(0')} \oplus   \DD_\infty^{(0)} \oplus \DD_\infty^{(2)}$; 
\item[{\rm (c)}]  $\DD_\infty^{(0)} \ot \VV  \cong\DD_\infty^{(1)} =   \VV$,  \quad   $\DD_\infty^{(0')} \ot \VV  \cong \DD_\infty^{(1)} =   \VV$.\end{itemize}  
   \end{prop}
 
The representation graph $\mathcal R_{\VV}(\DD_{\infty})$ of $\DD_{\infty}$ is the Dynkin diagram
$\mathsf{D}_\infty$ (see Section \ref{subsec:Ddiag}),    where each of the nodes 
 $\ac = 0,0',1,2,\dots$ corresponds to one of these irreducible 
$\DD_\infty$-modules.    The arguments in previous sections can be  modified to show the next result. 

\begin{thm}\label{thm:diinfcentbasis} Let $\Z_k(\DD_{\infty}) =   \mathsf{End}_{\DD_{\infty}}(\VV^{\otimes k})$.   Then
 \begin{itemize}\item [{\rm (a)}] $\Z_k(\DD_{\infty}) = \{ X \in \Z_k(\C_{\infty}) \mid  hX = Xh\}$,  where $h$ is as
 in \eqref{eq:defghinf}.
\item [{\rm (b)}]  A basis for $\Z_k(\DD_{\infty}) = \mathsf{End}_{\DD_{\infty}}(\VV^{\otimes k})$   is  the set 
\begin{equation}\label{eq:dibasisinf} \mathcal B^k(\DD_\infty) 
= \mathsf{\left \{E_{r,s} + E_{-r,-s} \,\big | \,  r,s \in \{-1,1\}^k, \  r \succ -r,  \  \vert r \vert = \vert s \vert  \right \},} \end{equation} where $\mathsf{r \succ -r}$ is the order in Definition \ref{def:lexorder}. 
\item[{\rm (c)}]  The dimension of $\Z_k(\DD_{\infty})$ is given by
\begin{eqnarray}\label{eq:didiminf} 
\dimm \Z_k(\DD_{\infty}) &=& \half \dimm \Z_k(\C_{\infty}) =  \half  {2k \choose k}  = {2k-1 \choose k}  \\
&=& \half \left(\hbox{\rm coefficient of } \ z^k  \ \hbox{\rm in} \  (1+z)^{2k} \right). \nonumber \end{eqnarray}    
\item[{\rm (d)}]  For $\ell = 0,1,2, \dots,k$  such that $k - \ell$ is even,  let 
\begin{equation}\label{eq:Kinfdef} \tilde{\kl}_\ac  =  \left \{ \mathsf{r}\in \{-1,1\}^k  \ \Big  | \   
k-2 \mathsf{\vert r\vert}  = \ac  \right \} = \left \{ \mathsf{r} \in \{-1,1\}^k  \ \Big  | \  
\mathsf{\vert r\vert} = \textstyle{\half}(k-\ac)  \right \}.  \end{equation}
\begin{itemize}
\item[{\rm (i)}] When $1 \leq \ell \leq k$, then  $\Z_{k}^{(\ac)} = \Z_{k}(\DD_\infty)^{(\ac)}= \spann_{\mathbb C}\{ \vf_{\mathsf{t}} \mid \mathsf{t} \in \tilde{\kl}_\ac \}$
 is an irreducible $\Z_k(\DD_{\infty})$-module with
 \begin{equation}\label{eq:diminfj} \dimm \Z_{k}^{(\ac)}= {k \choose b}\  \ \hbox{\rm  where} \ \ b = \half(k-\ell).  \end{equation} 
For all 
$\mathsf{r,s,t}  \in {\kl}_\ac$, \ 
$\mathsf{e_{r,s}} \vf_{\mathsf{t}} = \mathsf{\delta_{s,t}} \vf_{\mathsf{r}}$, where $\mathsf{e_{r,s} = E_{r,s} + E_{-r,-s}} \in \Z_k(\DD_\infty)$. 
\item[{\rm (ii)}] When $\ell = 0$ (necessarily $k$ is even),  then 
\[\Z_k^{(0)} = \spann_{\mathbb C}\{ \vf_{\mathsf{t}}+ i^{-k}\vf_{\mathsf{-t}}\ \mid \mathsf{t} \in \tilde{\kl}_0\}  \ \hbox{\rm and} \ 
\Z_k^{(0')} = \spann_{\mathbb C}\{ \vf_{\mathsf{t}}- i^{-k}\vf_{\mathsf{-t}}\ \mid \mathsf{t} \in \tilde{\kl}_0\},\]
are irreducible $\Z_k(\DD_{\infty})$-modules with dimensions
 \begin{equation}\label{eq:diminf0} \dimm \Z_{k}^{(0)} =  \dimm \Z_{k}^{(0')}=  \frac{1}{2}{k \choose b}\  \ \hbox{\rm  where} \ \ b = \half k,  \end{equation} and  
\[ \mathsf{e^{\pm}_{r,s}}(\vf_{\mathsf t} \pm i^{-k} \vf_{\mathsf {-t}}) \, = \,
\delta_{\mathsf{s,t}}(\vf_{\mathsf r} \pm i^{-k} \vf_{\mathsf {-r}}), \qquad  
\mathsf{e^{\pm}_{r,s}}(\vf_{\mathsf t} \mp i^{-k} \vf_{\mathsf {-t}}) \, = \, 
0\]
for $\mathsf{e^{\pm}_{r,s}} := \frac{1}{2}\Big((\mathsf{E_{r,s} + E_{-r,-s}}) \pm i^{-k}
(\mathsf{E_{r,-s} + E_{-r,s}})\Big)$ with $\mathsf{r,s} \in \tilde{\kl}_0$. \end{itemize}
\item[{\rm (e)}] $\Z_k(\DD_{\infty}) \cong \mathsf{Q}_k$, where  $\mathsf{Q}_k$ is the subalgebra of the 
planar rook algebra $\mathsf{PR}_k$ having basis
\begin{itemize}
\item[{\rm (i)}]  $\{\mathsf{X_{R,S}} \mid  0  \leq  |\mathsf{R}| = |\mathsf{S}| \leq \half (k-1)\}$
when $k$  is odd;
\item[{\rm (ii)}]  $\{\mathsf{X_{R,S}} \mid  0 \leq  |\mathsf{R}| = |\mathsf{S}| \leq \half(k-2)\}
\cup 
\{ \mathsf{X_{\pm R,\pm S}}  \mid  |\mathsf{R}| = |\mathsf{S}| =  \half k, \  \mathsf{R  \succ -R}, \ \mathsf{S\succ -S} \}$ when $k$ is even;
\end{itemize}
where the  $\mathsf{X_{R,S}}$ for  $\mathsf{R,S} \subseteq \{1,\dots, k\}$ are the
matrix units in Remark  \ref{rem:prook}, and $\succ$ is the order coming from  the
order in Definition \ref{def:lexorder} and  from the identification of a subset $\mathsf{R}$  with the tuple $\mathsf{r} \in \{-1,1\}^k$. 
\end{itemize} 
\end{thm}  

 \begin{remark}   In Remark  \ref{rem:prook}, we identified 
a subset $\mathsf{R} \subseteq \{1,\dots, k\}$ with the $k$-tuple
$\mathsf{r} = (r_1,\dots, r_k)$ such that $r_j = -1$ if $j \in \mathsf{R}$ and $r_j = 1$
if $j \not \in \mathsf{R}$.  The element  $\mathsf{e_{r,s} = E_{r,s} + E_{-r,-s}}$
  in part (b) of this  theorem has the property that
  $\mathsf{\vert r \vert = \vert s \vert}$ and $\mathsf{r \succ -r}$.  Therefore, 
$|\mathsf{R}| = \mathsf{\vert r \vert}$ (the number of $-1$s in $\mathsf{r}$)  is less than or equal to $\mathsf{\vert -r \vert}$  (the number of  
of $-1$s in $\mathsf{-r}$) by the definition of $\succ$, and  $|\mathsf{R}| \leq \lfloor \half k \rfloor$.      Thus, when $k$ is odd, sending  each such $\mathsf{e_{r,s}}$ to 
the corresponding matrix unit $\mathsf{X_{R,S}}$ will give the desired isomorphism in part (e). 
When $k$ is even,
\begin{eqnarray} && \hspace{-.58truein} \left \{\mathsf{e_{r,s} =E_{r,s} + E_{-r,-s}} \ \Big | \  \mathsf{r,s} \in \{-1,1\}^k, \  \mathsf{r \succ -r,  \ \,  \vert r \vert = \vert s \vert} < \textstyle{\half}k  \right \} \\
&& \hspace{.8 truein} \bigcup \ \  \left \{\mathsf{e_{r,s}^{\pm} \ \Big | \  r,s} \in \{-1,1\}^k, \  \mathsf{r \succ -r, \,  \  \vert r \vert = \vert s \vert} =\textstyle{\half} k  \right \} \nonumber \end{eqnarray}
 is a basis for $\Z_k(\DD_{\infty})$.   Note that $\CC\mathsf{e_{r,s}^{+}} \oplus
 \CC\mathsf{e_{r,s}^{-}}  = 
 \CC\mathsf{e_{r,-s}^{+}} \oplus  \CC\mathsf{e_{r,-s}^{-}} $, so we can assume $\mathsf{s \succ -s}$.  Then 
sending $\mathsf{e_{r,s}}$ to $\mathsf{X_{R,S}}$ when $\mathsf{\vert r \vert = \vert s \vert} < \textstyle{\half} k$, and  $\mathsf{e_{r,s}^{\pm}}$ to $\mathsf{X_{\pm R, \pm S}}$ when
$\mathsf{\vert r \vert = \vert s \vert}=  \textstyle{\half} k$ gives the isomorphism in (e).  

In the planar rook algebra $\mathsf{PR}_k$, the elements $\mathsf{X_{R,S}}$ with
$\mathsf{\vert R \vert = \vert S \vert} = \half k$  form a matrix unit basis of 
a matrix algebra of dimension ${k \choose b}^2 $ where $b = \half k$.  
The subsets  $\{\mathsf{X_{R,S}} \mid \mathsf{R \succ -R}, \mathsf{S \succ -S}\}$
and $\{\mathsf{X_{-R,-S}} \mid \mathsf{R \succ -R}, \mathsf{S \succ -S}\}$ each give
a basis of a matrix algebra of dimension $\big(\half {k \choose b}\big)^2$.  The sum 
of those two matrix subalgebras is included in $\mathsf{Q}_k$.       \end{remark}

\begin{section}{Dynkin Diagrams, Bratteli Diagrams, and Dimensions}
\subsection{Representation graphs and Dynkin diagrams}\label{subsec:Ddiag} 
The representation graph $\mathcal{R}_{\VV}(\G)$ for  a finite subgroup $\G$ of $\SU_2$ is the corresponding extended affine Dynkin diagram of type $\hat{\mathsf A}_{n-1}, \hat{\mathsf D}_{n+1},  \hat{\mathsf E}_{6}, \hat{\mathsf E}_{7}, \hat{\mathsf E}_{8}.$  In the figures below, the label on the node is the index of the  irreducible $\G$-module,  and the label above the node is its dimension.  The trivial module is shown in white and the defining module $\VV$ is shown in black
\[
\begin{array}{rll}
\SU_2: \qquad  & \begin{array}{c} \includegraphics[scale=.65,page=1]{mckay-diagrams.pdf} \end{array} &\quad { ({\mathsf A}_{+\infty})}  \\
\C_n: \qquad  & \begin{array}{c} \includegraphics[scale=.65,page=2]{mckay-diagrams.pdf}  \end{array} &\quad (\hat{\mathsf A}_{n-1}) \\
\DD_n: \qquad  & \begin{array}{c} \includegraphics[scale=.65,page=3]{mckay-diagrams.pdf}  \end{array} &\quad (\hat{\mathsf D}_{n+2}) \\
\TT: \qquad & \begin{array}{c} \includegraphics[scale=.65,page=4]{mckay-diagrams.pdf} \end{array} &\quad (\hat{\mathsf E}_{6}) \\ 
\OO: \qquad & \begin{array}{c} \includegraphics[scale=.65,page=5]{mckay-diagrams.pdf}\end{array} &\quad (\hat{\mathsf E}_{7}) \\
\II: \qquad & \begin{array}{c} \includegraphics[scale=.65,page=6]{mckay-diagrams.pdf} \end{array} & \quad (\hat{\mathsf E}_{8}) \\
\C_\infty:\qquad& \begin{array}{c} \includegraphics[scale=.65,page=7]{mckay-diagrams.pdf}  \end{array} &\quad { ({\mathsf A}_{\infty})} \\
\DD_\infty:\qquad& \begin{array}{c} \includegraphics[scale=.65,page=8]{mckay-diagrams.pdf}  \end{array} &\quad { ({\mathsf D}_{\infty})} \\
\end{array}
\]

\subsection{Bratteli diagrams}\label{subsec:Bratteli}

The first few rows of the Bratteli diagrams $\mathcal{B}_{\VV}(\G)$ for  finite subgroups $\G$ of $\SU_2$ are displayed here.  The  nodes  on level $k$  label the irreducible $\G$-modules that appear in $\VV^{\otimes k}$.  The label below each node at level $k$  is the multiplicity of the corresponding $\G$-module in $\VV^{\otimes k}$, which is also the dimension of the corresponding module for  $\Z_k(\G)$ with that same label.  The right-hand column contains the sum of the squares of these dimensions and equals $\dimm\Z_k(\G)$.
An edge between level $k$ and level $k+1$ is highlighted if it \emph{cannot} be obtained as the reflection, over level $k$, of an edge between level $k-1$ and $k$.  The non-highlighted edges  correspond to the Jones basic construction discussed in Section 1.6.  Note that the highlighted edges produce an embedding of the representation graph $\mathcal{R}_\VV(\G)$ into the Bratteli diagram $\mathcal{B}_\VV(\G)$   in all cases except when $\G$ is a cyclic group $\C_n$ with $n$ odd, where the highlighted edges give the double of $\mathcal R_\VV(\G)$.

 \[
\begin{array}{rrll}
\SU_2: \qquad &\hskip.5in & \begin{array}{c}\includegraphics[scale=.7,page=9]{mckay-diagrams.pdf}\end{array} & \hskip1in \\
\\
\end{array}\]
 
Observe that $\C_5$ and $\C_{10}$ have isomorphic Bratteli diagrams; they each correspond to Pascal's triangle on a cylinder of ``diameter" 5: 
\[
\begin{array}{rrll}
\C_5: &\hskip.5in & \begin{array}{c}\includegraphics[scale=.55,page=10]{mckay-diagrams.pdf}\end{array} & \hskip1in \\
\\
\end{array}
\]

\[
\begin{array}{rrll}
\C_{10}: &\hskip.5in & \begin{array}{c}\includegraphics[scale=.55,page=11]{mckay-diagrams.pdf}\end{array} & \hskip1in \\ \\
\end{array}
\] 

\[
\begin{array}{rrll}
\C_\infty: &\hskip.5in & \begin{array}{c}\includegraphics[scale=.55,page=16]{mckay-diagrams.pdf}\end{array} & \hskip1in \\
\end{array}
\]

\begin{multicols}{2}
$\DD_{ 6}$:

\[
\includegraphics[scale=.52,page=12]{mckay-diagrams.pdf} 
\]

\columnbreak
$\DD_\infty$:\vspace{-.1truein}

\[
\includegraphics[scale=.52,page=17]{mckay-diagrams.pdf}
\]  
\end{multicols}  \msk \msk

\newpage
\begin{multicols}{2}
$\TT$: 

\[
\includegraphics[scale=.52,page=13]{mckay-diagrams.pdf}
\]

\columnbreak
$\OO$: 

\[
\includegraphics[scale=.52,page=14]{mckay-diagrams.pdf}
\]
\end{multicols}  
 
$\II$: 
\[\includegraphics[width=3.3in, height=2.4in,page=15]{mckay-diagrams.pdf}
\]

\subsection{Dimensions}
                                      
Using an inductive proof on the structure of the Bratteli diagram, we can compute the dimensions of the irreducible $\Z_k(\G)$-modules. The dimension of the centralizer algebra
$\Z_k(\G)$ is the multiplicity of the trivial $\G$-module (the white node)  at level $2k$. That dimension is also the sum of the squares of the irreducible $\G$-modules occurring in $\VV^{\ot k}$.  
We record  $\dimm \Z_k(\G)$ for $\G = \TT, \OO, \II$ in the next result.  The cyclic and dihedral
cases can be found in Sections 3 and 4.  

\begin{thm} \label{thm:ExceptionalDimensions} For $k \ge 1$,
\begin{enumerate}  
\item[\rm (a)] $\dimm \Z_k(\TT) = \displaystyle{\frac{4^k+8}{12}}$  {\rm (\cite{OEIS}  OEIS sequence  \href{http://oeis.org/A047849}{A047849})}. For 
$k \ge 2$,  the dimensions of the irreducible $\Z_k(\TT)$-modules, and thus also the multiplicities of the irreducible $\TT$-modules in $\VV^{\otimes k}$, are as shown in the following diagram:

\[ 
 \begin{array}{c}\includegraphics[scale=.8,page=18]{mckay-diagrams.pdf}\end{array}  \\
\]
\item[\rm (b)] $\dimm \Z_k(\OO) = \displaystyle{\frac{4^k +6\cdot 2^k+8}{24}}$  {\rm (\cite{OEIS} OEIS sequence  \href{http://oeis.org/A007581}{A007581})}.  For $k \ge 2$,  the dimensions of the irreducible 
$\Z_k(\OO)$-modules, and thus also the multiplicities of the irreducible $\OO$-modules in $\VV^{\otimes k}$, are as shown in the following diagram:

\[ 
 \begin{array}{c}\includegraphics[scale=.8,page=19]{mckay-diagrams.pdf}\end{array}  \\
\]

 \item[\rm (c)] $\dimm \Z_k(\II) = \displaystyle{\frac{4^k + 12L_{2k} + 20}{60}}$, where $L_n$ is the Lucas number defined by $L_0 = 2, L_1 = 1,$  and $L_{n+2} = L_{n+1} + L_{n}$.  For $k \ge 6$,  the dimensions of the irreducible $\Z_k(\II)$-modules, and thus also the multiplicities of the irreducible $\II$-modules in $\VV^{\otimes k}$, are as shown in the following diagram:
 \[ 
\hspace{-.5cm} \begin{array}{c}\includegraphics[scale=.8,page=20]{mckay-diagrams.pdf}\end{array}  \\
\]
 \end{enumerate}
\end{thm}

\begin{proof}
The proofs of the dimension formulas for the irreducible modules are by induction on $k$. The base cases are given in the Bratteli diagrams in the previous section.  The inductive step  amounts to verifying that the dimension formula for each node at level $k$ (for $k = 2n$ and $k = 2n+1$) equals the sum of the dimension formulas  for the nodes at level $k-1$ that are connected to the given one by an edge.  Each of these is a straightforward calculation.  The fact that $\dimm \Z_k(\G) = \dimm \Z_{2k}^{(0)}$ follows from \eqref{eq:evid}.
\end{proof}

\end{section}

\medskip

\noindent {\small Seattle, WA, 98103, USA}  \\
{\small(jmbarnes@google.com)} \\
{\small J. Barnes contributed
to these investigations while an undergraduate at Macalester College  and  is now  
at Google, Inc.}  \medskip

\noindent {\small Department of Mathematics,  
 University of Wisconsin-Madison, 
 Madison, WI 53706-1388, USA}  \\
 {\small (benkart@math.wisc.edu)}
 \medskip

\noindent 
{\small Department of Mathematics, Statistics, and Computer Science, Macalester College, St. Paul, MN 55105, USA} \\
{\small (halverson@macalester.edu)}

 \end{document}